\nonstopmode

\documentclass[10pt]{amsart}

\usepackage{amsmath}
\usepackage{amsfonts}
\usepackage{amssymb}
\usepackage{amscd}
\usepackage{amsthm}
\usepackage{amsbsy}
\usepackage{graphicx}
\usepackage{bm}
\usepackage[all]{xy}
\usepackage[pdftex,bookmarks=true,bookmarksnumbered=true
,pagebackref=true
]{hyperref}
\newcommand{\nmb}[2]{\ifx!#1{\ref{nmb:#2}}%
\else\if.#1{\label{nmb:#2}}%
\else\if0#1{\label{nmb:#2}}%
\else{{#2}}%
\fi\fi\fi}
\swapnumbers
\newtheorem{proposition}[subsection]{Proposition}
\newtheorem*{proposition*}{Proposition}
\newtheorem{theorem}[subsection]{Theorem}
\newtheorem*{theorem*}{Theorem}
\newtheorem{lemma}[subsection]{Lemma}
\newtheorem*{lemma*}{Lemma}
\newtheorem{corollary}[subsection]{Corollary}
\newtheorem*{corollary*}{Corollary}
\newtheorem*{conjecture*}{Conjecture}

\def\o{\operatorname{\circ}}
\def\X{\mathfrak X}
\def\al{\alpha}
\def\be{\beta}
\def\ga{\gamma}
\def\de{\delta}
\def\ep{\varepsilon}
\def\ze{\zeta}
\def\et{\eta}
\def\th{\theta}

\def\ka{\kappa}
\def\la{\lambda}
\def\rh{\rho}
\def\si{\sigma}
\def\ta{\tau}
\def\ph{\varphi}

\def\om{\omega}
\def\Ga{\Gamma}
\def\De{\Delta}

\def\Ph{\Phi}

\def\Om{\Omega}
\def\i{^{-1}}
\def\x{\times}
\def\p{\partial}
\let\on=\operatorname
\def\L{\mathcal L}
\def\Diff{\on{Diff}}
\def\Emb{\on{Emb}}
\def\Imm{\on{Imm}}
\def\Met{\on{Met}}
\def\vol{\on{vol}}
\def\Vol{\on{Vol}}
\def\Fl{\on{Fl}}
\def\supp{\on{supp}}
\def\Lip{\on{Lip}}
\newcommand{\East}[2]{-\raisebox{0.1pt}{$\mkern-16mu\frac{\;\;#1\;}{\;\;#2\;}\mkern-16mu$}\to}

\def\symbol{\par\noindent\hangindent=1cm \hangafter=1}
\allowdisplaybreaks
\title{Manifolds of mappings for continuum mechanics}
\author{Peter W.\ Michor}

\address{Fakult\"at f\"ur Mathematik,
Universit\"at Wien, Os\-kar-Mor\-gen\-stern-Platz 1, A-1090 Wien, Austria.}
\email{peter.michor@univie.ac.at.}
\thanks{Peter W.\ Michor: Universit\"at Wien, Os\-kar-Mor\-gen\-stern-Platz 1, A-1090 Wien, Austria.
}

\begin{document}


\thispagestyle{empty}

\begin{abstract}
\vskip 3mm\footnotesize{

\vskip 4.5mm
\noindent
After an introduction to convenient calculus in infinite dimensions, the foundational material for manifolds of mappings is presented. The central character is the smooth convenient manifold $C^{\infty}(M,N)$ of all smooth mappings from a  finite dimensional Whitney manifold germ $M$ into a smooth manifold $N$. A Whitney manifold germ is a smooth (in the interior) manifold with a very general boundary, but still admitting a continuous Whitney extension operator. This notion is developed here for the needs of geometric continuum mechanics. 

\vspace*{2mm} \noindent{\bf 2010 Mathematics Subject Classification:
}  58B20, 58D15, 35Q31

\vspace*{2mm} \noindent{\bf Keywords and Phrases: }}
Convenient calculus, Manifolds of mappings, Diffeomorphism groups, Shape spaces, weak Riemannian 
metrics. 

\end{abstract}

\maketitle

\tableofcontents

\newpage

\section{Introduction}\nmb0{1}

At the  birthplace of the notion of manifolds, in the 
Habilitations\-schrift \cite[end of section I]{Riemann1854}, Riemann mentioned infinite dimensional 
manifolds explicitly. 
The translation into English in \cite{RiemannClifford} reads as follows:
\begin{quote}
\emph{There are manifoldnesses in which the 
determination of position requires not a finite number, but either an endless series or a continuous 
manifoldness of determinations of quantity. Such manifoldnesses are, for example, the possible 
determinations of a function for a given region, the possible shapes of a solid figure, \&c.}
\end{quote} 
 
Reading this with a lot of good will one can interpret it as follows: When Riemann sketched 
the general notion of a manifold, he also had in mind the notion of an infinite dimensional manifold of 
mappings between manifolds. He then went on to describe the notion of 
Riemannian metric and to talk about curvature.

The dramatis personae of this foundational chapter are named in the following diagram:
$$
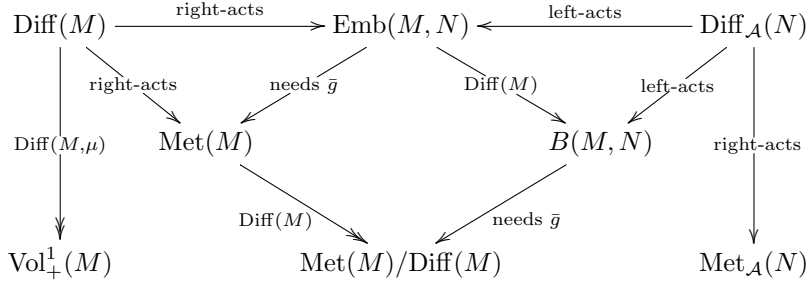
\xymatrix@C=.4cm@R=1cm{
{\Diff(M)}  \ar[rr]^-{\text{ right-acts }}  \ar[rd]|{\text{ right-acts }} 
\ar@{->>}[dd]|{\Diff(M,\mu)}
& &\Emb(M,N) 
\ar[dl]|{\text{needs  }\bar g} \ar[dr]|{\Diff(M)}
& & \Diff_{\mathcal A}(N) \ar[dd]|{\text{ right-acts} } \ar[ll]_-{\text{ left-acts }}
\ar[ld]|{\text{ left-acts }}
\\
& \Met(M) 
\ar@{->>}[rd]_{\Diff(M)} & & B(M,N) \ar[ld]^{\text{needs  }\bar g} &
\\
\Vol_{+}^1(M) && {\Met(M)}/{\Diff(M)} && \Met_{\mathcal A}(N) & 
}
$$
In this diagram:

\symbol
$M$ is a finite dimensional compact smooth manifold. 

\symbol
$N$ is a finite dimensional smooth manifolds without boundary, and $\bar g $ is one fixed background Riemannian metric on   $N$ which we always assume to be of bounded geometry; see Section \nmb!{5}.

\symbol
$\Met(N)=\Ga(S^2_+ T^*N)$ is the space of all Riemannian metrics on $N$.

\symbol
$\Diff(M)$ is the regular Fr\'echet Lie group of all diffeomorphisms on the compact manifold  $M$ with corners.

\symbol
$\Diff_{\mathcal A}(N),\; \mathcal A\in\{H^\infty,\mathcal S, c\} $ the regular Lie group of all smooth diffeomorphisms
 of decay $\mathcal A$ towards $\on{Id}_N$.

\symbol
$\Emb(M,N)$ is the infinite dimensional smooth manifold of all embeddings $M\to N$, 
which is the total space of a smooth principal fiber bundle $\Emb (M,N)\to B(M,N)=\Imm(M,N)/\Diff(M)$ 
with structure group $\Diff(M)$ and base manifold $B(M,N)$, the space of all smooth submanifolds of $N$ of type $M$.
It is possible to extend $\Emb(M,N)$ to the manifold of $\Imm(M,N)$ and $B(M,N)$ to the infinite dimensional orbifold $B_i(M,N)$.

\symbol
$\Vol^1_+(M)\subset\Ga(\on{vol}(M))$ is the space of all  positive smooth probability densities on the manifold $M$ with corners.

Since it will be of importance for geometric continuum mechanics, I will allow the source manifold $M$ to be quite general: $M$ can be a manifold with corners; see Section \nmb!{3}. This setting is worked out in detail in \cite{Michor80}.
Or $M$ can be a Whitney manifold germ, a notion originating in this paper; see Sec.\ \nmb!{4}. 

In this foundational chapter I will describe 
the theory of manifolds of mappings, of groups of diffeomorphisms, of manifolds of submanifolds (with corners), and of some striking results about weak Riemannian geometry on these spaces. See \cite{Bauer2014} for an overview 
article which is much more comprehensive for the aspect of shape spaces.

An explicit construction of manifolds of smooth mappings modeled on Fr\'echet spaces 
was described by  \cite{Eells58}. Differential calculus beyond the realm of Banach spaces has some 
inherent difficulties even in its definition; see section \nmb!{2}.
Smoothness of composition and inversion was first treated on the group of all smooth 
diffeomorphisms of a compact manifold in \cite{Leslie67}; however, there was a gap in the proof, 
which was first filled by \cite{Gutknecht77}. 
Manifolds of $C^k$-mappings and/or mappings of Sobolev classes were treated by 
\cite{Eliasson67}, \cite{Eells66}, Smale-Abraham \cite{Abraham62}, and \cite{Palais68}. Since these 
are modeled on Banach spaces, they allow solution methods for equations and have found a lot of 
applications. See in particular \cite{EbinMarsden70}. 

In preparation of this chapter I noticed, that the canonical chart construction for the manifold 
$C^{\infty}(M,N)$ even works if we allow $M$ to be a \emph{Whitney manifold germ}: These are modeled on open subsets  of closed subsets of $\mathbb R^m$ which (1) admit a continuous Whitney extension operator and (2) are the closure of their interior. See Section \nmb!{4} for a thorough discussion. Many results for them described below are preliminary, e.g., \nmb!{6.4}, \nmb!{7.2}.  I expect that they can be strengthened considerably, but I had not enough time to pursue them during the preparation of this chapter. 

I thank Reuven Segev and Marcelo Epstein for asking me for a contribution to this volume, and I thank 
them and Leonhard Frerick, Andreas Kriegl, Jochen Wengenroth, and Armin Rainer for helpful discussions.

\section{A short review of convenient calculus in infinite dimensions}\nmb0{2}

Traditional differential calculus works 
well for finite dimensional vector spaces and for Banach spaces. 
Beyond Banach spaces,
the main difficulty is that composition of 
linear mappings stops to be jointly continuous at the level of Banach 
spaces, for any compatible topology. Namely, if for a locally convex vector space $E$ and its dual 
$E'$ the evaluation mapping $\on{ev}:E\x E'\to \mathbb R$ is jointly continuous, then there are open 
neighborhoods of zero $U\subset E$ and $U'\subset E'$ with $\on{ev}(U\x U')\subset [-1,1]$. But then 
$U'$ is contained in the polar of the open set $U$, and thus is bounded. So $E'$ is normable, and a 
fortiori $E$ is normable. 

For locally convex spaces which are more general than Banach spaces,
we sketch here the convenient approach as explained in 
\cite{FrolicherKriegl88} and \cite{KrieglMichor97}.

The name \emph{convenient calculus} mimics the paper \cite{Steenrod67}
whose results (but not the name `convenient') was predated by
\cite{Brown61}, \cite{Brown63}, \cite{Brown64}.
They discussed compactly generated spaces as a cartesian closed category for algebraic topology.  
Historical remarks on only those developments of calculus beyond Banach spaces that led to convenient 
calculus are given in \cite[end of chapter I, p.\ 73ff]{KrieglMichor97}.

\subsection{The $c^\infty$-topology}\nmb.{2.1}
Let $E$ be a 
locally convex vector space. A curve $c:\mathbb R\to E$ is called 
\emph{smooth} or $C^\infty$ if all derivatives exist and are 
continuous. Let 
$C^\infty(\mathbb R,E)$ be the space of smooth curves. It can be 
shown that the set $C^\infty(\mathbb R,E)$ does not entirely depend on the locally convex 
topology of $E$, only on its associated bornology (system of bounded sets); see 
\cite[2.11]{KrieglMichor97}.
The final topologies  with respect to the following sets of mappings into $E$ (i.e., the finest topology on $E$ such that each map is continuous) coincide; see 
\cite[2.13]{KrieglMichor97}:
\begin{enumerate}
\item $C^\infty(\mathbb R,E)$.
\item The set of all Lipschitz curves 
(so that $\{\frac{c(t)-c(s)}{t-s}:t\neq s, |t|,|s|\le C\}$ 
is bounded in $E$, for each $C$). 
\item The set of injections $E_B\to E$ where $B$ runs through all bounded 
absolutely convex subsets in $E$, and where 
$E_B$ is the linear span of $B$ equipped with the Minkowski 
functional $\|x\|_B:= \inf\{\la>0:x\in\la B\}$.
\item The set of all Mackey-convergent sequences $x_n\to x$ 
(i.e., those for which there exists a sequence 
$0<\la_n\nearrow\infty$ with $\la_n(x_n-x)$ bounded).
\end{enumerate}
\emph{The resulting unique topology  is called the $c^\infty$-topology on $E$ and we write 
$c^\infty E$ for the resulting topological space.} 

In general 
(on the space $\mathcal{D}$ of test functions for example) it is finer 
than the given locally convex topology, it is not a vector space 
topology, since addition is no longer jointly 
continuous. Namely, even 
$c^\infty (\mathcal D\x \mathcal D)\ne c^\infty\mathcal D\x c^\infty\mathcal D$.

The finest among all locally convex topologies on $E$ 
which are coarser than $c^\infty E$ is the bornologification of the 
given locally convex topology. If $E$ is a Fr\'echet space, then 
$c^\infty E = E$. 

\subsection{Convenient vector spaces}\nmb.{2.2} 
A locally convex vector space 
$E$ is said to be a \emph{convenient 
vector space} if one of the following equivalent conditions holds
(called $c^\infty$-completeness); see \cite[2.14]{KrieglMichor97}:
\begin{enumerate}
\item For any $c\in C^\infty(\mathbb R,E)$ the (Riemann-) integral 
$\int_0^1c(t)dt$ exists in $E$.
\item Any Lipschitz curve in $E$ is locally Riemann integrable.
\item A curve $c:\mathbb R\to E$ is $C^\infty$ if and only if $\la\o c$ is $C^\infty$
for all $\la\in E^*$, where $E^*$ is the dual  
of all continuous linear functionals on $E$. 
\begin{itemize}
\item
Equivalently, for all 
$\la\in E'$, the dual  of all bounded linear functionals. 
\item
Equivalently, for all $\la\in \mathcal V$, where $\mathcal V$ is a subset of $E'$ which 
recognizes bounded subsets in $E$; see \cite[5.22]{KrieglMichor97}
\end{itemize}
{We call this \emph{scalarwise} $C^\infty$.}
\item Any Mackey-Cauchy sequence (i.e.,  $t_{nm}(x_n-x_m)\to 0$  
for some $t_{nm}\to \infty$ in $\mathbb R$) converges in $E$. 
This is visibly a mild completeness requirement.
\item If $B$ is bounded closed absolutely convex, then $E_B$ 
is a Banach space.
\item If $f:\mathbb R\to E$ is scalarwise $\Lip^k$, then $f$ is 
$\Lip^k$, for $k>1$.
\item If $f:\mathbb R\to E$ is scalarwise $C^\infty$ then $f$ is 
differentiable at 0.
\end{enumerate}

Here a mapping $f:\mathbb R\to E$ is called $\Lip^k$ if all 
derivatives up to order $k$ exist and are Lipschitz, locally on 
$\mathbb R$. That $f$ is scalarwise $C^\infty$ (resp., $\Lip^k$) 
means $\la\o f$ is $C^\infty$ (resp., $\Lip^k$) for all continuous 
(equiv., bounded) linear functionals on $E$.

\subsection{Smooth mappings}\nmb.{2.3}

Let $E$, and $F$ be convenient vector spaces, 
and let $U\subset E$ be $c^\infty$-open. 
A mapping $f:U\to F$ is called \emph{smooth} or 
$C^\infty$, if $f\o c\in C^\infty(\mathbb R,F)$ for all 
$c\in C^\infty(\mathbb R,U)$. See \cite[3.11]{KrieglMichor97}.

If $E$ is a Fr\'echet space, then this notion coincides with all other reasonable notions of 
$C^\infty$-mappings; see below. 
Beyond Fr\'echet spaces, as a rule, there are more smooth mappings in the 
convenient setting than in other settings, e.g., $C^\infty_c$. Moreover, any smooth mapping is continuous for the $c^\infty$-topologies, but in general not for the locally convex topologies: As shown in the beginning of Section \nmb!{2}, the evaluation mapping $\on{ev}:E\x E'\to \mathbb R$ is continuous only if $E$ is normable. On Fr\'echet spaces each smooth mapping is continuous; see the end of \nmb!{2.1}.

\subsection{Main properties of smooth calculus}\nmb.{2.4}
In the following all locally convex spaces are assumed to be convenient. 
\begin{enumerate}
\item For maps on Fr\'echet spaces the notion of smooth mapping from \nmb!{2.3} 
coincides with all other reasonable definitions. On 
$\mathbb R^2$ this is a nontrivial statement; see \cite{Boman67} or \cite[3.4]{KrieglMichor97}.
\item Multilinear mappings are smooth if and only if they are 
bounded; see \cite[5.5]{KrieglMichor97}.
\item If $E\supseteq U\East{f}{} F$ is smooth then the derivative 
$df:U\x E\to F$ is  
smooth, and also $df:U\to L(E,F)$ is smooth where $L(E,F)$ 
denotes the convenient space of all bounded linear mappings with the 
topology of uniform convergence on bounded subsets; see \cite[3.18]{KrieglMichor97}.
\item The chain rule holds; see \cite[3.18]{KrieglMichor97}.
\item The space $C^\infty(U,F)$ is again a convenient vector space 
where the structure is given by the  injection
$$
C^\infty(U,F) \East{C^\infty(c,\la)}{} \hspace{-.7cm}
\prod_{c\in C^\infty(\mathbb R,U), \la\in F^*} \hspace{-.6cm}
C^\infty(\mathbb R,\mathbb R),
\quad f\mapsto (\la\o f\o c)_{c,\la},
$$
and where $C^\infty(\mathbb R,\mathbb R)$ carries the topology of compact 
convergence in each derivative separately; see \cite[3.11 and 3.7]{KrieglMichor97}.
\item The exponential law holds; see \cite[3.12]{KrieglMichor97}.: For $c^\infty$-open $V\subset F$, 
$$
C^\infty(U,C^\infty(V,G)) \cong C^\infty(U\x V, G)
$$
is a linear diffeomorphism of convenient vector spaces.
\\
{\bf Note that this result (for $U=\mathbb R$) is the main assumption of variational calculus. Here it is a theorem.}
\item A linear mapping $f:E\to C^\infty(V,G)$ is smooth (by \thetag{2} equivalent to bounded) if 
and only if $E \East{f}{} C^\infty(V,G) \East{\on{ev}_v}{} G$ is smooth 
for each $v\in V$. 
(\emph{Smooth uniform boundedness theorem};
see \cite[theorem 5.26]{KrieglMichor97}.)
\item A mapping $f:U\to L(F,G)$ is smooth if and only if 
$$U \East{f}{} L(F,G) \East{\on{ev}_v}{} G$$ 
       is smooth for each $v\in F$, because then it is scalarwise smooth by the classical uniform 
       boundedness theorem. 
\item The following canonical mappings are smooth. This follows from the exponential law by simple 
       categorical reasoning; see \cite[3.13]{KrieglMichor97}.
\begin{align*}
&\operatorname{ev}: C^\infty(E,F)\x E\to F,\quad 
\operatorname{ev}(f,x) = f(x)\\
&\operatorname{ins}: E\to C^\infty(F,E\x F),\quad
\operatorname{ins}(x)(y) = (x,y)\\
&(\quad)^\wedge :C^\infty(E,C^\infty(F,G))\to C^\infty(E\x F,G)\\
&(\quad)^\vee :C^\infty(E\x F,G)\to C^\infty(E,C^\infty(F,G))\\
&\operatorname{comp}:C^\infty(F,G)\x C^\infty(E,F)\to C^\infty(E,G)\\
&C^\infty(\quad,\quad):C^\infty(F,F_1)\x C^\infty(E_1,E)\to 
\\&\qquad\qquad\qquad\qquad
\to C^\infty(C^\infty(E,F),C^\infty(E_1,F_1))\\
&\qquad (f,g)\mapsto(h\mapsto f\o h\o g)\\
&\prod:\prod C^\infty(E_i,F_i)\to C^\infty(\prod E_i,\prod F_i)
\end{align*}
\end{enumerate}

This ends our review of the standard results of convenient calculus. Just for the curious reader 
and to give a flavor of the arguments, 
we enclose a lemma that is used many times in the proofs of the results above.

\begin{lemma*} {\rm (Special curve lemma, \cite[2.8]{KrieglMichor97})}
Let $E$ be a  locally convex vector space.
Let $x_n$ be a 
sequence which converges fast
to $x$ in $E$; i.e., for each 
$k\in \mathbb N$ the sequence $n^k(x_n-x)$ is bounded. 
Then the \emph{infinite polygon} through the $x_n$ can be 
parameterized as a  
smooth curve $c:\mathbb R\to E$ such that $c(\frac1n)=x_n$ and $c(0)=x$.
\end{lemma*}

\subsection{Remark}\nmb.{2.5}
Convenient calculus (i.e., having properties 6 and 7) exists also for:
\begin{itemize}
\item  Real analytic mappings; see \cite{KrieglMichor90} or \cite[Chapter II]{KrieglMichor97}.
\item  Holomorphic mappings; see \cite{KrieglNel85} or \cite[Chapter II]{KrieglMichor97} (using the notion of 
\cite{Fantappie30,Fantappie33}).
\item  Many classes of Denjoy--Carleman ultradifferentiable functions, both of Beurling type and of 
Roumieu type, see \cite{KMRc,KMRq,KMR15c,KMR15d}.
\item With some adaptations, $\Lip^k$; see \cite{FrolicherKriegl88}. One has to adapt the exponential law \nmb!{2.4}.9 in the obvious way. 
\item With more adaptations, even $C^{k,\al}$ 
(the $k$-th derivative is H\"older-continuous with index $0<\al\le 1$); see 
\cite{FF89}, \cite{Faure91}. 
Namely, if $f$ is $C^{k,\al}$ and $g$ is $C^{k,\be}$, then $f\o g$ is $C^{k,\al\be}$.
\end{itemize}
Differentiability $C^n$ cannot be described by a convenient approach 
(i.e., allowing result like \nmb!{2.4}). Only such differentiability notions allow this, which can be described by boundedness conditions only.   

We shall treat $C^n$ mapping spaces using the following result.

\subsection{Recognizing smooth curves}
\nmb.{2.6}
The following result is very useful if one wants to apply convenient calculus to spaces which are 
not tied to its categorical origin, like the Schwartz spaces $\mathcal S$, $\mathcal D$, or Sobolev 
spaces; for its uses see \cite{MichorMumford13} and \cite{KMR14}.
In what follows $\si(E,\mathcal V)$ denotes the initial (also called  weak) topology on $E$ with respect to a set $\mathcal V\subset E'$.  

\begin{theorem*}
{\rm \cite[Theorem 4.1.19]{FrolicherKriegl88}}
Let $c:\mathbb R\to E$ be a curve in a convenient vector space $E$. Let 
$\mathcal{V}\subset E'$ be a subset of bounded linear functionals such that 
the bornology of $E$ has a basis of $\sigma(E,\mathcal{V})$-closed sets. 
Then the following are equivalent:
\begin{enumerate}
\item $c$ is smooth
\item There exist locally bounded curves $c^{k}:\mathbb R\to E$ such that
      $\la\o c$ is smooth $\mathbb R\to \mathbb R$ with $(\la\o c)^{(k)}=\la\o
      c^{k}$, for each $\la\in\mathcal V$. 
\end{enumerate}
If $E=F'$ is the dual of a convenient vector space $F$, then for any point separating subset
$\mathcal{V}\subseteq F$ the bornology of $E$ has a basis of 
$\si(E,\mathcal{V})$-closed subsets, by {\rm \cite[~4.1.22]{FrolicherKriegl88}}.
\end{theorem*}

This theorem is surprisingly strong: 
note that $\mathcal V$ does not need to recognize bounded sets. 
We shall use the theorem in situations where $\mathcal V$ is just the set of all point evaluations on suitable Sobolev spaces.

\subsection{Fr\"olicher spaces}\nmb.{2.7}
Following \cite[Section 23]{KrieglMichor97} we describe here the following simple concept.
A \emph{Fr\"olicher space} or a space with \emph{smooth structure}
is a triple $(X,\mathcal{C}_X, \mathcal{F}_X)$ 
consisting of a set $X$, a subset $\mathcal{C}_X$ of the set of all 
mappings $\mathbb R\to X$, and a subset $\mathcal{F}_X$ of the set of all 
functions $X\to \mathbb R$, with the following two properties:
\begin{enumerate}
\item[(\nmb:{1})] A function $f:X\to \mathbb R$ belongs to $\mathcal{F}_X$ if and only if 
       $f\o c \in C^\infty(\mathbb R,\mathbb R)$ for all $c\in \mathcal{C}_X$.
\item[(\nmb:{2})] A curve $c:\mathbb R\to X$ belongs to $\mathcal{C}_X$ if and only if 
       $f\o c \in C^\infty(\mathbb R,\mathbb R)$ for all $f\in \mathcal{F}_X$.
\end{enumerate}
Note that a set $X$ together with any subset $\mathcal{F}$ of the set of 
functions $X\to \mathbb R$ generates a unique Fr\"olicher space $(X,\mathcal{C}_X, 
\mathcal{F}_X)$, where we put in turn:
\begin{align*}
\mathcal{C}_X &:= \{c:\mathbb R\to X: f\o c\in C^\infty(\mathbb R,\mathbb R)
     \text{ for all }f\in \mathcal{F}\},\\
\mathcal{F}_X &:= \{f:X\to \mathbb R: f\o c\in C^\infty(\mathbb R,\mathbb R)
     \text{ for all }c\in \mathcal{C}_X\},
\end{align*}
so that $\mathcal{F}\subseteq \mathcal{F}_X$. The set $\mathcal{F}$ will be called a 
\emph{generating set of functions for the Fr\"olicher space}. 
Similarly a set $X$ together with any subset $\mathcal C$ of the set of curves $\mathbb R\to X$
generates a Fr\"olicher space; $\mathcal C$ is then called a \emph{generating set of curves} for this Fr\"olicher structure.
Note that \emph{a locally convex space $E$ is convenient if and only if it is a 
Fr\"olicher space with the structure whose space $\mathcal C_E$ of smooth curves is the one described in \nmb!{2.1}, or whose space $\mathcal F_E$ of smooth functions is described in 
\nmb!{2.3}}. This follows directly from \nmb!{2.2}. 

On each Fr\"olicher space we shall consider the final topology with respect to all  smooth curves $c:\mathbb R\to X$ in $\mathcal C_X$; i.e., the coarsest topology such that each such $c$ is continuous. This is in general finer that the initial topology with respect to all functions in $\mathcal F_X$. 

A mapping $\ph:X\to Y$ between two Fr\"olicher spaces is called 
\emph{smooth} if one of
the following three equivalent conditions hold
\begin{enumerate}
\item[(\nmb:{3})] For each $c\in \mathcal{C}_X$ the composite $\ph\o c$ is in 
       $\mathcal{C}_Y$. Note that here $\mathcal C_X$ can be replaced by a generating set 
       $\mathcal C$ of curves in $X$.
\item[(\nmb:{4})] For each $f\in \mathcal{F}_Y$ the composite $f\o \ph$ is in 
       $\mathcal{F}_X$. Note that $\mathcal{F}_Y$ can be replaced by a generating set of functions.
\item[(\nmb:{5})] For each $c\in \mathcal{C}_X$ and for each $f\in \mathcal{F}_Y$ the 
       composite $f\o\ph\o c$ is in $C^\infty(\mathbb R,\mathbb R)$.
\end{enumerate}

 The set of all smooth mappings from $X$ to $Y$ will be 
denoted by $C^\infty(X,Y)$. Then we have 
$C^\infty(\mathbb R,X)=\mathcal{C}_X$ and $C^\infty(X,\mathbb R)=\mathcal{F}_X$.

\emph{Fr\"olicher spaces and smooth mappings form a category which is complete, cocomplete, and cartesian closed}, by \cite[23.2]{KrieglMichor97}.

Note that there is the finer notion of diffeological spaces $X$ introduced by Souriau: These come equipped with a set of mappings from open subsets of $\mathbb R^n$'s into $X$ subject to some obvious properties concerning reparameterizations by $C^{\infty}$-mappings; see 
\cite{Iglesias-Zemmour13}. The obvious functor associating the generated Fr\"olicher space to a diffeological space is both left and right adjoint to the embedding of the category of Fr\"olicher spaces into the category of diffeological spaces.  A characterization of those diffeological spaces which are Fr\"olicher spaces is in \cite[Section 2.3]{Watts12}.

\section{Manifolds with corners}\nmb0{3}

In this section we collect some results which are essential for the extension of the  convenient setting for manifolds of mappings to a source manifold which has corners and which need not be compact.  

\subsection{Manifolds with corners}
\nmb.{3.1}
For more information we refer to \cite{DouadyHerault73}, \cite{Michor80}, \cite{Melrose96}.
Let $Q=Q^m=\mathbb R^m_{\ge 0}$ be the positive orthant or quadrant. By Whitney's extension theorem or Seeley's theorem (see also the discussion in \nmb!{4.1} -- \nmb!{4.3}),
the restriction $C^{\infty}(\mathbb R^m)\to C^{\infty}(Q)$ is a surjective continuous linear mapping which admits a continuous linear section (extension mapping); so $C^{\infty}(Q)$ is a direct summand in $C^{\infty}(\mathbb R^m)$. A point $x\in Q$ is called a \emph{corner of codimension (or index)} $q>0$ if $x$ lies in the intersection of $q$ distinct coordinate hyperplanes. Let $\p^q Q$ denote the set of all corners of codimension $q$.

A manifold with corners (recently also called a quadrantic manifold) $M$ 
is a smooth manifold modeled on open subsets of $Q^m$.
We assume that it is connected and second countable; then it is paracompact and each open cover admits a subordinated smooth partition of unity. 

We do not assume that $M$ is oriented, but for Moser's theorem we will eventually assume that $M$ is compact. 
Let $\p^q M$ denote the set of all corners of codimension $q$. Then 
$\p^q M$ is a submanifold without boundary of codimension $q$ in $M$; 
it has finitely many 
connected components if $M$ is compact. 
We shall consider $\p M$ as stratified into the connected components of all $\p^q M$ for $q > 0$. 
Abusing notation we will call $\p^q M$ the boundary stratum of codimension $q$; this will lead to no confusion. Note that $\p M$ itself is not a manifold with corners.
We shall denote by $j_{\p^q M}:\p^q M\to M$ the embedding of the boundary stratum of codimension $q$ into $M$, and by $j_{\p M}:\p M\to M$ the whole complex of embeddings of all strata.

Each diffeomorphism of $M$ restricts to a diffeomorphism of $\p M$ and to a diffeomorphism of each stratum $\p^q M$. The Lie 
algebra of $\Diff(M)$ consists of all vector fields $X$ on $M$ such that $X|_{\p^q M}$ is tangent to 
$\p^q M$. We shall denote this Lie algebra by $\X(M,\p M)$.

\begin{lemma}\nmb.{3.2}
Any manifold with corners $M$ is a submanifold with corners of an open manifold $\tilde M$ of the same dimension, and each smooth function on $M$ extends to a smooth function on $\tilde M$.
Each smooth vector bundle over $M$ extends to a smooth vector bundle over $\tilde M$.
Each immersion (embedding) of $M$ into a smooth manifold $N$ without boundary is the restriction of an immersion (embedding) of a (possibly smaller) $\tilde M\supset M$ into $N$.  
\end{lemma}

\begin{proof}
Choose a vector field $X$ on $M$ which is complete, and along $\p M$ is nowhere 0 and pointing into the interior. Then for $\ep>0$ we can replace $M$ by the   flow image 
$\Fl_{\ep}^X(M)$ which is contained in the interior $\tilde M = M\setminus \p M$.  
The extension properties follow from the Whitney extension theorem. An immersion extends, since its rank cannot fall locally.  An embedding $f$ extends since $\{(f(x),f(y)): (x,y)\in M\x M\setminus \text{Diag}_M\}$ has positive distance to the closed $\text{Diag}_N$ in $N\x N$, locally in $M$, and we can keep it that way; see \cite[5.3]{Michor80} for too many details.
\end{proof}

\subsection{Differential forms on manifolds with corners}\nmb.{3.3}
There are several differential complexes on a manifold with corners.  
If $M$ is not compact there are also the versions with compact support. 
\begin{itemize}
\item Differential forms that vanish near $\p M$. If $M$ is compact, this is the same as
the differential complex $\Om_c(M\setminus \p M)$ of differential forms with compact support 
in the open interior $M\setminus \p M$. 
\item $\Om(M,\p M) = \{\al\in \Om(M): j_{\p^q M}^*\al =0 \text{ for all } q\ge 1\}$, the complex of differential forms that pull back to 0 on each boundary stratum. 
\item $\Om(M)$, the complex of all differential forms. Its cohomology equals 
singular cohomology with real coefficients of $M$, since $\mathbb R\to \Om^0\to \Om^1\to \dots$
is a fine resolution of the constant sheaf on $M$; for that one needs existence of smooth partitions of unity and the Poincar\'e lemma which hold on manifolds with corners.
The Poincar\'e lemma can be proved as in \cite[9.10]{Michor08} in each quadrant.
\end{itemize}
If $M$ is an oriented manifold with corners of dimension $m$ and if $\mu\in \Om^m(M)$ is a nowhere vanishing form of top degree, then $\X(M)\ni X\mapsto i_X\mu\in \Om^{m-1}(M)$ is a linear isomorphism. 
Moreover, $X\in \X(M,\p M)$ (tangent to the boundary) if and only if $i_X\mu\in\Om^{m-1}(M,\p M)$.

\subsection{Towards the long exact sequence of the pair $(M,\p M)$}\nmb.{3.4} 
Let us consider the short exact sequence of differential graded algebras
$$
0\to \Om(M,\p M) \to \Om(M) \to \Om(M)/\Om(M,\p M)\to 0\,.
$$
The complex $\Om(M)/\Om(M,\p M)$ is a subcomplex of the product  of $\Om(N)$ for all connected components $N$ of all 
$\p^q M$. The quotient consists of forms which extend continuously over boundaries to $\p M$ with its induced topology in such a way that one can extend them to smooth forms on $M$; this is contained in the space of `stratified forms' as used in \cite{Valette15}. There Stokes' formula is proved for stratified forms.

\begin{proposition}[Stokes' theorem]\nmb.{3.5} 
For a connected oriented manifold $M$ with corners of dimension $\dim(M)=m$ and for any $\om\in\Om^{m-1}_c(M)$ we have
$$
\int_M d\om = \int_{\p^1M} j_{\p^1 M}^*\om\,.
$$
\end{proposition}

Note that $\p^1M$ may have several components. Some of these might be non-compact.

We shall deduce this result from Stokes' formula for a manifold with boundary by making precise the fact 
that $\p^{\ge 2} M$ has codimension 2 in $M$ and has codimension 1 with respect to $\p^1 M$. The proof also works for manifolds with more general boundary strata, like manifolds with cone-like singularities.
A lengthy full proof can be found in \cite{Conrad}.

\begin{proof}
We first choose a smooth decreasing function $f$ on $\mathbb R_{\ge 0}$ 
such that $f=1$ near 0 and $f(r)=0$ for $r\ge \ep$.
Then $\int_0^\infty f(r)dr <\ep$ and for $Q^m=\mathbb R^m_{\geq 0}$ with $m \geq 2$,
\begin{align*}
	\Big|\int_{Q^m} f'(|x|) \, dx\Big| &= C_m \Big|\int_0^\infty f'(r) r^{m-1} \, dr \Big| = 
	C_m \Big|\int_0^\infty f(r) (r^{m-1})' \, dr \Big|
	\\
	&= C_m  \int_0^\ep f(r) (r^{m-1})' \, dr \le C_m \ep^{m-1}\,, 
\end{align*}
where $C_m$ denotes the surface area of $S^{m-1} \cap Q^m$.
Given $\om\in\Om^{m-1}_c(M)$ we use the function $f$ on quadrant charts on $M$ to construct a function 
$g$ on $M$ that is 1 near $\p^{\ge2}M = \bigcup_{q\ge2}\p^q M$, has support close to $\p^{\ge 2}M$ and satisfies $\left| \int_M dg \wedge \om \right| < \ep$. Then $(1-g)\om$ is an $(m-1)$-form with compact support in the manifold 
with boundary $M\setminus \p^{\ge 2}M$, and Stokes' formula (cf.\ \cite[10.11]{Michor08}) now says
$$
\int_{M\setminus \p^{\ge 2}M} d((1-g)\om) = \int_{\p^1M} j_{\p^1 M}^*((1-g)\om)\,.
$$
But $\p^{\ge 2}M$ is a null set in $M$ and the quantities
$$
\Big| \int_M d((1-g)\om) - \int_M d\om\Big|  \quad\text{ and }\quad
\Big| \int_{\p^1M} j_{\p^1 M}^*((1-g)\om) - \int_{\p^1M} j_{\p^1 M}^*\om\Big| 
$$
are small if $\ep$ is small enough.
\end{proof}

 \subsection{Riemannian manifolds with bounded geometry}
 \nmb.{3.6}
 If $M$ is not necessarily compact without boundary we equip $M$ with a Riemannian 
 metric $g$ of bounded geometry which exists by \cite[Theorem 2']{Greene78}. This means that
 \\\indent $(I)$\qquad 
 The injectivity radius of $(M,g)$ is positive.
 \\\indent $(B_\infty)$\quad Each iterated covariant derivative of the curvature 
 \\\hphantom{A}\qquad\qquad is uniformly $g$-bounded: 
         $\|\nabla^i R\|_g<C_i$ for $i=0,1,2,\dots$.
 \\
 The following is a compilation of special cases of results collected in 
 \cite[chapter 1]{Eichhorn2007}. 
 
 \begin{proposition*} [\cite{Kordjukov91}, \cite{Eichhorn92}]
 If $(M,g)$ satisfies $(I)$ and $(B_\infty)$ then the following holds
 \begin{enumerate}
 	\item $(M,g)$ is complete.
 	\item There exists $\ep_0>0$ such that for each $\ep\in (0,\ep_0)$ there
 is a countable cover of 
         $M$ by geodesic balls $B_\ep(x_\al)$ such that the cover of $M$ by the
 balls $B_{2\ep}(x_\al)$ 
         is still uniformly locally finite.  
 	\item Moreover there exists a partition of unity $1= \sum_\al \rh_\al$
 on $M$ such that $\rh_\al\ge 0$, 
         $\rh_\al\in C^\infty_c(M)$, $\on{supp}(\rh_\al)\subset B_{2\ep}(x_\al)$,
 and 
 				$|D_u^\be \rh_\al|<C_\be$ where $u$ are normal
 (Riemannian exponential) coordinates in 
         $B_{2\ep}(x_\al)$. 
 	\item In each $B_{2\ep}(x_\al)$, in normal coordinates, we have 
 	      $|D_u^\be g_{ij}|<C'_\be$,  
 	      $|D_u^\be g^{ij}|<C''_\be$, and 
 	      $|D_u^\be \Ga^m_{ij}|<C'''_\be$,
 				where all constants are independent of $\al$.  
 \end{enumerate}
 \end{proposition*}

\subsection{Riemannian manifolds with bounded geometry allowing corners}\nmb.{3.7}
If $M$ has corners, we choose an open  manifold $\tilde M$ of the same dimension which contains $M$ as a submanifold with corners; see \nmb!{3.1}.  It is very desirable to prove that
then there exists a Riemannian metric $\tilde g$ on $\tilde M$ with bounded geometry such that each boundary component of each $\p^q M$ is totally geodesic. 

For a compact manifold with boundary (no corners of codimension $\ge 2$), existence of such a Riemannian metric  was proven in \cite[2.2.3]{G-BV14} in a more complicated context. A simple proof goes as follows: Choose a tubular neighborhood $U$ of $\p M$ in $\tilde M$ and use the  symmetry $\ph(u)=-u$ in the vector bundle structure  on $U$.  Given a metric $\tilde g$ on $\tilde M$, then $\p M$ is totally geodesic for the metric $\frac12 (\tilde g +\ph^*\tilde g)$ on $U$, since $\p M$ (the zero section) is the fixed point set of the isometry $\ph$. Now glue this metric to the $\tilde g$ using a partition of unity for the cover $U$ and $\tilde M\setminus V$ for a closed neighborhood $V$ of $\p M$ in $U$.

\emph{Existence of a  geodesic spray on a manifold with corners which is tangential to each boundary component $\p^q M$ was proved in \cite[2.8, see also 10.3]{Michor80}.}  A direct proof of this fact can be distilled from the proof of lemma \nmb!{5.9} below.
This is sufficient for constructing charts on the diffeomorphism group $\Diff(M)$ in \nmb!{6.1} below.

\section{Whitney manifold germs}\nmb0{4}

More general than manifolds with corners, Whitney manifold germs allow for quite singular boundaries but still controlled enough so that a continuous Whitney extension operator to an open neighborhood manifold exists. 

\subsection{Compact Whitney subsets}
\nmb.{4.1}
Let $\tilde M$ be an open smooth connected $m$-dimensional manifold. A closed connected subset $M\subset \tilde M$ is called a \emph{Whitney subset}, or $\tilde M\supset M$ is called a \emph{Whitney pair}, if
\begin{enumerate}
\item $M$ is the closure of its open interior in $\tilde M$, and:
\item There exists a continuous linear extension operator 
$$\mathcal E:\mathcal W(M) \to C^{\infty}(\tilde M,\mathbb R)$$ 
from the linear space $\mathcal W(M)$ of all Whitney jets of infinite order with its natural Fr\'echet topology (see below)
into the space $C^{\infty}(\tilde M,\mathbb R)$ of smooth functions on $\tilde M$ with the Fr\'echet topology of uniform convergence on compact subsets in all derivatives separately.  
\end{enumerate}
We speak of  a \emph{compact Whitney subset} or \emph{compact Whitney pair} if $M$ is compact.
In this case, in \thetag{2},  we may equivalently  require that $\mathcal E$ is linear continuous into the Fr\'echet space $C^{\infty}_L(\tilde M,\mathbb R)\subset C^{\infty}_c(\tilde M,\mathbb R)$ of smooth functions with support in a compact subset $L$ which contains $M$ in its interior, by using a suitable bump function. 

The property of being a Whitney pair is obviously invariant under diffeomorphisms (of $\tilde M$)
which act linearly and continuously both on $\mathcal W(M)$ and $C^{\infty}(\tilde M,\mathbb R)$ in a natural way.  

This property of being a Whiney pair is local in the following sense: If $\tilde M_i \supset M_i$ covers $\tilde M\supset M$, then  $\tilde M\supset M$ is a Whitney pair if and only if each $\tilde M_i \supset M_i$ is a Whitney pair, see Theorem \nmb!{4.4} below. 

\emph{More details:}
For  $\mathbb R^m\supset M$, by an extension operator 
$\mathcal E:\mathcal W(M) \to C^{\infty}(\tilde M,\mathbb R)$ we mean that 
$\p_\al \mathcal E(F)|_M = F^{(\al)}$ for each multi-index $\al\in\mathbb N_{\ge 0}^m$ and each Whitney jet $F\in \mathcal W(M)$. We recall the definition of a Whitney jet $F$. If $M\subset \mathbb R^m$ is compact, then  
\begin{align*}
F&=(F^{(\al)})_{\al\in\mathbb N_{\ge 0}^m} \in \prod_\al C^0(M)\qquad \text{ such that for }
\\
T_y^n(F)(x) &= \sum _{|\al|\le n} \frac{F^{(\al)}(y)}{\al!}(x-y)^\al \qquad \text{ the remainder seminorm }
\\
q_{n,\ep}(F) &:= \sup\Big\{\frac{|F^{(\al)}(x) - \p^\al T^n_y(F)(x)|}{|x-y|^{n-|\al|}}: 
\substack{|\al|\le n, x, y \in M \\ 0<|x-y|\le \ep}
 \Big\} = o(\ep);
\end{align*}
so $q_{n,\ep}(F)$ goes to 0 for $\ep\to 0$, for each $n$ separately. The $n$-th Whitney seminorm is then 
$$
\|F\|_n = \sup\{ |F^{(\al)}(x)|: x\in M, |\al|\le n\} + \sup\{q_{n,\ep}(F): \ep>0\}\,.
$$
For closed but non-compact $M$ one uses the projective limit over a countable compact exhaustion of $M$. This describes the natural Fr\'echet topology on the space of Whitney jets for closed subsets of $\mathbb R^m$. The extension to manifolds is obvious.  

Whitney proved in \cite{Whitney34a} that a linear extension operator always exists for a closed subset $M\subset \mathbb R^m$, but not always a continuous one, for example for $M$ a point. 
For a finite differentiability class $C^n$ there exists always a continuous extension operator.

\begin{proposition}\nmb.{4.2}
For a Whitney pair $\tilde M\supset M$, the space of $\mathcal W(M)$ of Whitney jets on $M$ is linearly isomorphic to the space 
$$C^{\infty}(M,\mathbb R):= \{f|_M: f\in C^{\infty}(\tilde M,\mathbb R)\}\,.$$
\end{proposition}

\begin{proof}
This follows from \cite[3.11]{Frerick07}, where the following is proved: If $f\in C^{\infty}(\mathbb R^m,\mathbb R)$ vanishes on a Whitney subset $M\subset \mathbb R^m$, 
then $\p^\al f|_M = 0$ for each multi-index $\al$.
Thus any continuous extension operator is injective.
\end{proof}
  
\subsection{Examples and counterexamples of Whitney pairs} 
\nmb.{4.3}
We collect here results about closed subsets of $\mathbb R^m$ which are or are  not Whitney subsets.

\smallskip\noindent\thetag{a}
If $M$ is a manifold with corners, then $\tilde M\supset M$ is a Whitney pair. This follows from   Mityagin \cite{Mityagin61} or Seeley \cite{Seeley64}. 

\smallskip\noindent\thetag{b}
If $M$ is closed in $\mathbb R^m$ with dense interior and with Lipschitz boundary, then $\mathbb R^m\supset M$ is a Whitney pair; by \cite[p 181]{Stein70}. In \cite[Theorem I]{Bierstone78} Bierstone proved that a closed subset $M\subset \mathbb R^n$ with dense interior is a Whitney pair, if it has H\"older $C^{0,\al}$-boundary for $0<\al\le 1$ which may be chosen on each $M\cap \{x: N\le |x|\le N+2\}$ separately.   A fortiori, each subanalytic subset in $\mathbb R^n$ gives a Whitney pair,  \cite[Theorem II]{Bierstone78}.

\smallskip\noindent\thetag{c} 
If $f\in C^{\infty}(\mathbb R_{\ge 0},\mathbb R)$ which is flat at 0 (all derivatives vanish at 0), and if $M$ is a closed subset containing 0 of $\{(x,y): x\ge0, |y|\le |f(x)|\}\subset \mathbb R^2$,
then $\mathbb R^2\supset M$ is not a Whitney pair; see \cite[Beispiel 2]{Tidten79}.

\smallskip\noindent\thetag{d} 
For $r\ge 1$, the set $\{x\in \mathbb R^m: 0\le x_1\le 1, x_2^2+\dots+x_m^2\le x_1^{2r}\}$ is called the parabolic cone of order $r$. Then the following result \cite[Satz 4.6]{Tidten79} holds:
\\
\emph{A closed subset $M\subset \mathbb R^m$ is a Whitney subset, if the following condition holds:
For each compact $K\subset \mathbb R^m$ there exists a parabolic cone $S$ and a family $\ph_i:S\to \phi_i(S)\subset M\subset \mathbb R^m$ of diffeomorphisms such that 
$K\cap M \subseteq \bigcup _{i} \overline{\ph_i(S)}$ and 
$\sup _i |\ph_i|_k < \infty$, $\sup_i |\ph_i\i|_k< \infty$ for each $k$ separately. }

\smallskip\noindent\thetag{e}
A characterization of closed subsets admitting continuous Whitney extension operators has been found by Frerick \cite[4.11]{Frerick07} in terms of local Markov inequalities, which however, is very difficult to check directly.
\\
\emph{Let $M\subset \mathbb R^m$ be closed. Then the following are equivalent:
\begin{enumerate}
\item[(e1)]
$M$ admits a continuous linear Whitney extension operator 
$$\mathcal E:\mathcal W(M)\to C^{\infty}(\mathbb R^m,\mathbb R)\,.$$
\item[(e2)]
For each compact $K\subset M$ and $\th\in(0,1)$ there is $r\ge 0$ and $\ep_0>0$ such that for all $k\in\mathbb N_{\ge 1}$ there is $C\ge 1$ such that
$$
|dp(x_0)| \le \frac C{\ep^r} \sup_{\substack{|y-x_0|\le \ep\\ y\in \mathbb R^m}} |p(y)|^\th 
\sup_{\substack{|x-x_0|\le\ep\\ x\in M}} |p(x)|^{1-\th}
$$
for all $p\in \mathbb C[x_1,\dots,x_m]$ of degree $\le k$, for all $x_0\in K$, and for all $\ep_0>\ep>0$. 
\item[(e3)]
For each compact $K\subset M$ there exists a compact $L$ in $\mathbb R^m$ containing $K$ in its interior, such that for all 
$\th\in(0,1)$ there is $r\ge1$ and $C\ge1$ such that 
$$
 \sup_{x\in K}|dp(x)| \le C \deg(p)^r \sup_{y\in L} |p(y)|^\th \sup_{z\in L\cap M} |p(z)|^{1-\th}
$$
for all $p\in \mathbb C[x_1,\dots,x_m]$.
\end{enumerate}
} 

\smallskip\noindent\thetag{f}
Characterization \thetag{e} has been generalized to a characterization of compact subsets of $\mathbb R^m$ which admit a continuous Whitney extension operator with linear (or even affine) loss of derivatives, in \cite{FJW11}. In the paper \cite{FJW13} a similar characterization is given for an extension operator without loss of derivative, and a sufficient geometric condition is 
formulated \cite[Corollary 2]{FJW13} which 
 even implies that there are closed sets with empty interior  admitting continuous Whitney extension operators, like the Sierpi\'nski triangle or Cantor subsets. 
Thus we cannot omit assumption \thetag{\nmb!{4.1}.1} that $M$ is the closure of its open interior in 
$\tilde M$ in our definition of  Whitney pairs. 

\smallskip\noindent\thetag{g} The following result by Frerick \cite[Theorem 3.15]{Frerick07} gives an easily verifiable sufficient condition:
\\
\emph{Let $K\subset \mathbb R^m$ be compact and assume that there exist $\ep_0>0$, $\rh>0$, $r\ge 1$ such that for all $z\in \p K$ and $0<\ep<\ep_0$ there is $x\in K$ with $B_{\rh\ep^r}(x) \subset K\cap B_\ep(z)$. Then $K$ admits a continuous linear Whitney extension operator $\mathcal W(F)\to C^{\infty}(\mathbb R^m,\mathbb R)$.}
\\
This implies \thetag{a}, \thetag{b}, and \thetag{d}. 

\begin{theorem}
\nmb.{4.4}
Let $\tilde M$ be an open manifold and let $M\subset \tilde M$ be a compact subset that is the closure of its open interior. 
$M\subset \tilde M$ is a Whitney pair if and only if  for every smooth atlas
$(\tilde M\supset U_\al, u_\al:U_\al\to u_\al(U_\al)\subset \mathbb R^m)_{\al\in A}$ of the open manifold 
$\tilde M$, 
each 
$u_\al(M\cap U_\al)\subset u_\al(U_\al)$ is a Whitney pair. 

Consequently, for a Whitney pair $M\subset \tilde M$  and  $U\subset \tilde M$ open,  
$M\cap U\subset \tilde M\cap U$ is also a Whitney pair.
\end{theorem}

\begin{proof} 
\thetag{\nmb:{1}}
We consider a locally finite countable smooth atlas
$(\tilde M\supset U_\al, u_\al:U_\al\to u_\al(U_\al)\subset \mathbb R^m)_{\al\in \mathbb N}$ of  
$\tilde M$ such that each 
$u_\al(U_\al)\supset u_\al(M\cap U_\al)$ is a Whitney pair.  

We use a smooth `partition of unity' $f_\al\in C^{\infty}_c(U_\al,\mathbb R_{\ge 0})$ on $\tilde M$ with 
$\sum_\al f_\al^2 =1$. The following mappings induce linear embeddings onto closed direct summands of the Fr\'echet spaces:
\begin{gather*}
\xymatrix{
C^{\infty}(\tilde M,\mathbb R) \ar@/^2mm/[rrr]^{f\mapsto (f_\al.f)_\al} & & &
\prod_\al C^{\infty}(U_\al,\mathbb R) \ar@/^2mm/[lll]^{\sum_\al f_\al . g_\al\gets (g_\al)_\al}
}
\\
\xymatrix{
\mathcal W(M) \ar@/^/[rrr] && & \prod_\al \mathcal W(U_\al\cap M) \ar@/^/[lll]
}
\end{gather*}
If each $u_\al(U_\al)\supset u_\al(U_\al\cap M)$
is a Whitney pair, then so is $U_\al\supset U_\al\cap M$, via the isomorphisms induced by $u_\al$, and 
$$\xymatrix{
\mathcal W(M)  \ar@/^2mm/[rrr]_{f\mapsto (f_\al.f)_\al} & && \prod_\al \mathcal W(U_\al\cap M) \ar[d]^{\prod_\al \mathcal E_\al} \\
C^{\infty}(\tilde M,\mathbb R)  & & &
\prod_\al C^{\infty}(U_\al,\mathbb R) \ar@/^2mm/[lll]_{\sum_\al f_\al . g_\al\gets (g_\al)_\al}
}
$$
is a continuous Whitney extension operator, so that $\tilde M\supset M$ is a Whitney pair.
This proves the easy direction of the theorem. 

\bigskip
The following argument for the converse direction is inspired by \cite{FW19p}. 
 
\smallskip\noindent
\thetag{\nmb:{2}} (See \cite[Def.~3.1]{Frerick07}, \cite[Section 29-31]{MeiseVogt97}) A Fr\'echet space $E$ is said to have \emph{property} (DN) if for one (equivalently, any) increasing system   $(\|\cdot\|_n)_{n\in \mathbb N}$ of seminorms generating the topology the following holds:
\begin{itemize}
\item 
There exists a continuous seminorm  $\|\;\|$ on $E$ (called a \emph{dominating norm}, hence the name (DN)) such that for all (equivalently, one) $0<\th<1$ and all $m\in \mathbb N$ there exist $k\in \mathbb N$ and $C>0$ with 
$$\|\;\|_m\le C\|\;\|_k^\th \cdot\|\;\|^{1-\th}\,.$$
\end{itemize}
The property (DN) is an isomorphism invariant and is inherited by closed linear subspaces. 
The Fr\'echet space $\mathfrak s$ of rapidly decreasing sequences has property (DN).

\smallskip\noindent
\thetag{\nmb:{3}} (\cite[Satz 2.6]{Tidten79}, see also \cite[Theorem 3.3]{Frerick07})
\emph{A closed subset $M$ in $\mathbb R^m$ 
admits a continuous linear extension operator 
$\mathcal W(M)\to C^{\infty}(\mathbb R^m,\mathbb R)$ 
if and only if for each $x\in\p M$ there exists a compact neighborhood $K$ of $x$ in $\mathbb R^m$ such that
$$
\mathcal W_K(M) := \big\{ f\in \mathcal W(M): \on{supp}(f^{(\al)})\subset K \text{ for all }\al\in\mathbb N_{\ge 0}^m\big\}.
$$
has property (DN).}

\medskip
We assume  from now on that $\tilde M\supset M$ is a Whitney pair.

\smallskip\noindent
\thetag{\nmb:{4}} Given a compact set $K\subset \tilde M$, let $L\subset \tilde M$ be a compact smooth manifold with smooth boundary which contains $K$ in its interior. Let $\tilde L$ be the double of $L$, i.e., $L$ smoothly glued to another copy of $L$
along the boundary; $\tilde L$ is a compact smooth manifold containing $L$ as a submanifold with boundary. 

Then $C^{\infty}(\tilde L,\mathbb R)$ is isomorphic to the space $\mathfrak s$ of rapidly decreasing sequences: This is due to  \cite{Vogt83}. In fact, using a Riemannian metric $g$ on $\tilde L$, the expansion in an orthonormal basis of eigenvectors of  $1+\De^g$ of a function $h\in L^2$ has coefficients in  $\mathfrak s$ if and only if $h\in C^\infty(\tilde L,\mathbb R)$, because $1+\De^g: H^{k+2}(\tilde L)\to H^k(\tilde L)$ is an isomorphism for Sobolev spaces $H^k$ with $k\ge 0$, and since the eigenvalues $\mu_n$ of $\De^g$  satisfy $\mu_n \sim  C_{\tilde L}\cdot n^{2/\dim(\tilde L)}$ for $n\to \infty$, by Weyl's asymptotic formula. 
Thus $C^{\infty}(\tilde L,\mathbb R)$ has  property (DN).  

Moreover, $C^{\infty}_L(\tilde M,\mathbb R)= \{f\in C^{\infty}(\tilde M,\mathbb R): \on{supp}(f)\subset L\}$ is a closed linear subspace of $C^{\infty}(\tilde L,\mathbb R)$, by extending each function by 0. Thus also $C^{\infty}_L(\tilde M,\mathbb R)$ has property  (DN).  

We choose now a function $g\in C^{\infty}_L(\tilde M,\mathbb R_{\ge0})$ with $g|_K =1$ and consider: 
$$
\xymatrix@R=1.3em{
\mathcal W_K(M) \ar@{^{(}->}[d] \ar@{-->}[rr]^{\mathcal E_K} && C^{\infty}_L(\tilde M,\mathbb R) 
\\
\mathcal W(M) \ar[rr]^{\mathcal E_M} && C^{\infty}(\tilde M,\mathbb R) \ar[u]_{f\mapsto g.f} 
}
$$
The resulting composition $\mathcal E_K$ is a continuous linear embedding onto a closed linear subspace of the space  $C^{\infty}_L(\tilde M,\mathbb R)$ which has (DN).  Thus we proved: 

\smallskip\noindent
\thetag{\nmb:{5}} \textbf{Claim.}
\emph{If $\tilde M\supset M$ is a Whitney pair and $K$ is compact in $\tilde M$, the Fr\'echet space 
$\mathcal W_K(M)$ has property} (DN). 

\smallskip\noindent
\thetag{\nmb:{6}}
We consider now a smooth chart $\tilde M\supset U \xrightarrow{u} u(U) = \mathbb R^m$.
For $x\in \p u(M)$ let $K$ be a compact neighborhood of $x$ in $\mathbb R^m$. 
The chart $u$ induces a linear isomorphism 
$$ 
\mathcal W_K(u(M\cap U))\xrightarrow{u^*} \mathcal W_{u^{-1}(K)}(U\cap M) \cong \mathcal  W_{u^{-1}(K)}(M),
$$ 
where the right-hand side mapping is given by extending each $f^{(\al)}$ by 0.
By  claim \thetag{\nmb|{5}} the  Fr\'echet space $W_{u^{-1}(K)}(M)$ has property (DN); consequently  also the isomorphic space $\mathcal W_K(u(M\cap U))$ has property (DN). By \thetag{\nmb|{3}} we conclude that 
$\mathbb R^m=u(U)\supset u(M\cap U)$ is a Whitney pair. 

\smallskip\noindent
\thetag{\nmb:{7}}
If we are given a general chart  $\tilde M\supset U \xrightarrow{u} u(U) \subset \mathbb R^m$, we cover $U$ by a locally finite atlas 
$(U\supset U_\al, u_\al:U_\al\to u_\al(U_\al)= \mathbb R^m)_{\al\in \mathbb N}$. By \thetag{\nmb|{6}}  each  
$\mathbb R^m = u_\al(U_\al)\supset u_\al(M\cap U_\al)$ is a Whitney pair, and by the argument in \thetag{\nmb|{1}} the pair $U\supset M\cap U$ is a Whitney pair, and thus the diffeomorphic 
$u(U)\supset u(U\cap M)$ is also a Whitney pair.
\end{proof}

\subsection{Our use of Whitney pairs}
\nmb.{4.5}
We consider an equivalence class of Whitney pairs $\tilde M_i\supset M_i$ for $i=0,1$ where 
$\tilde M_0\supset M_0$ is equivalent to $\tilde M_1\supset M_1$ if there exist an open submanifolds 
$\tilde M_i \supset\hat M_i\supset M_i$ 
and a diffeomorphism $\ph: \hat M_0\to \hat M_1$ with $\ph(M_0)=M_1$.
By a \emph{germ of a Whitney manifold} we mean an equivalence class of Whitney pairs as above. 
Given a Whitney pair $\tilde M\supset M$ and its corresponding germ, we 
may keep $M$ fixed and equip it with all open connected neighborhoods of $M$ in $\tilde M$; each neighborhood is then a representative of this germ; called an \emph{open neighborhood manifold} of $M$. In the following we shall speak of a \emph{Whitney manifold germ} $M$ and understand that it comes with open manifold neighborhoods $\tilde M$. If we want to stress a particular neighborhood we will write $\tilde M\supset M$.

The \emph{boundary} $\p M$ of a Whitney manifold germ is the topological boundary of 
$M$ in $\tilde M$. It can be a quite general set as seen from the examples \nmb!{4.3} and the discussion in \nmb!{4.9}. But infinitely flat cusps cannot appear. 

\subsection{Other approaches}
\nmb.{4.6}
We remark that there are other settings, like the concept of a \emph{manifold  with rough boundary}; see \cite{RobertsSchmeding18} and literature cited there. The main idea there is to start with closed subsets $C\subset \mathbb R^m$ with dense interior, to use  the space of functions which are $C^n$ in the interior of $C$ such that all partial derivatives extend continuously to $C$. Then one looks for sufficient conditions (in particular for $n=\infty$) on $C$ such  that there exists a continuous Whitney extension operator on the space of these functions, and builds manifolds from that. The condition in  \cite{RobertsSchmeding18} are in the spirit of \thetag{\nmb!{4.3}.d}. By extending these functions and restricting their jets to $C$ we see that the manifolds with rough boundary are Whitney manifold germs.  

Another possibility is to consider closed subsets $C\subset \mathbb R^m$ with dense interior such there exists a continuous linear extension operator on the space 
$C^\infty(C)= \{f|_C: f\in C^{\infty}(\mathbb R^m)\}$ with the quotient locally convex topology.  These are exactly the  Whitney manifold pairs $\mathbb R^m\supset M$, by \nmb!{4.2}.
In this setting, for $C^n$ with $n<\infty$ there exist continuous extension operators 
$C^n_b(C)\to C^n_b(\mathbb R^m)$ (where the subscript $b$ means bounded for all derivatives separately)
for arbitrary subsets $C\subset \mathbb R^m$; see \cite{Feffermann07}.

We believe that our use of Whitney manifold germs is quite general, simple, and avoids many technicalities. But it is aimed at the case $C^{\infty}$; for $C^k$ or $W^{k,p}$ other approaches, like the one in \cite{RobertsSchmeding18}, might be more appropriate.

\subsection{Tangent vectors and vector fields on Whitney manifold germs}
\nmb.{4.7}
In line with the more general convention for vector bundles in \nmb!{4.8} below, we define the tangent bundle $TM$ of a Whitney manifold germ $M$ as  the restriction $TM=T\tilde M|_M$. For $x\in \p M$, a tangent vector $X_x\in T_xM$ is said to be \emph{interior pointing} if there exist a curve $c:[0,1)\to M$ which is smooth into $\tilde M$ with $c'(0)=X_x$. And $X_x\in T_xM$ is called \emph{tangent to the boundary} if there exists a curve $c:(-1,1)\to \p M$ which is smooth into $\tilde M$ with $c'(0)=X_x$. The \emph{space of vector fields on $M$}  is given as
$$\X(M)=\{X|_M: X\in \X(\tilde M)\}.$$ 
Using a continuous linear extension operator, $\X(M)$ is isomorphic to a locally convex direct summand in $\X(\tilde M)$. If $M$ is a compact Whitney manifold germ, $\X(M)$ is a direct summand even in 
 $\X_L(\tilde M)=\{X\in \X(\tilde M): \supp(X)\subseteq L\}$ 
 where $L\subset \tilde M$ is a compact set containing $M$ in its interior. 
 We define the \emph{space of vector fields on $M$ which are tangent to the boundary} as
\begin{multline*}
 \X_\p(M) = \big\{ X|_M: X\in \X(\tilde M), 
 x\in \p M \implies \Fl^X_t(x)\in \p M 
 \\ 
 \text{ for all } t \text{ for which } \Fl^X_t(x) \text{ exists in } \tilde M
 \big\}\,,
\end{multline*}
where $\Fl^X_t$ denotes the flow mapping of the vector field $X$ up to time $t$ which is locally defined on $\tilde M$.
Obviously, for $X\in \X_\p(M)$ and $x\in \p M$ the tangent vector $X(x)$ is tangent to the boundary in the sense defined above.  I have no proof that the converse is true:

\noindent
\textbf{Question.} \emph{Suppose that $X\in \X(\tilde M)$ has the property that for each $x\in \p M$ the tangent vector $X(x)$ is tangent to the boundary. Is it true that then $X|_M\in \X_\p(M)$?}

A related question for which I have no answer is:

\noindent
\textbf{Question.} \emph{
For each $x\in \p M$ and tangent vector $X_x\in T_xM$ which is tangent to the boundary, is there a smooth vector field $X\in \X_{c,\p}(M)$ with $X(x)=X_x$? 
}

\begin{lemma*} For a Whitney manifold germ $M$,
the space $\X_\p(M)$ of vector field tangent to the boundary 
is a closed linear sub Lie algebra of $\X(M)$. The space $\X_{c,\p}(M)$ of vector fields with compact support tangent to the boundary is a closed linear sub Lie algebra of $\X_c(M)$. 
\end{lemma*}

\begin{proof} By definition,
for $X\in \X(\tilde M)$ the restriction  $X|_M$ is in $\X_\p(M)$ if and only if 
$x\in \p M$ implies that $\Fl^X_t(x)\in \p M$ for all $t$ for which $\Fl^X_t(x)$ exists in $\tilde M$. 
These  conditions describe a set of continuous equations, since $(X,t,x)\mapsto \Fl^X_t(x)$ is smooth; 
see the proof of \nmb!{6.1} for a simple argument.  Thus $X\in \X(\tilde M)$ is closed. 

The formulas (see, e.g.,  \cite[p 56,58]{Nelson69})
\begin{gather*}
\lim_{n\to \infty}( Fl^X_{t/n}\o \Fl^Y_{t/n})^n(x) = \Fl^{X+Y}_t(x)
\\
\lim_{n\to \infty} \Big(\Fl^Y_{-(t/n)^{1/2}} \o \Fl^X_{-(t/n)^{1/2}} \o \Fl^Y_{(t/n)^{1/2}}\o \Fl^X_{(t/n)^{1/2}}\Big)^n (x) = \Fl^{[X,Y]}_t(x)
\end{gather*}
shows that $\X_\p (M)$ is a Lie subalgebra.
\end{proof}

\subsection*{The smooth partial  stratifications of the boundary of a Whitney manifold germ}
Given a Whitney manifold germ $\tilde M\supset M$ of dimension $m$, for each $x\in \p M$ we denote by $\mathcal L^\infty(x)$ the family consisting of each  maximal connected open smooth submanifold $L$ of $\tilde M$ which contains $x$ and is contained in $\p M$. 
Note that for $L\in \mathcal L^\infty(x)$ and $y\in L$ we have $L\in \mathcal L^\infty(y)$. 
$\{T_xL: L\in \mathcal L^\infty(x)\}$ is a set of linear subspaces of $T_x\tilde M$. The collective of these for all $x\in \p M$ is something like a `field of quivers of vector spaces' over $\p M$. It might be the key to eventually construct charts for  the regular Fr\"olicher Lie group $\Diff(M)$ treated in \nmb!{6.3} below, and for constructing charts for the Fr\"olicher space $\Emb(M,N)$ in \nmb!{7.2} below.

\subsection{Mappings, bundles, and sections}
\nmb.{4.8}
Let $M$ be Whitney manifold germ and let $N$ be a manifold without boundary.
By a smooth mapping $f:M\to N$ we mean $f=\tilde f|_M$ for a smooth mapping $\tilde f:\tilde M\to N$ for an open manifold neighborhood $\tilde M\supset M$. Whitney jet on $M$ naively make sense only if they take values in a vector space or, more generally, in a vector bundle. 
One could develop the notion of Whitney jets of infinite order with values in a manifold
as sections of $J^\infty(M,N)\to M$ with Whitney conditions of each order. We do not know whether this has been written down formally. But we can circumvent this easily by considering a closed embedding $i:N\to \mathbb R^p$ and a tubular neighborhood $p:U\to i(N)$; i.e., $U$ is an open neighborhood and is (diffeomorphic to) the total space of a smooth vector bundle which projection $p$. 

Then we can consider a Whitney jet on $M$ with values in $\mathbb R^p$ (in other words, a $p$-tuple of Whitney jets) such that the 0 order part lies in $i(N)$. Using a continuous Whitney extension operator, we can extend the Whitney jet to a smooth mapping $\tilde f:\tilde M\to \mathbb R^p$. Then consider the open set $\tilde f\i(U)\subset \tilde M$ instead of $\tilde M$, and replace $\tilde f$ by $p\o\tilde f$. So we just extended the given Whitney jet to a smooth mapping $\tilde M\to N$, and also showed, that the space of Whitney jets is isomorphic to the space 
$$
C^{\infty}(M,N)=\{f|_M: f\in C^{\infty}(\tilde M, N), \tilde M\supset M\}.
$$  
Note that the neighborhood $\tilde M$ can be chosen independently of the mapping $f$, but dependent on $N$. This describes a nonlinear extension operator $C^\infty(M,N)\to C^\infty(\tilde M, N)$; we shall see in section \nmb!{5} that this extension operator is continuous and even smooth in the manifold structures.

For finite $n$ we shall need the space $C^{\infty,n}(\mathbb R\x M,\mathbb R^p)$ of restrictions to $M$ of mappings $\mathbb R\x \tilde M\ni (t,x)\mapsto f(t,x) \in\mathbb R^p$ which are  $C^{\infty}$ in $t$ and $C^n$ in $x$. If $\tilde M$ is open in $\mathbb R^m$ we mean by this that any partial derivative $ \p_t^k\p_x^\al f$  of any order $k\in \mathbb N_{\ge 0}$ in $t$ and of  order 
$|\al|\le n$ in $x$ exists and is continuous on $\mathbb R\x \tilde M$. This carries over to an open manifold $\tilde M$, and finally, using again a tubular neighborhood $p:U\to i(N)$ as above, to the space $C^{\infty,n}(\mathbb R\x M, N)$, for any open manifold $N$. For a treatment of $C^{m,n}$-maps leading to an exponential law see \cite{AlzaareerSchmeding15}; since $C^n$ is not accessible to a convenient approach, a more traditional calculus has to be used there.  

By a (vector or fiber) bundle $E\to M$ over a germ of a Whitney manifold $M$ we mean the restriction to $M$ of a (vector or fiber) bundle $\tilde E\to \tilde M$, i.e., of a (vector or fiber) bundle over an open manifold neighborhood.
By a smooth section of $E\to M$ we mean the restriction of a smooth section of $\tilde E\to \tilde M$ for a neighborhood $\tilde M$. 
Using classifying smooth mappings into a suitable Grassmannian for vector bundles over $M$ and using the discussion above one could talk about Whitney jets of vector bundles and extend them to a manifold neighborhood of $M$. 

We shall use the following spaces of sections of a vector bundle $E\to M$ over a Whitney manifold germ $M$. This is more general than  \cite[Section 30]{KrieglMichor97}, since we allow Whitney manifold germs as base.  
\begin{itemize}
\item 
$\Ga(E)=\Ga(M\gets E)$, the space of smooth sections, i.e., restrictions of smooth sections of $\tilde E \to \tilde M$ for a fixed neighborhood $\tilde M$, with the Fr\'echet space topology of compact convergence on the isomorphic space of Whitney jets of sections. 
\item
$\Ga_c(E)$, the space of smooth sections with compact support, with the inductive limit (LF)-topology.
\item
$\Ga_{C^n}(E)$, the space of $C^n$-section, with the Fr\'echet space topology of compact convergence on the space of Whitney $n$-jets. If $M$ is compact and $n$ finite, $\Ga_{C^n}(E)$ is a Banach space. 
\item
$\Ga_{H^s}(E)$, the space of Sobolev $H^s$-sections, for $0\le s \in \mathbb R$. Here $M$ should be a compact Whitney manifold germ. The measure on $M$ is the restriction of the volume density with respect to a Riemannian metric on $\tilde M$. One also needs a fiber metric on $E$. The space 
$\Ga_{H^k}(E)$ is independent of all choices, but the inner product depends on the choices.
One way to define  $\Ga_{H^k}(E)$ is to use a finite atlas which trivializes $\tilde E|_L$ over a compact manifold with smooth boundary $L$ which is a neighborhood of $M$ in $\tilde M$ and a partition of unity, and then use the Fourier transform description of the Sobolev space.  
For a careful description see 
\cite{BBHM18}.
For $0 \le k<s-\dim(M)/2$ we have $\Ga_{H^s}(E)\subset \Ga_{C^k}(E)$ continuously.
\item
More generally, for $0\ge s\in \mathbb R$ and $1<p<\infty$ we also consider $\Ga_{W^{s,p}}(E)$, the space of $W^{s,p}$-sections: For integral $s$, all (covariant) derivatives up to order $s$ are in $L^p$. 
For $0\le k<s-\dim(M)/p$ we have $\Ga_{H^s}(E)\subset \Ga_{C^k}(E)$ continuously.
\end{itemize}

\subsection{Is Stokes' theorem valid for Whitney manifold germs?}
\nmb.{4.9}
This seems far from obvious. Here is an example, due to  \cite{FW19p}:

By  the first answer to the MathOverflow question \cite{26000} there is a  set $K$ in $[0,1]\subset \mathbb R$ which is the closure of its open interior such that the boundary is a Cantor set with positive Lebesgue measure. Moreover, $\mathbb R\supset K$ is a Whitney pair by \cite{Tidten01}, or by the local Markov inequalities \cite[Proposition 4.8]{Frerick07}, or by \cite{FJW11}. To make this connected, consider 
$K_2:=(K\x [0,2])\cup ([0,1]\x [1,2])$ in $\mathbb R^2$. Then $\mathbb R^2\supset K_2$ is again a Whitney pair, but $\p K_2$ has positive 2-dimensional Lebesgue measure.    

As an aside we remark that  Cantor-like closed sets in $\mathbb R$ might or might not admit continuous extension operators;  see
\cite[Beispiel 1]{Tidten79},  \cite{Tidten01}, and the final result in \cite{AGK02}, where a complete characterisation is given in terms of the logarithmic dimension of the Cantor-like set. 

\begin{theorem}\nmb.{4.10} \textrm{\cite[Theorem 4]{KlassenMichor19}} Let $M$ be a connected 
compact oriented Whitney manifold germ.  
Let $\om_0,\om\in\Om^{m}(M)$ be two 
volume forms (both $>0$) with $\int_M\om_0=\int_M\om$.
Suppose that there is a diffeomorphism $f:M\to M$ such that $f^*\om|_U = \om_0|_U$ for an open neighborhood of $\p M$ in $M$. 

Then there exists a diffeomorphism $\tilde f:M\to M$ with $\tilde f^*\om = \om_0$ such that $\tilde f$ equals $f$ on an open neighborhood of $\p M$.
\end{theorem} 

This relative Moser theorem for Whitney manifold germs is modeled on the standard proof of Moser's theorem in 
\cite[Theorem 31.13]{Michor08}. The proof of \cite[Theorem 4]{KlassenMichor19} is for manifolds with corners, but it works without change for Whitney manifold germs.

\section{Manifolds of mappings}\nmb0{5}

In this section we demonstrate how convenient calculus allows for very short and transparent 
proofs of the core results in the theory of manifolds of smooth mappings. We follow \cite{KrieglMichor97} but we allow the source manifold 
to be a  Whitney manifold germ. 
In \cite{Michor80} $M$ was allowed to have corners. 
We will treat manifolds of smooth mappings, and of $C^n$-mappings, and we will also mention  the case of Sobolev mappings. 


\begin{lemma}[Smooth curves into spaces of sections of vector bundles]
\nmb.{5.1}
Let $p\colon E\to M$ be a vector bundle over a compact smooth manifold $M$, possibly with corners. 
\\
\thetag{1}  
Then the space $C^{\infty}(\mathbb R,\Ga(E))$ of all smooth curves in $\Ga(E)$ consists of all  
$c\in C^{\infty}(\mathbb R\x M,E)$ with $p\o c = \on{pr}_2:\mathbb R\x M\to M$.
\\
\thetag{2}
Then the space $C^{\infty}(\mathbb R,\Ga_{C^n}(E))$ of all smooth curves in $\Ga_{C^n}(E)$ consists of all  $c\in C^{\infty,n}(\mathbb R\x M,E)$ (see \nmb!{4.8})
with $p\o c = \on{pr}_2:\mathbb R\x M\to M$.
\\
\thetag{3}
If $M$ is a compact manifold or a compact Whitney manifold germ, then for each $1<p<\infty$ and $s\in(\dim(M)/p,\infty)$ the space $C^\infty(\mathbb R,\Ga_{W^{s,p}}(E))$ of smooth curves in 
$\Ga_{W^{s,p}}(E)$ consists of all continuous mappings 
$c:\mathbb R\x M \to E$ with $p\o c = \on{pr}_2:\mathbb R\x M\to M$ such that the following two conditions hold:
\begin{itemize}
 \item For each $x\in M$ the curve $t\mapsto c(t,x)\in E_x$ is smooth; \\
        let $(\p^k_t c)(t,x) = \p_t^k(c(\;,x))(t)$. 
 \item For each $k\in \mathbb N_{\ge0}$, the curve $\p_t^kc$ has values in $\Ga_{W^{s,p}}(E)$ 
       so that $\p_t^kc :\mathbb R\to \Ga_{W^{s,p}}(E)$, and 
       $t \mapsto \|\p_t^k c(t,\cdot)\|_{\Ga_{W^{s,p}}(E)}$ is bounded, 
       locally in $t$.
\end{itemize}
\thetag{4}
If $M$ is an open manifold, then  the space $C^{\infty}(\mathbb R,\Ga_c(E))$ of all smooth curves in the space  $\Ga_c(E)$ of smooth sections with compact support consists of all  
$c\in C^{\infty}(\mathbb R\x M,E)$ with $p\o c = \on{pr}_2:\mathbb R\x M\to M$ such that 
\begin{itemize}
\item for each compact interval $[a,b]\subset \mathbb R$ there is a compact subset $K\subset M$
such that $c(t,x)=0$ for $(t,x)\in [a,b]\x (M\setminus K)$.
\end{itemize}
Likewise for the space $C^{\infty}(\mathbb R,\Ga_{C^n,c}(E))$ of smooth curves in the space of $C^n$-sections with compact support. 
\\
\thetag{5} Let $p\colon E\to M$ be a vector bundle over a compact Whitney manifold germ.
Then  the space $C^\infty(\mathbb R,\Ga(E))$ of smooth curves in $\Ga(E)$ consists of all smooth mappings $c:\mathbb R\x \tilde M \to \tilde E$ with $p\o c = \on{pr}_2:\mathbb R\x \tilde M\to \tilde M$ for some open neighborhood manifold $\tilde M$ and extended vector bundle $\tilde E$. We may even assume that  there is a  compact submanifold with smooth boundary $L\subset \tilde M$ containing $M$ in its interior such that $c(t,x)=0$ for $(t,x)\in \mathbb R\x (\tilde M\setminus L)$.  
Using the last statement of \nmb!{4.1}, this is equivalent to the space of all smooth mappings 
$c:\mathbb R\x M \to E\subset \tilde E$ with $p\o c = \on{pr}_2:\mathbb R\x M\to M$.
\\
\thetag{6}
 Let $p\colon E\to M$ be a vector bundle over a non-compact Whitney manifold germ $M\subset \tilde M$,  then  the space $C^{\infty}(\mathbb R,\Ga_c(E))$ of all smooth curves in the space  
 $$\Ga_c(E)=\{s|_M: s\in \Ga_c(\tilde M \gets \tilde E)\}$$ 
of smooth sections with compact support (see \nmb!{4.8}) consists of all smooth mappings 
$c:\mathbb R\x \tilde M \to \tilde E$ with $p\o c = \on{pr}_2:\mathbb R\x \tilde M\to \tilde M$ such that 
\begin{itemize}
\item for each compact interval $[a,b]\subset \mathbb R$ there is a compact subset $K\subset \tilde M$
such that $c(t,x)=0$ for $(t,x)\in [a,b]\x (M\setminus K)$.
\end{itemize}

\end{lemma}

\begin{proof} \thetag{1} This follows from the exponential law \nmb!{2.4}.6 after trivializing the bundle.
\\
\thetag{2} We trivialize the bundle, assume that $M$ is open in $\mathbb R^m$, and then prove this directly. In \cite[3.1 and 3.2]{KrieglMichor97} one finds a very explicit proof of the case $n=\infty$, which one can restrict to our case here. 
\\
\thetag{3}
To see this we first choose a second vector bundle $F\to M$ such that $E\oplus_M F$ is a trivial bundle, i.e., isomorphic to $M\x \mathbb R^n$ for some $n\in\mathbb N$. 
Then $\Ga_{W^{s,p}}(E)$ is a direct summand in $W^{s,p}(M,\mathbb R^n)$, so that we may assume without loss that $E$ is a trivial bundle, and then, that it is 1-dimensional. 
So we have to identify $C^\infty(\mathbb R,W^{s,p}(M,\mathbb R))$. 
But in this situation we can just apply Theorem \nmb!{2.6} for the set 
$\mathcal V\subset W^{s,p}(M,\mathbb R)'$ consisting of all point evaluations $\on{ev}_x:H^s(M,\mathbb R)\to \mathbb R$ and use that $W^{s,p}(M,\mathbb R)$ is a reflexive Banach space. 
\\
\thetag{4} This is like \thetag{1} or \thetag{2} where we have to assure that the curve $c$ takes values in the space of sections with compact support which translates to the condition.
\\
\thetag{5} and \thetag{6} follow from \thetag{4} after extending to $\tilde E\to \tilde M$. 
\end{proof}

\begin{lemma}
\nmb.{5.2} 
Let $E_1,E_2$ be vector bundles over smooth manifold or a Whitney manifold germ $M$, 
let $U\subset E_1$ be an open neighborhood of the image of a smooth section, 
let $F:U\to E_2$ be a fiber preserving smooth mapping.
Then the following statements hold:
\begin{enumerate}
\item 
[\thetag{1}] If $M$ is compact,
the set $\Ga(U):=\{h\in \Ga(E_1): h(M)\subset U\}$ is open in 
$\Ga(E_1)$, and the mapping
$F_*:\Ga(U) \to \Ga(E_2)$ given by $h\mapsto  F\o h$
is smooth. Likewise for spaces $\Ga_c(E_i)$, if $M$ is not compact.
\item[\thetag{2}]
If $M$ is compact, for $n\in \mathbb N_{\ge 0}$  
the set 
$$
\Ga_{C^n}(U):=\{h\in \Ga_{C^n}(E_1): h(M)\subset U\}
$$ 
is open in 
$\Ga_{C^n}(E_1)$, and the mapping
$F_*:\Ga_{C^n}(U) \to \Ga_{C^n}(E_2)$ given by $h\mapsto  F\o h$
is smooth.
\item[\thetag{3}] 
If $M$ is compact and  $s > \dim(M)/p$, 
the set 
$$
\Ga_{W^{s,p}}(U):=\{h\in \Ga_{W^{s,p}}(E_1): h(M)\subset U\}
$$ 
is open in 
$\Ga_{W^{s,p}}(E_1)$, and the mapping
$F_*:\Ga_{W^{s,p}}(U) \to \Ga_{W^{s,p}}(E_2)$ given by $h\mapsto  F\o h$,
is smooth.
\end{enumerate}
\end{lemma}

If the restriction of $F$ to each fiber of $E_1$ is real analytic, then $F_*$ is real analytic; but in this paper we concentrate on $C^{\infty}$ only.
This lemma is a variant of  the so-called Omega-lemma; e.g., see \cite{Michor80}. Note how simple the proof is using convenient calculus.

\begin{proof} Without loss suppose that $U=E_1$.

\noindent\thetag{1} and \thetag{2} follow easily since $F_*$ maps smooth curves to smooth curves; see their description in  (\nmb!{5.1}.1) and (\nmb!{5.1}.2).

\noindent\thetag{3} 
Let $c:\mathbb R\ni t\mapsto c(t,\;)\in \Ga_{W^{s,p}}(E_1)$ be a smooth curve. 
As $s>\dim(M)/2$, it holds for each $x\in M$ that the mapping $\mathbb R\ni t \mapsto F_x(c(t,x))\in (E_2)_x$ is smooth. 
By the Fa\`a di Bruno formula (see \cite{FaadiBruno1855} for the 1-dimensional version, preceded in \cite{Arbogast1800} by 55 years), we have for each $p\in\mathbb N_{>0}$, $t \in \mathbb R$, and $x \in M$ that
\begin{multline*}
\p_t^p F_x (c(t,x)) = 
\\ =
\sum_{j\in\mathbb N_{>0}} \sum_{\substack{\al\in \mathbb N_{>0}^j\\ \al_1+\dots+\al_j =p}}
\frac{1}{j!}d^j (F_x) (c(t,x))\Big(
\frac{\p_t^{(\al_1)}c(t,x)}{\al_1!},\dots,
\frac{\p_t^{(\al_j)}c(t,x)}{\al_j!}\Big)\,.
\end{multline*}
For each $x\in M$ and $\al_x\in (E_2)_x^*$ the mapping $s\mapsto \langle s(x),\al_x\rangle$ is a continuous linear functional on the Hilbert space $\Ga_{W^{s,p}}(E_2)$. 
The set $\mathcal V_2$ of all of these functionals separates points and therefore satisfies the condition of Theorem~\nmb!{2.6}. We also have for each $p\in\mathbb N_{>0}$, $t \in \mathbb R$, and $x \in M$ that
\begin{align*}
\p_t^p\langle F_x (c(t,x)),\al_x\rangle &= \langle\p_t^p F_x (c(t,x)),\al_x\rangle
= \langle\p_t^p F_x (c(t,x)),\al_x\rangle.
\end{align*}
Using the explicit expressions for $\p_t^p F_x (c(t,x))$ from above we may apply Lemma~(\nmb!{5.1}.3) to conclude that $t\mapsto F(c(t,\;))$ is a smooth curve $\mathbb R\to \Ga_{H^s}(E_1)$. 
Thus, $F_*$ is a smooth mapping.
\end{proof}

\subsection{The manifold structure on $C^\infty(M,N)$ and $C^k(M,N)$}
\nmb.{5.3}
Let $M$ be a compact or open finite dimensional smooth manifold or even a compact Whitney manifold germ, and let $N$ be a smooth manifold. We use an 
auxiliary Riemannian metric $\bar g$ on $N$ and its exponential mapping $\exp^{\bar g}$; some of its  properties  are summarized in the following diagram:
$$\xymatrix@R=1.3em
{
& 0_N \ar@{_{(}->}[d] \ar@(l,dl)[dl]_-{\text{zero section\quad }} 
& & N \ar@{^{(}->}[d] \ar@(r,dr)[rd]^-{\text{ diagonal}} & 
\\
TN &  V^N \ar@{_{(}->}[l]^{\text{ open  }} \ar[rr]^{ (\pi_N,\exp^{\bar g}) }_{\cong} & & V^{N\x N} 
\ar@{^{(}->}[r]_{\text{ open  }} & N\x N
}$$
Without loss we may assume that $V^{N\x N}$ is symmetric: 
$$(y_1,y_2)\in V^{N\x N} \iff (y_2,y_1)\in V^{N\x N}.$$

\noindent $\bullet$
If $M$ is compact, then $C^\infty(M,N)$, the space of smooth mappings $M\to N$, has the following manifold structure.
A chart, centered at $f\in C^\infty(M,N)$, is:
\begin{align*}
C^\infty(M,N)\supset U_f &=\{g: (f,g)(M)\subset V^{N\x N}\} \East{u_f}{} \tilde U_f 
\subset \Ga(M\gets f^*TN)
\\
u_f(g) = (\pi_N,&\exp^{\bar g})\i \o (f,g),\quad u_f(g)(x) = (\exp^{\bar g}_{f(x)})\i(g(x)) 
\\
(u_f)\i(s) &= \exp^{\bar g}_f\o s, \qquad (u_f)\i(s)(x) = \exp^{\bar g}_{f(x)}(s(x))
\end{align*}
Note that $\tilde U_f$ is open in $\Ga(M\gets f^*TN)$ if $M$ is compact. 

\noindent $\bullet$
If $M$ is open, then the compact $C^\infty$-topology on $\Ga(f^*TN)$ is not suitable since $\tilde U_f$ is in general not open. We have to control the behavior of sections near infinity on $M$. 
One solution is to use the space  $\Ga_c(f^*TN)$ of sections with compact support as modeling spaces and to adapt the topology on $C^\infty(M,N)$ accordingly. This has been worked out in \cite{Michor80} and \cite{KrieglMichor97}.  

\noindent $\bullet$
If $M$ is compact Whitney manifold germ with neighborhood manifold $\tilde M\supset M$ we use the 
Fr\'echet space $\Ga(M\gets f^*TN)= \{s|_M: s\in \Ga_L(\tilde M\gets \tilde f^* TN)\}$ where $L\subset \tilde M$ is a compact set containing $M$ in its interior and $\tilde f:\tilde M\to N$ is an extension of $f$ to a suitable manifold neighborhood of $M$.  
Via an extension operator the Fr\'echet space $\Ga(M\gets f^*TN)$ is a direct summand in the Fr\'echet space $ \Ga_L(\tilde M\gets \tilde f^* TN)$ of smooth sections with support in $L$. 

\noindent $\bullet$
Likewise, for a non-compact Whitney manifold germ we use the convenient (LF)-space 
$$\Ga_c(M\gets f^*TN) =  \{s|_M: s\in \Ga_c(\tilde M\gets \tilde f^* TN)\}$$ 
of sections with compact support.

\noindent $\bullet$
On the space $C^{k}(M,N,)$ for $k\in\mathbb N_{\ge 0}$ we use only charts 
as described above with the center $f\in C^\infty(M,N)$, namely
$$
C^k(M,N)\supset U_f =\{g: (f,g)(M)\subset V^{N\x N}\} \East{u_f}{} \tilde U_f 
\subset \Ga_{C^k}(M\gets f^*TN)\,.
$$
We claim that these charts cover $C^k(M,N)$: Since $C^\infty(M,N)$ is dense in $C^k(M,N)$ in the Whitney $C^k$-topology, for any $g\in C^k(M,N)$ there exists $f\in C^\infty(M,N,)\cap U_g$. But then $g\in U_f$ since $V^{N\x N}$ is symmetric. This is true for compact $M$. For a compact Whitney manifold germ 
we can apply this argument in a compact neighborhood $L$ of $M$ in $\tilde M$,  replacing $\tilde M$ by the interior of $L$ after the fact.

\noindent $\bullet$
On the space $W^{s,p}(M,N)$ for $\dim(M)/p<s\in\mathbb R$ we use only charts 
as described above with the center $f\in C^\infty(M,N)$, namely:
\begin{multline*}
W^{s,p}(M,N)\supset U_f =\{g: (f,g)(M)\subset V^{N\x N}\} \East{u_f}{} 
\\
\East{u_f}{}\tilde U_f 
\subset \Ga_{W^{s,p}}(M\gets f^*TN)\,.
\end{multline*}
These charts cover $W^{s,p}(M,N)$, by the following argument: Since $C^\infty(M,N)$ is dense in $W^{s,p}(M,N)$ and since $W^{s,p}(M,N)\subset C^{k}(M,N)$ via a continuous injection for $0\le k<s-\dim(M)/p$, a suitable $C^0-\sup$-norm neighborhood of $g\in W^{s,p}(M,N)$  contains a smooth $f\in C^{\infty}(M,N)$, thus $f\in U_g$ and by symmetry of $V^{N\x N}$ we have $g\in U_f$.
This is true for compact $M$. For a compact Whitney manifold germ 
we can apply this argument in a compact neighborhood which is a manifold with smooth boundary $L$ of $M$ in $\tilde M$ and apply the argument there. 

\emph{In each case, we equip $C^{\infty}(M,N)$ or $C^k(M,N)$ or $W^{s,p}(M,N)$ with the initial topology with respect to all chart mappings described above:} The coarsest topology, so that all chart mappings $u_f$ are continuous.

\emph{For non-compact $M$ the direct limit $\Ga_c(f^*TN)= \varinjlim_{L} \Ga_L(f^*TN)$ over a compact exhaustion $L$ of $M$ in the category of locally convex vector spaces is strictly coarser that the direct limit in the category of Hausdorff topological spaces.} It is more \emph{convenient} to use the latter topology which is called $c^\infty$ topology; compare with \nmb!{2.1}.

\begin{lemma}\nmb.{5.4} 
\thetag{1}
If $M$ is a compact smooth manifold or is a compact Whitney manifold germ,
$$
C^\infty(\mathbb R,\Ga(M\gets f^*TN)) = \Ga(\mathbb R\x M\gets \on{pr_2}^* f^*TN)\,.
$$  
For smooth $f\in C^\infty(M,N)$, 
$$
C^\infty(\mathbb R,\Ga_{C^n}(M\gets f^*TN)) = \Ga_{C^{\infty,n}}(\mathbb R\x M\gets \on{pr_2}^* f^*TN)\,.
$$
\\
\thetag{2}
If $M$ is a non-compact smooth manifold of Whitney manifold germ, the sections on the right 
hand-side have to satisfy the corresponding conditions of lemma (\nmb!{5.1}.4).
\end{lemma}

\noindent
For a compact Whitney manifold germ $M$ the space 
$ \Ga(\mathbb R\x M\gets \on{pr_2}^* f^*TN)$ is a direct summand in the space 
$ \Ga_{\mathbb R\x L}(\mathbb R\x \tilde M\gets \on{pr_2}^* f^*TN)$ of sections with support in $\mathbb R\x L$ for a fixed compact set $L\subset \tilde M$ containing $M$ in its interior. Likewise 
$ \Ga_{C^{\infty,n}}(\mathbb R\x M\gets \on{pr_2}^* f^*TN)$ is a direct summand in the space 
$ \Ga_{C^{\infty,n}, \mathbb R\x L}(\mathbb R\x \tilde M\gets \on{pr_2}^* f^*TN)$ of $C^{\infty,n}$-sections. 
One could introduce similar notation for $C^\infty(\mathbb R,\Ga_{W^{s,p}}(M\gets f^*TN))$.

\begin{proof}
This follows from lemma \nmb!{5.1}. 
\end{proof}

\begin{lemma}\nmb.{5.5} Let $M$ be a smooth manifold or Whitney manifold germ, compact or not, and let $N$ be  a manifold. Then 
the chart changes for charts centered on smooth mappings are smooth ($C^\infty$) on the space 
$C^\infty(M,N)$, also on $C^k(M,N)$ for $k\in \mathbb N_{\ge 0}$, and on  $W^{s,p}(M,N)$ for $1<p<\infty$ and $s>\dim(M)/p$:
$$
 \tilde U_{f_1}\ni s  \mapsto (u_{f_2,f_1})_*(s) := 
(\exp^{\bar g}_{f_2})\i\o \exp^{\bar g}_{f_1}\o s\in \tilde U_{f_2}\,.
$$
\end{lemma}

\begin{proof} This follows from Lemma~\nmb!{5.2}, since
any chart change is just compositions from the left by a smooth fiber respecting locally defined diffeomorphism. 
\end{proof}

\begin{lemma}\nmb.{5.6} \thetag{1} If $M$ is a  compact manifold or a compact Whitney manifold germ, then 
$$C^\infty(\mathbb R,C^\infty(M,N))\cong C^\infty(\mathbb R\x M,N)\,.$$ 
\\
\thetag{2}
If $M$ is not compact, $C^\infty(\mathbb R,C^\infty(M,N))$ consists of all smooth 
$c:\mathbb R\x M\to N$ such that
\begin{itemize}
\item for each compact interval $[a,b]\subset \mathbb R$ there is a compact subset $K\subset M$
such that $c(t,x)$ is constant in $t\in [a,b]$ for each $x\in M\setminus K$.
\end{itemize}
\end{lemma}

\begin{proof}
By  lemma \nmb!{5.4}.
\end{proof}

\begin{lemma}\nmb.{5.7} 
Composition $(f,g)\mapsto g\o f$  is 
smooth as a mapping
\begin{gather*}
C^\infty(P,M)\x C^\infty(M,N)\to C^\infty(P,N)
\\
C^k(P,M)\x C^\infty(M,N)\to C^k(P,N)
\\
W^{s,p}(P,M)\x C^\infty(M,N)\to W^{s,p}(P,N)
\end{gather*}
for $P$ a manifold or a Whitney manifold germ, compact or not, and for $M$ and $N$  manifolds.
\end{lemma}

For more general $M$ the description becomes more complicated. See the special case of the diffeomorphism group of a Whitney manifold germ $M$ in \nmb!{6.3} below.

\begin{proof}
Since it maps smooth curves to smooth curves.
\end{proof}

\begin{corollary}\nmb.{5.8} For $M$ a manifold or a Whitney manifold germ and a manifold $N$,
the tangent bundle of the manifold $C^{\infty}(M,N)$ of mappings is given by 
\begin{align*}
TC^\infty(M,N)&= C^\infty(M,TN) \East{C^\infty(M,\pi_N)=(\pi_N)_*}{} C^\infty(M,N)\,,
\\
TC^k(M,N)&= C^k(M,TN) \East{C^k(M,\pi_N)=(\pi_N)_*}{} C^k(M,N)\,,
\\
TW^{s,p}(M,N)&= W^{s,p}(M,TN) \East{W^{s,p}(M,\pi_N)=(\pi_N)_*}{} W^{s,p}(M,N)\,.
\end{align*}
\end{corollary}

\begin{proof}
This follows from the chart structure and the fact that sections of $f^*TN\to M$ correspond to mappings $s:M\to TN$ with $\pi_N\o s =f$.
\end{proof}

\subsection{Sprays respecting fibers of submersions}
\nmb.{5.9}
 Sprays are versions of Christoffel symbols and lead to exponential mappings. They are easier to adapt to fibered manifolds than Riemannian metrics.
Recall that a spray $S$ on a manifold  $N$ without boundary  is a smooth mapping $S:TN\to T^2N$ 
with the following properties:
\begin{itemize}
\item $\pi_{TN}\o S= \on{Id}_{TN}$; $S$ is a vector field.
\item $T(\pi_N)\o S= \on{Id}_{TN}$; $S$ is a `differential 
       equation of second order'.
\item Let $m^N_t:TN\to TN$ and $m^{TN}_t:T^2N\to T^2N$ be 
       the scalar multiplications. Then $S\o m^N_t = 
       T(m^N_t).m^{TN}_t.S$.
\end{itemize}
Locally, in charts of $TN$ and $T^2N$ induced by a chart of $N$, a spray looks like $S(x,v)=(x,v;v;\Ga(x,v))$ where $\Ga$ is quadratic in $v$. 
For a spray $S\in \X(TN)$ on a manifold $N$,
we let $\exp(X):=\pi_N(\Fl^S_1(X))$, then the mapping 
$\exp:TN\supset V\to N$
is smooth, defined on an 
open neighborhood $V$ of the zero section in $TN$, which is called 
the \emph{exponential mapping} of the spray $S$. 
Since $T_{0_x}(\exp|_{T_xN})=\on{Id}_{T_xN}$ (via 
       $T_{0_x}(T_xN)=T_xN$), by the inverse function theorem 
       $\exp_x:=\exp|_{T_xN}$ is a diffeomorphism near $0_x$ in 
       $TN$ onto an open neighborhood of $x$ in $N$. 
Moreover the mapping 
$(\pi_N,\exp):TN\supset\tilde V\to N\x N$
       is a diffeomorphism from an open neighborhood $\tilde V$ of the 
       zero section in $TN$ onto an open neighborhood of the 
       diagonal in $N\x N$. 

\begin{lemma*}
Let $q:N\to M$ be a smooth surjective submersion between connected manifolds without boundary.
Then there exists a spray $S$ on $N$ which is tangential to the fibers of $q$, i.e.,  
$S(T(q\i(x)))\subset T^2(q\i(x))$ for each $x\in M$.
\end{lemma*}

This is a simplified version of \cite[10.9]{Michor80}.

\begin{proof} In suitable charts on $N$ and $M$ the submersion $q$ looks like a linear projection 
$(y_1,y_2)\mapsto y_1$. The local expression $T(\text{chart})\to T^2(\text{chart})$ of  a spray is   
\begin{multline*}
S\big((y_1,y_2),(v_1,v_2)\big) =
\\
= \big((y_1,y_2),(v_1,v_2);(v_1,v_2),(\Ga^1(y_1,y_2;v_1,v_2),\Ga^2(y_1,y_2,v_1,v_2))\big)
\end{multline*}
where $\Ga^i(y_1,y_2,v_1,v_2)$ is quadratic in $(v_1,v_2)$. The spray is tangential to the fibers of $q$ if and only if  $\Ga^1(y_1,y_2,0,v_2)=0$. This clearly exists locally (e.g., choose $\Ga^1=0$).
Now we use a partition of unity $(\ph_\al)$ subordinated to a cover $N= \bigcup_{\al}U_\al$ with such charts and glue local sprays with the induced partition of unity $(\ph_\al\o \pi_N)$ subordinated to the cover $TN=\bigcup_{\al}TU\al$ for the vector bundle $\pi_{TN}:T^2N\to TN$. Locally this looks like 
(where $y=(y_1,y_2)$ etc.)
\begin{align*}
\Big(\sum_\al (\ph_\al\o \pi_N).S_\al\Big)(y,v) &=
\Big(y,v; \sum_\al \ph_\al(y)v, 
 \sum_\al \ph_\al(y)\big(\Ga^1_\al(y,v),(\Ga^2_\al(y,v)\big)
\Big)
\\&
= 
\Big(y,v; v, 
 \big(\sum_\al \ph_\al(y)\Ga^1_\al(y,v),\sum_\al \ph_\al(y)\Ga^2_\al(y,v)\big)
\Big)
\end{align*}
and is therefore a spray which is tangential to the fibers of $q$. 
\end{proof}

\begin{proposition}
\nmb.{5.10} {\rm \cite[10.10]{Michor80}}
Let $q:N\to M$ be a smooth surjective submersion between connected manifolds without boundary.
The space $S^q(M,N)$ of all smooth sections of $q$ is a splitting smooth submanifold of $C^{\infty}(M,N)$. 
Similarly, the spaces $S^q_{C^N}(M,N)$ and $S^q_{W^{s,p}}(M,N)$ of $C^N$-sections and $W^{s,p}$-sections are smooth splitting submanifolds of 
$C^N(M,N)$ or $W^{s,p}(M,N)$ (for $s>\dim(M)/p$), respectively.  
\end{proposition}

The proof given here is simpler than the one in  \cite[10.10]{Michor80}.

\begin{proof} Let us first assume that $M$ is compact. 
Given a smooth section $f\in S^q(M,N)$, consider the chart centered at $f$ from \nmb!{5.3}
\begin{align*}
C^\infty(M,N)\supset U_f &=\{g: (f,g)(M)\subset V^{N\x N}\} \East{u_f}{} \tilde U_f 
\subset \Ga(M\gets f^*TN)
\\
u_f(g) = (\pi_N,&\exp^{S})\i \o (f,g),\quad u_f(g)(x) = (\exp^{S}_{f(x)})\i(g(x)) 
\\
(u_f)\i(s) &= \exp^{S}_f\o s, \qquad (u_f)\i(s)(x) = \exp^{S}_{f(x)}(s(x))
\end{align*}
where we use the exponential mapping with respect to a spray $S$ on $N$ which is tangential to the fibers of $q$. Using an unrelated auxiliary Riemannian metric $\bar g$ on $N$ we can smoothly split the tangent bundle  $TN = V^q(N)\oplus H^q(N)$ into the vertical bundle of  all vectors tangent to the fibers of $q$, and into its orthogonal complement with respect to $\bar g$. The orthonormal projections $P^{\bar g}:TN \to V^q(N)$ and $\on{Id}_{TN}-P^{\bar g}: TN\to H^q(N)$ induce  the direct
sum decomposition 
$$\Ga(M\gets f^*TN) = \Ga(M\gets f^*V^q(N))\oplus \Ga(M\gets f^*TN)
\quad s\mapsto (P^{\bar g}.s, s-P^{\bar g}.s).
$$    
Now $g\in U_f$ is in $S^q(M,N)$ if and only if $u_f(g)\in \Ga(f^*V^q(N))$.

If $M$ is not compact we may use the spaces of sections with compact support as described in \nmb!{5.3}. Similarly for the cases of $C^N$-sections or $W^{s,p}$-sections. 
\end{proof}

\begin{corollary}
\nmb.{5.11}
Let $p:E\to M$ be a fiber bundle over a compact Whitney manifold germ $M$. Then the space 
$\Ga(E)$ of smooth sections is a splitting smooth submanifold of $C^{\infty}(M,E)$. Likewise for the spaces $\Ga_{C^N}(E)$ and $\Ga_{W^{s,p}}(E)$ of $C^N$-sections and $W^{s,p}$-sections. 
\end{corollary}

\begin{proof}
Recall from \nmb!{4.8} that $E=\tilde E|_M$ for a smooth fiber bundle $\tilde E\to \tilde M$. 
There the result follows from \nmb!{5.10}.
Using  (fixed) extension operators 
$$
\Ga(M\gets f^*TE)\to \Ga_L(\tilde M\gets \tilde f^*T\tilde E)
$$
etc.\ we can extend this the case of Whitney manifold germs.  
\end{proof}

\section{Regular Lie groups}
\nmb0{6}

\subsection{Regular Lie groups}
\nmb.{6.1}
We consider a  smooth Lie group $G$ 
with Lie algebra $\mathfrak g=T_eG$ modeled on convenient 
vector spaces. 
The notion of a regular Lie group is originally due to 
\cite{OmoriMaedaYoshioka80,OmoriMaedaYoshioka81,OmoriMaedaYoshioka81b,OmoriMaedaYoshioka82,OmoriMaedaYoshioka83,OmoriMaedaYoshiokaKobayashi83}
for Fr\'echet Lie groups, was 
weakened and made more transparent by \cite{Milnor84}, and then carried over to convenient Lie 
groups in \cite{KM97r}, see also \cite[38.4]{KrieglMichor97}.
We shall write $\mu:G\x G\to G$ for the multiplication with $x.y = \mu(x,y) = \mu_x(y) = \mu^y(x)$ 
for left and right translation.

A Lie group $G$
is called {\em regular}  if the following holds:
\begin{itemize}
\item 
For each smooth curve 
$X\in C^{\infty}(\mathbb R,\mathfrak g)$ there exists a curve 
$g\in C^{\infty}(\mathbb R,G)$ whose right logarithmic derivative is $X$, i.e.,
$$
\begin{cases} g(0) &= e \\
\p_t g(t) &= T_e(\mu^{g(t)})X(t) = X(t).g(t)
\end{cases} 
$$
The curve $g$ is uniquely determined by its initial value $g(0)$, if it
exists.
\item
Put $\on{evol}^r_G(X)=g(1)$ where $g$ is the unique solution required above. 
Then $\on{evol}^r_G: C^{\infty}(\mathbb R,\mathfrak g)\to G$ is required to be
$C^{\infty}$ also. We have $\on{Evol}^X_t:= g(t) =\on{evol}^r_G(tX)$.
\end{itemize}
Of course we could equivalently use the left logarithmic derivative and the corresponding left evolution operator. Group inversion maps the two concepts into each other. See 
\cite[Section 38]{KrieglMichor97} for more information.
Up to now, every Lie group modeled on convenient vector spaces is regular. 

There are other notions of regularity for infinite dimensional Lie groups: For example, one may require that each curve 
$X\in L^1_{\text{loc}}(\mathbb R,\mathfrak g)$ admits an absolutely continuous curve $\on{Evol}^X:\mathbb R\to G$ whose right logarithmic derivative is $X$. See \cite{Glockner15} or \cite{Hanusch17} and references therein. It might be that all these notions of regularity are equivalent for Lie groups modeled on convenient vector spaces.  

\begin{theorem} \nmb.{6.2}
For each manifold $M$ with or without corners, the diffeomorphism group $\Diff(M)$ is a regular Lie group. Its Lie algebra is the space $\X(M)$ of all vector fields with the negative of the usual bracket as Lie bracket, if $M$ is compact without boundary. It is the space $\X_c(M)$ of fields with compact support, if $M$ is an open manifold. It is the space $\X_\p(M)$ of \nmb!{4.7} of vector fields tangent to the boundary, if $M$ is a compact manifold with corners. If $M$ is not compact with corners, then the Lie algebra is the space  $\X_{c,\p}(M)$ of boundary respecting vector fields with compact support. 
\end{theorem}

\begin{proof} If $M$ is a manifold without boundary then $\Diff(M)\East{open}{} C^\infty(M,M)$.
If $M$ is open, then the group of diffeomorphisms differing from the identity only on a compact set is open in $\Diff(M)$. 

If $M$ has corners we use an open manifold $\tilde M$ containing $M$ as a submanifold with corners as in \nmb!{3.2}. In the description of the chart structure \nmb!{5.3} for $\Diff(\tilde M)$ we have to use  the exponential mapping for a geodesic spray on  $\tilde M$ such that each component of each $\p^q M$ is totally geodesic. This spray exists; see \nmb!{3.7} or \nmb!{5.9}.  Restricting all sections to $M$ then yields a smooth chart centered at the identity for $\Diff(M)$. Then we use right translations of this chart. 
The explicit chart structure on $\Diff(M)$ is described in \cite[10.16]{Michor80}.
Extending all sections to $\tilde M$ via a fixed continuous linear Whitney extension operator respecting compact support identifies $\Diff(M)$ as a splitting smooth closed submanifold of $\Diff(\tilde M)$, but not as a subgroup.  

Composition is smooth by restricting it from $C^\infty(M,M )\x C^\infty(M, M)$, using $\nmb!{5.7}$ and its extension to the situation with corners. 

Inversion is smooth: If $t\mapsto f(t,\quad)$ is a smooth curve in $\Diff(M)$, then 
$f(t,\quad)\i$ satisfies the implicit equation $f(t,f(t,\quad)\i(x))=x$, so by the finite 
dimensional implicit function theorem, $(t,x)\mapsto f(t,\quad)\i(x)$ is smooth. So inversion maps 
smooth curves to smooth curves, and is smooth.

Let $X(t,x)$ be a time-dependent vector field on $M$ (in $C^\infty(\mathbb R,\X(M))$).
Then $\on{Fl}^{\p_t\x X}_s(t,x)=(t+s,\on{Evol}^X(t,x))$ satisfies the ordinary differential 
equation
$$
\quad\p_t\on{Evol}(t,x) = X(t,\on{Evol}(t,x)).
$$
If $X(s,t,x)\in C^\infty(\mathbb R^2,\X(M))$ is a smooth curve of smooth curves in $\X(M)$,
then obviously the solution of the equation depends smoothly also on the further variable $s$,
thus $\on{evol}$ maps smooth curves of time dependent vector fields to smooth curves of 
diffeomorphism. 
\end{proof}

\subsection{The diffeomorphism group of a Whitney manifold germ}
 \nmb.{6.3}
For a Whitney manifold germ $\tilde M\supset M$, we consider the diffeomorphism group 
\begin{multline*}
\Diff(M)=\{\ph|_M: \ph\in C^{\infty}(\tilde M,\tilde M),\; \ph(M)=M, 
\\
\ph \text{ is a diffeomorphism on an open neighborhood of   }M\}\,.
\end{multline*}
We also consider the following set  $\mathcal C$  of smooth curves:  Those $c:\mathbb R\to \Diff(M)$ which are of the form $c=\tilde c|_{\mathbb R\x M}$ for a smooth 
$$
\tilde c:\mathbb R\x \tilde M\to \tilde M\quad\text{ with } \tilde c(t,\;)|_M\in \Diff(M) \text{ for each }t\in \mathbb R. 
$$
Note that for $t$ in a compact interval $\tilde c(t,\;)$ is a diffeomorphism on a fixed open neighborhood of $M$ in 
$\tilde M$.

\begin{theorem}
\nmb.{6.4}
For a Whitney manifold germ $M$ the group $\Diff(M)$ is a Fr\"olicher space and a group with smooth composition and inversion. It has a convenient Lie algebra $\X_{c,\p}(M)$ with the negative of the usual bracket as Lie bracket, and it is regular: There exists an evolution operator  and it is smooth.
\end{theorem}

\begin{proof} 
The Fr\"olicher space structure is the one induced by the set $\mathcal C$ of smooth curves described above. 
I do not know whether this set of smooth curves is saturated, i.e.,  $\mathcal C= \mathcal C_{\Diff(M)}$ in the notation of \nmb!{2.7}; this might depend on the structure of the boundary.  

The proof is now quite similar to the one of  \nmb!{6.1}.
We claim that composition maps $\mathcal C\x \mathcal C$ to $\mathcal C\subseteq \mathcal C_{\Diff(M)}$, and that inversion maps $\mathcal C$ to $\mathcal C\subseteq \mathcal C_{\Diff(M)}$.
Since by definition each curve $c\in \mathcal C$ extend  to a smooth mapping $\tilde c:\mathbb R\x \tilde M\to \tilde M$ we can actually use a slight adaption of the proof of  \nmb!{6.1} for open manifolds.
\end{proof}

\subsection{The connected component of $\Diff(M)$ for a Whitney manifold germ $M$}
\nmb.{6.5}
We consider a Whitney manifold germ $M\subset \tilde M$.
As usual for Fr\"olicher space, we equip $\Diff(M)$ with the final topology with respect to all smooth curves in 
in the generating set $\mathcal C$ as described in \nmb!{6.3}. $\Diff(M)$ is actually a topological group, with the refined topology (i.e., the $c^\infty$-topology) on $\Diff(M)\x\Diff(M)$.
Let $\Diff_0(M)$ be the connected component of the identity in $\Diff(M)$ with respect to this topology.

\begin{theorem*}
For a Whitney manifold germ $M\subset \tilde M$ we actually have
$$
\Diff_0(M) = \{\tilde\ph|_M: \tilde\ph\in \Diff_0(\tilde M),\; \tilde\ph(M) = M\}\,.
$$
 Consequently, the subgroup 
$$\Diff{\tilde{\,}}(M) = \{\tilde\ph|_M: \tilde \ph\in \Diff(\tilde M), \, \tilde\ph(M)=M\}$$ 
is an open subgroup in 
$\Diff(M)$ and thus a normal subgroup, and the corresponding generating set $\mathcal C$ of smooth curves in 
$\Diff{\tilde{\,}}(M)$ is saturated.
\end{theorem*}

\begin{proof}
 Let $\ph\in \Diff_0(M)$. Then there exists a smooth curve $\ph:\mathbb R\to \Diff(M)$ with $\ph(0)=\on{Id}$ and $\ph(1)=\ph$ of the form $\ph=\tilde c|_{\mathbb R\x M}$ where $\tilde c:\mathbb R\x \tilde M\to \tilde M$ is a smooth mapping with $\tilde c(t,\;)|_M\in \Diff(M)$  for each $t\in \mathbb R$. Then $X(t,x) = (\p_t \ph(t))(\ph(t)\i(x))$ gives us a time-dependent vector field which is defined on 
$[0,1]\x U$ for some open neighborhood $U$ of $M$ in $\tilde M$, by the definition of $\Diff(M)$ in \nmb!{6.3}. Using a continuous extension operator on 
$X|_{[0,1]\x M}$ and a smooth bump function gives us a smooth time-dependent vector field  $\tilde X:[0,1]\x \tilde M\to T\tilde M$ with support in a fixed open neighborhood, say, such that $\tilde X|_{[0,1]\x M}= X|_{[0,1]\x M}$. Solving the ODE 
$\p_t \tilde\ph(t.x) = \tilde X(t,\tilde\ph(t,x))$ on $\tilde M$ gives us for $t=1$  a diffeomorphism 
$\tilde \ph \in \Diff(\tilde M)$   which extends $\ph$. 

Given any $\ph\in \Diff{\tilde{\,}}(M)$, the coset $\ph.\Diff_0(M)\subset\Diff(M)$ is the connected component of $\ph$ in $\Diff(M)$. This shows that $\Diff{\tilde{\,}}(M)$ is open in $\Diff(M)$.
\end{proof}

The construction in the proof  above actually describes a smooth mapping 
$$\mathcal E:\{c\in C^\infty(\mathbb R, \Diff_0(M)):c(0)=\on{Id}\}\to\{\tilde\ph\in \Diff_0(\tilde M): \tilde\ph(M)=M\}$$ 
such that $\mathcal E(c)|_M = c(1)$, since another smooth real parameter $s$ goes smoothly through solving the ODE. 

\subsection{Remark}
\nmb.{6.6}
In this paper I refrain from trying to give a general definition of a regular Fr\"olicher group, which would  be an abstract concept that catches the essential properties of $\Diff(M)$ for a Whitney manifold germ $M\subset \tilde M$. Let me just remark, that it probably would fit into the concept of manifolds based on smooth curves instead of charts as developed in \cite{Michor84}; those among them whose tangent spaces are Banach spaces turn out to be Banach manifolds.  Some Lie theoretic tools are developed in the beginning of Section \nmb!{8.5} below.


\subsection{Regular (right) half-Lie groups}
\nmb.{6.7}
A smooth manifold $G$ modeled on convenient vector spaces is called a \emph{(right) half Lie group}, if it is a group such that multiplication $\mu: G\x G\to G$ and inversion $\nu: G\to G$ are continuous (note that here we have to take the induced $c^\infty$-topology on the product $G\x G$ if the model spaces are not Fr\'echet), but each right translation $\mu^x:G\to G$, $\mu^x(y) = y.x$ is smooth. The notion of a half-Lie group was coined in \cite{KMR14}. See \cite{MarquisNeeb18} for a study of half-Lie groups in general, concentrating on semidirect products with representation spaces. 

Not every tangent vector in $T_eG$ can be extended to a 
left invariant vector field on the whole group, but they can be extended
to right invariant vector fields, which are only continuous and not differentiable in general.
The same holds for right invariant Riemannian metrics. The tangent space at the identity is not a 
Lie algebra in general; thus we refrain from calling it $\mathfrak g$. Have a look at the examples in \nmb!{6.8} to get a feeling for this. 

Let us discuss regularity on a (right) half-Lie group $G$: For a smooth curve $g:\mathbb R\to G$ the velocity curve $g':\mathbb R \to TG$ is still smooth, and for fixed $t$ the right logarithmic derivative 
$X(t):=g'(t).g(t)\i=T(\mu^{g(t)\i}).g'(t)$ lies in $T_{e}G$, but $t\mapsto X(t)$ is only continuous 
$\mathbb R\to T_{e}G$. 
A (right) half-Lie group $G$ is called \emph{$C^0$-regular} if for every $C^0$-curve 
$X:\mathbb R\to T_{e}G$ there exists a $C^1$-curve $\on{Evol}^X=g:\mathbb R\to G$ with $g(0)=e$ and $g'(t)=X(t).g(t)= T(\mu^{g(t)}).X(t)$. We also require that $X\mapsto\on{Evol}^X$ is smooth $C^0(\mathbb R,T_eG)\to C^1(\mathbb R, G)$. 

\begin{theorem}[Diffeomorphism groups of finite degrees of differentiability]\nmb.{6.8}
\thetag{1} 
For a compact smooth manifold $M$, possibly with corners, and for any $n\in \mathbb N_{\ge 1}$ the group $\Diff_{C^n}(M)$ of $C^n$-diffeomorphism of $M$ is a $C^0$-regular half-Lie group.

\smallskip\noindent\thetag{2}
For a compact smooth manifold $M$, possibly with corners, and for any $s\ge \dim{M}/p +1$, the group $\Diff_{W^{s,p}}(M)$ of Sobolev $W^{s,p}$-diffeomorphism of $M$ is a $C^0$-regular half-Lie group.
\end{theorem}

Note that the group of homeomorphisms of $M$ is {\bf not} open in $C^0_{\text{nice}}(M,M)$; see the proof below for $C^{\infty}_{\text{nice}}$.
Also note that $T_{\on{Id}}\Diff_{C^n}(M) = \X_{\p, C^n}(M)$ is the space of $C^n$-vector fields which are tangent to the boundary. This is not a Lie algebra, since the Lie bracket of two $C^n$ fields is a $C^{n-1}$ field in general.  

\begin{proof} \thetag{1} Following \cite[10.16]{Michor80},
we construct the smooth manifold structure by using the exponential mapping of a spray on $M$ which is tangential to the boundary; for existence see \nmb!{3.7} and \nmb!{5.9}. 
Let $C^{n}_{\text{nice}}(M,M,)$ be the set of all $C^n$-mappings $f:M\to M$ with $f\i(\p^qM)=\p^qM$ for each $q$. 
Then we use the (restriction of the) chart structure described in \nmb!{5.3}, using this exponential mappings, and using only charts centered at smooth mappings $f\in C^{\infty}_{\text{nice}}(M,M)$, as follows:
\begin{align*}
C^n_{\text{nice}}(M,N)&\supset U_f =\{g: (f,g)(M)\subset V^{M\x M}\} \East{u_f}{} \tilde U_f  \subset
\\
\subset \{s\in &C^{n}(M,TM): \pi_M\o s = f,  s(\p^q M) \subset T(\p^q M)\}
\subset \Ga_{C^n}( f^*T\tilde M)\,,
\\
u_f(g) &= (\pi_N,\exp^{\bar g})\i \o (f,g),\quad u_f(g)(x) = (\exp^{\bar g}_{f(x)})\i(g(x))\,, 
\\
(u_f)\i(s) &= \exp^{\bar g}_f\o s, \qquad (u_f)\i(s)(x) = \exp^{\bar g}_{f(x)}(s(x))\,.
\end{align*}
By the symmetry of $V^{M\x M}$ (see \nmb!{5.3}) these charts cover $C^{n}_{\text{nice}}(M,M)$, and the chart changes are smooth since they map smooth curves (as described in \nmb!{5.1}.2) to smooth curves; compare to \nmb!{5.7}. The group $\Diff_{C^n}(M)$ is open in $C^{n}_{\text{nice}}(M,M)$, by the implicit function theorem and some easy arguments. 

Continuity of composition and inversion are easy to check. Right translations are smooth since they map smooth curves.

$C^1$-regularity follows easily: Given $X\in C^0(\mathbb R,T_{\on{Id}}\Diff_{C^n}(M))$, view it as a time-dependent $C^n$-vector field on $M$ which is tangential to the boundary, a continuous curve in 
$\X_\p(M)$   and solve the corresponding ODE. The  evolution operator $\on{Evol}$ is smooth, since it maps smooth curves to smooth curves by standard ODE-arguments.   

\thetag{2}
This follows easily by adapting the proof of \thetag{1} above, using that  
$\Diff_{W^{s,p}}M\subset \Diff_{C^1}(M)$ by the Sobolev embedding lemma. 
\end{proof}

\subsection{Groups of smooth diffeomorphisms on $\mathbb R^n$}
\nmb.{6.9}
If we consider the group of all orientation preserving 
diffeomorphisms $\Diff(\mathbb R^n)$ of $\mathbb R^n$, it is not an open subset of 
$C^\infty(\mathbb R^n,\mathbb R^n)$ with the compact $C^\infty$-topology. 
So it is not a smooth manifold in the usual sense, but we may consider it as a Lie group in the 
cartesian closed category of  Fr\"olicher spaces, see \cite[Section 23]{KrieglMichor97}, with the structure 
induced by the injection 
$f\mapsto (f,f\i)\in C^\infty(\mathbb R^n,\mathbb R^n)\x C^\infty(\mathbb R^n,\mathbb R^n)$.
Or one can use the setting of `manifolds' based on smooth curves instead of charts, with lots of 
extra structure (tangent bundle, parallel transport, geodesic structure), described in 
\cite{Michor84}; this gives a category of smooth `manifolds' where 
those which have Banach spaces as tangent fibes are exactly the usual smooth manifolds modeled on 
Banach spaces, which is cartesian closed: $C^\infty(M,N)$ and $\Diff(M)$ are always `manifolds' for 
`manifolds' $M$ and $N$, and the exponential law holds. 

We shall now describe regular Lie groups in $\Diff(\mathbb R^n)$ 
which are given by 
diffeomorphisms of the form $f = \on{Id}_{\mathbb R} + g$ 
where $g$ is in some specific convenient vector space of bounded functions in 
$C^\infty(\mathbb R^n,\mathbb R^n)$.
Now we discuss these spaces on $\mathbb R^n$, we describe the smooth curves in them, and we describe the 
corresponding groups. These results are from \cite{MichorMumford13} and from \cite{KMR14,KMR15d} for the 
more exotic groups.

\subsubsection*{The group $\Diff_{\mathcal B}(\mathbb R^n)$}
The space $\mathcal B(\mathbb R^n)$ (called $\mathcal D_{L^\infty}(\mathbb R^n)$ by  
\cite{Schwartz66}) consists of all smooth functions which have all derivatives (separately) bounded.
It is a Fr\'echet space. 
By \cite{Vogt83}, the space $\mathcal B(\mathbb R^n)$ is linearly 
isomorphic to $\ell^\infty\hat\otimes\, \mathfrak s$ for any completed tensor-product between the 
projective one and the injective one, where $\mathfrak s$ is the nuclear Fr\'echet space of rapidly 
decreasing real sequences. Thus $\mathcal B(\mathbb R^n)$ is not reflexive, not nuclear, not 
smoothly paracompact.
\\
\emph{The space $C^\infty(\mathbb R,\mathcal{B}(\mathbb R^n))$ of smooth
curves in $\mathcal{B}(\mathbb R^n)$ consists of all functions 
$c\in C^\infty(\mathbb R^{n+1},\mathbb R)$ satisfying the following 
property:
\begin{enumerate}
\item[$\bullet$]
For all $k\in \mathbb N_{\ge0}$, $\al\in \mathbb N_{\ge0}^n$ and each $t\in \mathbb R$ the expression
$\p_t^{k}\p^\al_x c(t,x)$  is uniformly bounded in $x\in \mathbb R^n$, locally
in $t$. 
\end{enumerate} 
}
To see this use Theorem \nmb!{2.6} for the set $\{\on{ev}_x: x\in \mathbb R\}$ of point
evaluations in $\mathcal{B}(\mathbb R^n)$. 
Here $\p^\al_x = \frac{\p^{|\al|}}{\p x^\al}$ and $c^k(t)=\p_t^kf(t,\quad)$.
\\
\emph{$\Diff^+_{\mathcal B}(\mathbb R^n)=\bigl\{f=\on{Id}+g: g\in\mathcal B(\mathbb R^n)^n, 
    \det(\mathbb I_n + dg)\ge \ep > 0 \bigr\}$ denotes the corresponding group}, 
                see below.

\subsubsection*{The group $\Diff_{W^{\infty,p}}(\mathbb R^n)$}
For $1\le p <\infty$, the space 
$$W^{\infty,p}(\mathbb R^n)=\bigcap_{k\ge 1}L^p_k(\mathbb R^n)$$ 
is the intersection of all $L^p$-Sobolev 
spaces, the space of all smooth functions such that each partial derivative is in $L^p$. It is a reflexive Fr\'echet space.
It is called $\mathcal D_{L^p}(\mathbb R^n)$ in \cite{Schwartz66}.
By \cite{Vogt83}, the space $W^{\infty,p}(\mathbb R^n)$ is linearly isomorphic to 
$\ell^p\hat\otimes\, \mathfrak s$. Thus it is not nuclear, not Schwartz, not Montel, and smoothly 
paracompact only if $p$ is an even integer.
\\
\emph{The space $C^\infty(\mathbb R,H^\infty(\mathbb R^n))$ of smooth
curves in $W^{\infty,p}(\mathbb R^n)$ consists of all functions 
$c\in C^\infty(\mathbb R^{n+1},\mathbb R)$ satisfying the following 
property:
\begin{enumerate}
\item[$\bullet$]
For all $k\in \mathbb N_{\ge0}$, $\al\in \mathbb N_{\ge0}^n$  the expression 
$\|\p_t^{k}\p^\al_xf(t,\quad)\|_{L^p(\mathbb R^n)}$ is locally bounded  
near each $t\in \mathbb R$. 
\end{enumerate} 
}
The proof is literally the same as for $\mathcal B(\mathbb R^n)$, noting that the point evaluations are 
continuous on each Sobolev space $L^p_k$ with $k>\frac np$.
\\
\emph{$\Diff^+_{W^{\infty,p}}(\mathbb R^n)=\bigl\{f=\on{Id}+g: g\in W^{\infty,p}(\mathbb R^n)^n, 
                \det(\mathbb I_n + dg)>0\bigr\}$ 
                denotes the corresponding group}.

\subsubsection*{The group $\Diff_{\mathcal S}(\mathbb R^n)$}
The algebra $\mathcal S(R^n)$  of rapidly decreasing functions is a reflexive nuclear Fr\'echet space.
\\
\emph{The space $C^\infty(\mathbb R,\mathcal S(\mathbb R^n))$ of smooth
curves in $\mathcal S(\mathbb R^n)$ consists of all functions 
$c\in C^\infty(\mathbb R^{n+1},\mathbb R)$ satisfying the following 
property:
\begin{enumerate}
\item[$\bullet$]
For all $k,m\in \mathbb N_{\ge0}$ and $\al\in \mathbb N_{\ge0}^n$,  the expression
$(1+|x|^2)^m\p_t^{k}\p^\al_xc(t,x)$ is uniformly bounded in $x\in \mathbb R^n$, locally uniformly bounded  
in $t\in \mathbb R$.
\end{enumerate} 
}
\noindent
\emph{$\Diff^+_{\mathcal S}(\mathbb R^n)=\bigl\{f=\on{Id}+g: g\in \mathcal S(\mathbb R^n)^n, 
                \det(\mathbb I_n + dg)>0\bigr\}$ is the corresponding group.} 

\subsubsection*{The group  $\Diff_{c}(\mathbb R^n)$}
The algebra $C^\infty_c(\mathbb R^n)$ of all smooth functions with compact support is a nuclear 
(LF)-space.
\\ 
\emph{The space $C^\infty(\mathbb R,C^\infty_c(\mathbb R^n))$ of smooth
curves in $C^\infty_c(\mathbb R^n)$ consists of all functions 
$f\in C^\infty(\mathbb R^{n+1},\mathbb R)$ satisfying the following 
property:
\begin{enumerate}
\item[$\bullet$]
For 
each compact interval $[a,b]$ in $\mathbb R$ there exists a compact subset $K\subset \mathbb R^n$
such that $f(t,x)=0$ for  $(t,x)\in [a,b]\x (\mathbb R^n\setminus K)$.
\end{enumerate} 
}
\noindent
\emph{$\Diff_c(\mathbb R^n)=\bigl\{f=\on{Id}+g: g\in C^\infty_c(\mathbb R^n)^n, 
                \det(\mathbb I_n + dg)>0\bigr\}$ is the corresponding group.} 
    The case $\Diff_c(\mathbb R^n)$ is well-known since 1980. 

\subsubsection*{Ideal properties of function spaces}

The function spaces discussed are boundedly mapped into each other as follows:
$$\xymatrix{
C^\infty_c(\mathbb R^n) \ar[r]  & \mathcal S(\mathbb R^n)  \ar[r] & W^{\infty,p}(\mathbb R^n) 
\ar[r]^{p\le q} & W^{\infty,q}(\mathbb R^n) \ar[r]
 &  \mathcal B(\mathbb R^n)
}$$
and each space is a bounded locally convex algebra and a bounded $\mathcal B(\mathbb R^n)$-module.
Thus each space is an ideal in each larger space. 

\begin{theorem}[\cite{MichorMumford13} and \cite{KMR14}]\nmb.{6.10}
The sets of diffeomorphisms 
$$\Diff_c(\mathbb R^n),\quad 
\Diff_{\mathcal S}(\mathbb R^n),\quad
\Diff_{H^\infty}(\mathbb R^n), \text{ and }
\Diff_{\mathcal B}(\mathbb R^n)$$
are all smooth regular Lie groups.
We have the following smooth injective group homomorphisms
$$\xymatrix{
\Diff_c(\mathbb R^n) \ar[r] & \Diff_{\mathcal S}(\mathbb R^n) \ar[r] & \Diff_{W^{\infty,p}}(\mathbb R^n) \ar[r]  
&  \Diff_{\mathcal B}(\mathbb R^n)
}.$$
Each group is a normal subgroup in any other in which it is contained, in particular in 
$\on{Diff}_{\mathcal B}(\mathbb R^n)$.
\end{theorem}

The proof of this theorem relies on repeated use of the Fa\`a~di~Bruno formula for higher 
derivatives of composed functions. This offers difficulties on non-compact manifolds, where one 
would need a non-commutative Fa\`a~di~Bruno formula for iterated covariant derivatives.
In the paper \cite{KMR14} many more similar groups are discussed, modeled on spaces of 
Denjoy-Carle\-man ultradifferentiable functions. It is also shown that for $p>1$ the group
$\Diff_{W^{\infty,p}\cap L^1}(\mathbb R^n)$ is only a topological group with smooth 
right translations --- a property which is similar to the one of finite order Sobolev groups 
$\Diff_{W^{k,p}}(\mathbb R^n)$. Some of these groups were used extensively in 
\cite{MumfordMichor13}. 

\begin{corollary}\nmb.{6.11}
$\Diff_{\mathcal B}(\mathbb R^n)$ acts on $\Ga_c$, $\Ga_{\mathcal S}$ and $\Ga_{H^\infty}$ of any 
tensor bundle over $\mathbb R^n$ by pullback. The infinitesimal action of the Lie algebra $\X_{\mathcal B}(\mathbb R^n)$ on 
these spaces by the Lie derivative maps each of these spaces into itself. 
A fortiori, $\Diff_{H^\infty}(\mathbb R^n)$ acts on $\Ga_{\mathcal S}$ of any tensor bundle by pullback. 
\end{corollary}

\subsection{Trouv\'e groups}
\nmb.{6.12}
For the following see 
\cite{Trouve95}, \cite{Younes10}, \cite{NenningRainer19}.
Trouv\'e groups are useful for introducing topological metrics on certain groups of diffeomorphism on $\mathbb R^d$ starting from a suitable reproducing kernel Hilbert space of vector fields without using any Lie algebra structure; see \nmb!{8.12} below. 

Consider a time-dependent vector field $X:[0,1] \x \mathbb R^d\to \mathbb R$
of sufficient regularity (e.g, continuous in $t\in [0,1]$ and Lipschitz continuous in $x\in \mathbb R^d$ with $t$-integrable global Lipschitz constant) so that 
$$x(t) = x_0 + \int_0^t X(s,x(s))\,ds$$
is uniquely solvable for all $t\in [0,1]$ and $x_0\in \mathbb R^d$. Then we consider the \emph{evolution}
$\on{evol}^X(x_0) = x(1)$.   For 
$X\in L^1([0,1], C^1_b(\mathbb R^d,\mathbb R)^d)$ (where $f\in C^k_b$ if all iterated partial derivatives of order between 0 and $k$ are continuous and globally bounded) 
we have $\on{evol}^X\in \on{Id} + C^1_b(\mathbb R^d, \mathbb R^d)$ and is a diffeomorphism with $(\on{evol})\i \in \on{Id} + C^1_b(\mathbb R^d, \mathbb R^d)$. 
Given a convenient locally convex vector space $\mathcal A(\mathbb R^d,\mathbb R^d)$ of mappings $\mathbb R^d\to \mathbb R^d$ which continuously embeds into $C^1_b(\mathbb R^d,\mathbb R^d)$ and a suitable family of mappings $[0,1]\to \mathcal A(\mathbb R^d,\mathbb R^d)$, the associated \emph{Trouv\'e group} is given by 
$$
\mathcal G_{\mathcal A} := \{\on{evol}^X: X\in \mathcal F_ {\mathcal A}\}
$$
where $\mathcal F_ {\mathcal A} = \mathcal F([0,1], \mathcal A(\mathbb R^d,\mathbb R^d))$ is a suitable vector space of time-dependent vector fields. 
It seems that for a wide class of spaces $\mathcal A$ the Trouv\'e group $\mathcal G_{\mathcal A}$ is independent of the choice of $\mathcal F_{\mathcal A}$ if the latter contains the piecewise smooth curves and is contained in the curves which are integrable by seminorms; a precise statement is still lacking, but  see   \cite{NenningRainer19}, \cite{NenningRainer18}, \cite{Nenning19a},  and citations therein.  The space $\mathcal A$ is called \emph{$\mathcal F_ {\mathcal A}$-ODE-closed} if $\on{evol}^X \in \on{Id}+\mathcal A(\mathbb R^d,\mathbb R^d)$ for each $X\in \mathcal F_ {\mathcal A}$. For ODE-closed $\mathcal A$ the Trouv\'e group $\mathcal G_{\mathcal A}$ is contained in    
$\on{Id}+ \mathcal A(\mathbb R^d,\mathbb R^d)$. 

For some spaces $\mathcal A$ it has been proved that $\mathcal F_{\mathcal A}$ is equal to the connected component of the identity of 
$$
\{\on{Id} + f: f\in \mathcal A(\mathbb R^d,\mathbb R^d), \inf_{x\in \mathbb R^d} \det df(x) > -1\},
$$ 
namely:
\begin{itemize}
\item
For Sobolev spaces $W^{k,2}$ with $k>d/2$ by \cite{BruverisViallard17}; $\mathcal G_{\mathcal A}$ is a half Lie group. 
\item
For H\"older spaces by  \cite{NenningRainer18}. 
\item
For Besov spaces by \cite{Nenning19}. 
\item
For $\mathcal B$, $W^{\infty,p}$, Schwartz functions $\mathcal S$, $C^\infty_c$, and many classes of Denjoy-Carleman functions, where  $\mathcal G_{\mathcal A}$ is always a regular Lie group; see  \cite{NenningRainer19}.
\end{itemize}


\section{Spaces of embeddings or immersions, and shape spaces}\nmb0{7}

This is the main section in this chapter.

\subsection{The principal bundle of embeddings}
\nmb.{7.1}
For finite dimensional manifolds $M$, $N$ with $M$ compact,
$\Emb(M,N)$, the space of embeddings of $M$ into $N$, is open in $C^\infty(M,N)$, so it is a smooth 
manifold. $\Diff(M)$ acts freely and smoothly from the right on $\Emb(M,N)$.

\begin{theorem*} 
$\Emb(M,N)\to \Emb(M,N)/\Diff(M) = B(M,N)$ is a smooth principal fiber bundle with structure group 
$\Diff(M)$. Its base is a smooth manifold.
\end{theorem*}

This result was proved in \cite{Michor80III} for $M$ an open manifold without boundary; see also \cite{Michor80}.
Note that $B(M,N)$ is the smooth  manifold of all submanifolds of $N$ which are of diffeomorphism 
type $M$. Therefore it is also called the \emph{nonlinear Grassmannian} in \cite{G-BV14}, where this 
theorem is extended to the case when $M$ has boundary. From another point of view, $B(M,N)$ is 
called the \emph{differentiable Chow variety} in \cite{Micheli2013}. It is an example of a \emph{shape 
space}.

\begin{proof} We use an auxiliary Riemannian metric $\bar g$ on $N$. Given an embedding $f\in \Emb(M,N)$, we 
view $f(M)$ as a submanifold of $N$ and we split the tangent bundle of $N$ along $f(M)$ as 
$TN|_{f(M)}= \on{Nor}(f(M))\oplus Tf(M)$. The exponential mapping describes a tubular neighborhood 
of $f(M)$ via
$$
\on{Nor}(f(M))\East{\exp^{\bar g}}{\cong} W_{f(M)}\East{p_{f(M)}}{} f(M).
$$
If $g:M\to N$ is $C^1$-near to $f$, then $\ph(g):=f\i\o p_{f(M)}\o g\in \Diff(M)$ and we may 
consider 
$g\o \ph(g)\i\in \Ga(f^*W_{f(M)}) \subset \Ga(f^*\on{Nor}(f(M)))$. 
This is the required local splitting.
\end{proof}

\subsection{The space of immersions and the space of embeddings of a compact Whitney manifold germ}
\nmb.{7.2}
Let $\tilde M\supset M$ be a compact  Whitney manifold germ, and let $N$ be a smooth manifold with $\dim(M)\le \dim(N)$. We define the space of immersions as
$$
\Imm(M,N) = \{f=\tilde f|_M, f\in C^{\infty}(\tilde M,N), 
T_x\tilde f \text{ is injective for }x\in  M\}
$$
which is open in the smooth manifold $C^{\infty}(M,N)$ and is thus itself a smooth manifold.
Note that any extension of an immersion $f\in \Imm(M,N)$ to $\tilde f\in C^{\infty}(\tilde M,N)$ is still an immersion on an open neighborhood of $M$ in $\tilde M$. 

Likewise we let 
\begin{align*}
\Emb(M,N) &= 
\{f|_M, f\in C^{\infty}(\tilde M,N), T_xf \text{ is injective for }x\in  M, 
\\&\qquad\qquad
f:M\to N \text{ is a topological embedding}\}
\end{align*}
Since $M$ is compact, any extension of an embedding $f\in \Emb(M,N)$ to $\tilde f\in C^{\infty}(\tilde M,N)$ is an embedding on some open neighborhood of $M$ in $\tilde M$; see \cite[5.3]{Michor80} for a proof a related result.

\begin{theorem*} For a compact Whitney germ $M$ and a smooth manifold $N$ with 
$\dim(M)<\dim (N)$ the projection  
$$\pi:\Emb(M,N)\to \Emb(M,N)/\Diff(M) = B(M,N)$$ 
is a smooth principal fiber bundle of Fr\"olicher spaces with structure group the Fr\"olicher group 
$\Diff(M)$ from {\rm \nmb!{6.4}}. Its base is the quotient Fr\"olicher space. 
\end{theorem*}

\begin{proof} Since I do not know that $\Diff(M)$ is a smooth manifold, we treat all spaces here as Fr\"olicher spaces.
By definition, the right action of $\Diff(M)$ on $\Emb(M,N)$ is free, and smooth between the Fr\"olicher spaces. The quotient $B(M,N)$ carries the quotient Fr\"olicher  
structure with generating set of curves  $\{\pi\o c: c\in C^{\infty}(\mathbb R,\Emb(M,N))\}$, i.e.,   those which  lift to a smooth curve.
\end{proof}

\subsection{The orbifold bundle of immersions}
\nmb.{7.3} 
Let $M$ be a (not necessarily compact) manifold without boundary. Let $N$ be an open manifold with $\dim(M)\le \dim(N)$.
Then  
$\Imm(M,N)$, the space of immersions $M\to N$, is open in $C^\infty(M,N)$, and is thus a smooth 
manifold. The regular Lie group (or Fr\"olicher group if $M$ is a Whitney manifold germ) $\Diff(M)$ acts smoothly from the right, but no longer freely. 

An immersion $i:M\to N$ is called \emph{free} if $\on{Diff}(M)$ acts freely on it: $i\o f= i$ for 
$f\in \Diff(M)$ implies $f=\on{Id}_M$.

The space $B_i(M,N) = \Imm(M,N)/\Diff(M)$ is an example of a \emph{shape space}. It appeared in the form of 
$B_i(S^1,\mathbb R^2)$, the shape space of plane immersed curves, in \cite{Michor2006c} and 
\cite{Michor2007}.
The following theorem  is essentially due to 
\cite{CerveraMascaroMichor91}; since this paper contains some annoying misprints and is difficult to understand, we give here an extended version with a  more detailed proof. The reader may skip this proof and  jump directly to \nmb!{7.2} below.

\begin{theorem*}
[\cite{CerveraMascaroMichor91}] 
Let $M$ be a finite dimensional smooth manifold. Let $N$ be 
smooth finite dimensional manifolds with $\dim(M)\le \dim(N)$. Then the following holds:
\begin{enumerate}
\item[\thetag{\nmb:{1}}] 
 The diffeomorphism group 
$\on{Diff}(M)$ acts smoothly from the right on the 
manifold $\on{Imm}_{\text{\rm prop}}(M,N)$
of all smooth proper immersions $M\to N$, which is an open subset
of $C^\infty(M,N)$. 
\item[\thetag{\nmb:{2}}]
The space of orbits 
$\on{Imm}_{prop}(M,N)/\on{Diff}(M)$
is Hausdorff in the quotient topology.
\item[\thetag{\nmb:{3}}]
The set $\on{Imm}_{\text{\rm free, prop}}(M,N)$  of all 
proper free immersions
is open in $C^\infty(M,N)$ and is the total space of a smooth principal fiber 
bundle 
$\on{Imm}_{\text{\rm free,prop}}(M,N)\to \on{Imm}_{\text{\rm free,prop}}(M,N)/\on{Diff}(M)$.
\item[\thetag{\nmb:{4}}]
Let $i\in\on{Imm}(M,N)$ be an immersion which is 
not free. So we have a nontrivial isotropy subgroup 
$\Diff(M)_i\subset \Diff(M)$ consisting of all $f\in\Diff(M)$ with 
$i\o f= i$. Then the isotropy group $\Diff(M)_i$ acts
properly discontinuously on $M$. Thus the projection 
$q_1:M\to M_1:= M/\Diff(M)_i$ is a covering mapping onto a smooth manifold $M_1$. 
There exists an immersion $i_1:M_1\to N$ with $i=i_1\o q_1$. In particular, $\Diff(M)_i$ 
is countable, and is finite if $M$ is compact. There exists a further covering $q_2:M\to M_1\to M_2$ and 
a \emph{free} immersion $i_2:M_2\to N$ with $i=i_2\o q_2$.
\item[\thetag{\nmb:{5}}]
Let $M$ have the property that for any covering $M\to M_1$ of smooth manifolds, 
any diffeomorphism $M_1\to M_1$ admits a lift $M\to M$; e.g., $M$ simply connected, or $M=S^1$.
Let $i\in Imm(M,N)$ be an 
immersion which is not free, i.e., has non-trivial isotropy group 
$\Diff(M)_i$, and let $q_1:M\to M_1:=M/\Diff(M)_i$ be the corresponding covering map.
Then 
in the following commutative diagram the bottom mapping
$$\xymatrix@C+1cm@R=1.3em{
\on{Imm}_{\text{\rm free}}(M_1,N) \ar[r]^{(q_1)^*} \ar[d]^{\pi} &  \on{Imm}(M,N) \ar[d]_{\pi}
\\
\on{Imm}_{\text{\rm free}}(M_1,N)/\Diff(M_1) \ar[r] & \on{Imm}(M,N)/\Diff(M)
}$$ 
is the inclusion of a (possibly non-Hausdorff) manifold, 
the stratum of $\pi(i)$ in the 
stratification of the orbit space. This stratum consists of the 
orbits of all immersions which have $\Diff(M)_i$ as isotropy group.
See \thetag{\nmb|{23}} and \thetag{\nmb|{24}} below for a more complete description of the orbit 
structure near $i$.
\item[\thetag{\nmb:{6}}] 
\text{\cite{Swift93}} 
We have a right action 
of $\Diff(M)$ on $\on{Imm}(M,N)\x M$ which is given by $(i,x).f= (i\o f, f\i(x))$. This action  is free.
$$(\on{Imm}(M,N)\x M, \pi, (\on{Imm}(M,N)\x M)/\Diff(M), \Diff(M))$$
is a smooth principal fiber bundle with structure group $\Diff(M)$ and a smooth base manifold 
$S(M,N):= (\on{Imm}(M,N)\x M)/\Diff(M)$ which might possibly be non-Hausdorff.
If we restrict to the open subset $\on{Imm}_{\text{\rm prop}}(M,N)\x M$ of 
proper immersions times $M$ then the base space is Hausdorff.
\end{enumerate}
\end{theorem*}

\begin{proof} Without loss, let $M$ be connected. Fix an immersion $i:M\to N$.
We will now describe some data for $i$
which we will use throughout the proof. If we need these data for
several immersions, we will distinguish them by appropriate 
superscripts.

\medskip\noindent
{\bf\thetag{\nmb:{7}} Setup.}
There exist sets
$W_\al\subset\overline{W}_\al\subset U_\al\subset \overline{U}_\al\subset V_\al\subset M$ such that
$(W_\al)$ is an open cover of $M$, $\overline W_\al$ is compact,
and $V_\al$ is an open locally finite cover of $M$, each $W_\al$,
$U_\al$, and $V_\al$ is connected, and such that $i|V_\al:V_\al\to N$ is
an embedding for each $\al$.

Let $g$ be a fixed Riemannian metric on $N$ and let $\exp^N$ be
the induced geodesic exponential mapping. Then let 
$p:\mathcal N(i)\to M$ be the \emph{normal bundle} of $i$, defined in
the following way: For $x \in M$ let $\mathcal N(i)_x:=
(T_xi(T_xM))^\bot\subset T_{i(x)}N$ be the $g$-orthogonal
complement in $T_{i(x)}N$. Then 
$$\xymatrix@C+0.8cm@R=2em{
\mathcal N(i) \ar[r]_{\bar i} \ar[d]_{p} & TN \ar[d]^{\pi_N}
\\
M  \ar[r]^{i} &   N
}$$
is a vector bundle homomorphism over $i$, which is fiberwise
injective.  

Now let $U^i=U$ be an open neighborhood of the zero section of $\mathcal N(i)$ which
is so small that 
$(\exp^N\o \bar i)|(U|_{V_\al}): U|_{V_\al} \to N$
is a diffeomorphism onto its image which describes a tubular
neighborhood  of the submanifold $i(V_\al)$ for each $\al$.
Let 
$$\ta=\ta^i:= (\exp^N\,\o \bar i\,)|_U:\mathcal N(i)\supset U\to N.$$
It will serve us as a substitute for a tubular neighborhood of $i(M)$.

For any $f\in \Diff(M)_i=\{f\in\Diff(M): i\o f= i\}$ we have an induced vector bundle homomorphism 
$\bar f$ over $f$:
$$\xymatrix@C+0.8cm{
\mathcal N(i) \ar[r]_{\mathcal N(f)} \ar[d]_{p} \ar@/^0.8cm/[rr]_{\bar i} & 
\mathcal N(i) \ar[d]^{p} \ar[r]_{\bar i} & 
TN \ar[d]^{\pi_N}
\\
M  \ar[r]^{f} &   M \ar[r]^{i} & N
}$$

\medskip\noindent
{\bf\thetag{\nmb:{8}} Claim.}
\emph{Let $i\in \on{Imm}(M,N)$ and let 
$f\in\Diff(M)$ have a fixed point $x_0\in M$ and satisfy 
$i\o f=i$. Then $f=Id_M$.}

Namely,  we consider the sets $(U_\al)$ for the immersion $i$ 
of \thetag{\nmb|{7}}. Let us investigate $f(U_\al)\cap U_\al$.
If there is an $x\in U_\al$ with $y=f(x)\in U_\al$, we have
$(i|_{U_\al})(x)=((i\o f)|U_\al)(x)=(i|_{U_\al})(f(x))=(i|_{U_\al})(y)$.
Since $i|_{U_\al}$ is injective we have $x=y$, and  
$$f(U_\al)\cap U_\al=\{x\in U_\al: f(x)=x\}.$$
Thus $f(U_\al)\cap U_\al$ is closed in $U_\al$. Since it is also
open and since $U_\al$ is connected, we have 
$f(U_\al)\cap U_\al=\emptyset$ or $=U_\al$.

Now we consider the set $\{x\in M: f(x)=x\}$. We have just shown
that it is open in $M$. Since it is also closed and contains the
fixed point $x_0$, it coincides with $M$. Claim \thetag{\nmb|{7}} follows.

\medskip\noindent
{\bf\thetag{\nmb:{9}} Claim.}
\emph{If for an immersion $i\in
\on{Imm}(M,N)$ there is a point in $i(M)$ with only one preimage,
then $i$ is a free immersion.}

Let $x_0\in M$ be such that $i(x_0)$ has only one
preimage. If $i\o f=i$ for $f\in \Diff(M)$ then $f(x_0)=x_0$ and
$f=Id_M$ by claim \thetag{\nmb|{8}}.

Note that there are free immersions without a point in $i(M)$ with 
only one preimage: Consider a figure eight which consists of two 
touching circles. Now we may map the circle to the figure eight by 
going first $n$ times around the upper circle, then $m$ around the 
lower one with $n,m\ge 2$. 

\medskip\noindent
{\bf\thetag{\nmb:{10}} Claim.}
\emph{Let $i$ be a free immersion 
$M\to N$. Then there is an open neighborhood $\mathcal W(i)$ in $\on{Imm}(M,N)$ 
which is saturated for the $\Diff(M)$-action and which splits 
smoothly as 
$$\mathcal W(i)=\mathcal Q(i)\x \Diff(M).$$
Here $\mathcal Q(i)$ is a smooth splitting submanifold of $\on{Imm}(M,N)$, 
diffeomorphic to an open neighborhood of the zero section in $\Ga_c(M\gets \mathcal N(i))$. 
In particular the space $\on{Imm}_{\text{\rm free}}(M,N)$ 
is open in $C^\infty(M,N)$.}

\emph{Let $\pi:\on{Imm}(M,N)\to \on{Imm}(M,N)/\Diff(M)=B_{i}(M,N)$ be the projection onto the 
orbit space, which is equipped with the quotient topology. Then the mapping
$\pi|_ {\mathcal Q(i)}:\mathcal Q(i)\to \pi(\mathcal Q(i))$ is bijective 
onto an open subset of the quotient. If $i$ runs through 
$\on{Imm}_{\text{\rm free,prop}}(M,N)$ of all free and proper immersions 
these mappings define a smooth atlas for 
the quotient space, so that 
$$(\on{Imm}_{\text{\rm free,prop}}(M,N), \pi, \on{Imm}_{\text{\rm free,prop}}(M,N)/\Diff(M), 
\Diff(M))$$
is a smooth principal fiber bundle with structure group $\Diff(M)$.
}

The restriction to proper immersions is necessary because we are only
able to show that $\on{Imm}_{\text{\rm prop}}(M,N)/\Diff(M)$ is Hausdorff in \thetag{\nmb|{11}}
below.

For the proof of claim \thetag{\nmb|{10}},
we consider the setup \thetag{\nmb|{7}} for the free immersion $i$. Let 
$$\tilde{\mathcal U}(i):=\{j\in \on{Imm}(M,N):j(\overline{W}^i_\al)\subseteq 
     \ta^i(U^i|_ {U^i_\al})\text{ for all }\al, j \sim i \},$$
where $j\sim i$ means that $j=i$ off some compact set in $M$. 
Then by \nmb!{5.3} (for open $M$) the set $\tilde{\mathcal U}(i)$ is an 
open neighborhood of $i$ in $\on{Imm}(M,N)$.
For each $j\in\tilde{\mathcal U}(i)$ we define 
\begin{align*} 
&\ph_i(j):M\to U^i\subseteq \mathcal N(i),\\
&\ph_i(j)(x):=(\ta^i|_ {(U^i|_{ U^i_\al})})\i(j(x))
     \text{ if } x\in W^i_\al.
\end{align*}
Note that $\ph_i(j)$ is defined piecewise on $M$, but the pieces coincide when they overlap. 
Therefore a smooth curve through $j$ is mapped to a smooth curve and so 
$\ph_i:\tilde{\mathcal U}(i)\to C^\infty(M,\mathcal N(i))$ is a smooth mapping which is 
bijective onto the open set 
$$\tilde{\mathcal V}(i):=\{h\in C^\infty(M,\mathcal N(i)): 
     h(\overline W^i_\al)\subseteq U^i|_{ U^i_\al}
     \text{ for all }\al, h\sim 0\}$$
in $C^\infty(M,\mathcal N(i))$. Its inverse is given by the smooth 
mapping $\ta^i_*:h\mapsto \ta^i\o h$.
Now we consider the open subsets
\begin{align*}
\mathcal V(i) :&= \{h\in \tilde{\mathcal V}(i): p\o h \in \Diff_c(M)\} \subset \tilde{\mathcal V}(i)
\\
\mathcal U(i) :&= \ta^i_*(\mathcal V(i)) \subset \tilde{\mathcal U}(i)
\end{align*}
and the diffeomorphism $\ph_i:\mathcal U(i)\to \mathcal V(i)$.
For $h\in \mathcal V(i)$ we have $\ta^i_*(h\o f)=\ta^i_*(h)\o f$ for those 
$f\in\Diff(M)$ which are near enough to the identity so that 
$h\o f\in \mathcal V(i)$. And if $\ta^i\o h\o f = \ta^i\o h$ then $h\o f = 
h$ by the construction of $\mathcal N(i)$ in \thetag{\nmb|{7}}, 
and then $f=\on{Id_M}$ since $i$ is a free immersion; see the second diagram in \thetag{\nmb|{7}}. 

We consider now the open set 
$$\{h\o f: h\in \mathcal V(i), 
     f\in \Diff(M)\}\subseteq C^\infty(M,U^i).$$
Consider the smooth mapping from it into 
$\Ga_c(M\gets U^i)\x \Diff(M)$ given by 
$h\mapsto (h\o(p\o h)\i,p\o h)$, where $\Ga_c(M\gets U^i)$ is the space 
of sections with compact support of $U^i\to M$.
So if we let 
$\mathcal Q(i):= \ta^i_*(\Ga_c(M\gets U^i)\cap\mathcal V(i))\subset \on{Imm}(M,N)$
we have 
$$\mathcal W(i):= \mathcal U(i)\o \Diff_c(M)\cong \mathcal Q(i)\x \Diff(M) \cong
     (\Ga_c(M\gets U^i)\cap\mathcal V(i))\x \Diff(M),$$
since the action of $\Diff(M)$ on $i$ is free and by the argument above.
Consequently $\Diff(M)$ acts freely on each immersion in $\mathcal W(i)$, 
so $\on{Imm}_{\text{\rm free}}(M,N)$ is open in $C^\infty(M,N)$. Furthermore 
$$\pi|_{\mathcal Q(i)}: 
\mathcal Q(i)\to \on{Imm}_{\text{\rm free}}(M,N)/\Diff(M)$$
is bijective onto an open set in the quotient.

We  consider 
$$ 
\ph_i\o(\pi|_{\mathcal Q(i)})\i:\pi(\mathcal Q(i))\to \Ga_c(M\gets U^i)\subset C^\infty_c(N,\mathcal N(i))
$$
as a chart for the quotient space. 

In order to investigate the chart 
change let $j\in \on{Imm}_{\text{\rm free}}(M,N)$ be such that 
$\pi(\mathcal Q(i))\cap\pi(\mathcal Q(j))\ne\emptyset$. 
Then there is an immersion $h\in \mathcal W(i)\cap \mathcal Q(j)$, so there 
exists a unique $f_0\in\Diff(M)$ (given by $f_0=p\o\ph_i(h)$) such 
that $h\o f_0\i\in\mathcal Q(i)$. If we consider $j\o f_0\i$ instead of 
$j$ and call it again $j$, we have 
$\mathcal Q(i)\cap\mathcal Q(j)\ne\emptyset$ and consequently 
$\mathcal U(i)\cap\mathcal U(j)\ne\emptyset$.
Then the chart change is given as follows:
\begin{gather*}
\ph_i\o(\pi|_{\mathcal Q(i)})\i\o\pi\o(\ta^j)_*:
     \Ga_c(M\gets U^j)\to \Ga_c(M\gets U^i)\\
s\mapsto \ta^j\o s 
     \mapsto \ph_i(\ta^j\o s)\o(p^i\o\ph_i(\ta^j\o s))\i.
\end{gather*}
This is of the form $s\mapsto \be\o s$ for a locally defined 
diffeomorphism $\be:\mathcal N(j)\to\mathcal N(i)$ which is not fiber 
respecting, followed by $h\mapsto h\o(p^i\o h)\i$. Both composants 
are smooth by the general properties of manifolds of mappings. 
So the chart change is smooth.

We have to show that the quotient space 
$\on{Imm}_{\text{\rm prop,free}}(M,N)/\Diff(M)$ is Hausdorff. 

\medskip\noindent
{\bf\thetag{\nmb:{11}} Claim.}
\emph{
The orbit space
$\on{Imm}_{\text{\rm prop}}(M,N)/\Diff(M)$ of the space of all proper
immersions under the action of the diffeomorphism group is
Hausdorff in the quotient topology.
}

This follows from \thetag{\nmb|{18}} below. I am convinced that the whole orbit space 
$\on{Imm}(M,N)/\Diff(M)$ is Hausdorff, but I was unable to prove
this.

\medskip\noindent
{\bf\thetag{\nmb:{12}} Claim.}
\emph{
Let $i$ and $j\in \on{Imm}_{\text{\rm prop}}(M,N)$ with 
$i(M)\ne j(M)$ in $N$. Then their projections $\pi(i)$ and $\pi(j)$ 
are different and can be separated by open subsets in 
$\on{Imm}_{\text{\rm prop}}(M,N)/\Diff(M)$.
}

We suppose that $i(M)\nsubseteq \overline{j(M)} = j(M)$ 
(since proper immersions have closed images). 
Let $y_0\in i(M)\setminus\overline{j(M)}$, then we choose open 
neighborhoods $V$ of $y_0$ in $N$ and $W$ of $j(M)$ in $N$ such that 
$V\cap W = \emptyset$. We consider the sets
\begin{align*}
\mathcal V &:= \{k\in \on{Imm}_{\text{\rm prop}}(M,N):k(M)\cap V\neq\emptyset\}\quad\text{and}\\
\mathcal W &:= \{k\in \on{Imm}_{\text{\rm prop}}(M,N):k(M)\subseteq W\}.
\end{align*}
Then $\mathcal V$ and $\mathcal W$ are $\Diff(M)$-saturated 
disjoint open neighborhoods of $i$ and $j$, respectively, so 
$\pi(\mathcal V)$ and $\pi(\mathcal W)$ separate $\pi(i)$ and $\pi(j)$ in the quotient space 
$\on{Imm}_{\text{\rm prop}}(M,N)/\Diff(M)$.

\medskip\noindent
{\bf\thetag{\nmb:{13}} Claim.}
\emph{For a proper immersion $i:M\to N$ and 
$x\in i(M)$ let $\de(x)\in \mathbb N$ be the number of points in 
$i\i(x)$. Then  $\de:i(M)\to \mathbb N$ 
is upper 
semicontinuous, i.e.,  the set $\{x\in i(M):\de(x)\leq k\}$ is open in 
$i(M)$ for each $k$.}

Let $x\in i(M)$ with $\de(x)=k$ and let $i\i(x)=\{ y_1,\dots,y_k\}$. Then 
there are pairwise disjoint open neighborhoods $W_n$ of $y_n$ in $M$ 
such that $i|_{W_n}$ is an embedding for each $n$. The set 
$M\setminus(\bigcup_n W_n)$ is closed in $M$, and since $i$ is proper 
the set $i(M\setminus(\bigcup_n W_n))$ is also closed in $i(M)$ and 
does not contain $x$. So there is an open neighborhood $U$ of $x$ in 
$i(M)$ which does not meet $i(M\setminus(\bigcup_n W_n))$.  
Obviously $\de(z)\leq k$ for all $z\in U$.

\medskip\noindent
{\bf\thetag{\nmb:{14}} Claim.}
\emph{Consider two proper immersions $i_1$ 
and $i_2\in \on{Imm}_{\text{\rm prop}}(M,N)$ such that 
$i_1(M)=i_2(M)=:L\subseteq N$. 
Then we have mappings $\de_1, \de_2: L\to \mathbb N$ as in \thetag{\nmb|{13}}.
If  $\de_1\ne\de_2$ then the projections $\pi(i_1)$ and $\pi(i_2)$ are 
different and can be 
separated by disjoint open neighborhoods in 
$\on{Imm}_{\text{\rm prop}}(M,N)/\Diff(M)$.}

Let us suppose that $m_1=\de_1(y_0)\ne\de_2(y_0)=m_2$. 
There is a small connected open neighborhood $V$ of $y_0$ in $N$ such 
that $i_1\i(V)$ has $m_1$ connected components and $i_2\i(V)$ has 
$m_2$ connected components. This assertions 
describe Whitney $C^0$-open neighborhoods in $\on{Imm}_{\text{\rm prop}}(M,N)$ of 
$i_1$ and $i_2$ which are closed under the action of $\Diff(M)$, 
respectively. Obviously these two neighborhoods are disjoint.

\medskip\noindent
{\bf \thetag{\nmb:{15}} Assumption.} We assume   
that we are given two immersions $i_1$ and $i_2\in \on{Imm}_{\text{\rm prop}}(M,N)$ with
$i_1(M)=i_2(M)=:L$ such that the functions from \thetag{\nmb|{14}} are equal:
$\de_1=\de_2=:\de$.

Let $(L_\be)_{\be\in B}$ be the partition of $L$ consisting of all 
pathwise connected components of level sets $\{x\in L:\de(x)=c\}$, $c$ 
some constant. 

Let $B_0$ denote the set of all $\be\in B$ such that the interior of 
$L_\be$ in $L$ is not empty. Since $M$ is second countable, $B_0$ is
countable.

\medskip\noindent
{\bf\thetag{\nmb:{16}} Claim.}
\emph{$\bigcup_{\be\in B_0}L_\be$ is dense in $L$.
}

Let $k_1$ be the smallest number in $\de(L)$ and let $B_1$ be the set 
of all $\be\in B$ such that $\de(L_\be)=k_1$. Then by claim 
\thetag{\nmb|{13}} each $L_\be$ for $\be\in B_1$ is open. Let $L^1$ be the 
closure of $\bigcup_{\be\in B_1}L_\be$. Let $k_2$ be the smallest 
number in $\de(L\setminus L^1)$ and let $B_2$ be the set of all 
$\be\in B$ with $\be (L_\be)=k_2$ and 
$L_\be\cap(L\setminus L^1)\ne\emptyset$. Then by claim \thetag{\nmb|{13}} again  
$L_\be\cap(L\setminus L^1)\ne\emptyset$ is open in $L$ so $L_\be$ has 
non empty interior for each $\be\in B_2$. Then let $L^2$ denote the 
closure of $\bigcup_{\be\in B_1\cup B_2}L_\be$ and continue the process. 
If $\de(L)$ is bounded, the process stops. If $\de(L)$ is unbounded, 
by claim \thetag{\nmb|{13}} we always find new $L_\be$ with non empty 
interior, we finally exhaust $L$ and  claim \thetag{16} follows.

Let $(M^1_\la)_{\la\in C^1}$ be a suitably chosen cover of $M$ by subsets 
of the sets $i_1\i(L_\be)$ such that: 
\begin{enumerate}
\item[(i)] Each
$i_1|_ {\on{int} M^1_\la}$ is an embedding for each 
$\la$.
\item[(ii)] The set $C^1_0$ of all $\la$ with $M^1_\la$ having non 
empty interior is at most countable.
Let $(M^2_\mu)_{\mu\in C^2}$ be a cover chosen in a similar way for $i_2$.
\item[(iii)] For each pair $(\mu,\la)\in C^2_0\x C^1_0$ the  two open sets 
$i_2(\on{int}(M^2_\mu))$ and $i_1(\on{int}(M^1_\la))$ in $L$ are either equal or disjoint.
\end{enumerate}
Note that the union
$\bigcup_{\la\in C^1_0}\on{int} M^1_\la$ is dense in $M$ and thus
$\bigcup_{\la\in C^1_0}\overline{M^1_\la}= M$; similarly 
for the $M^2_\mu$.

\medskip\noindent
{\bf\thetag{\nmb:{17}} Procedure.}
Given immersions $i_1$ and $i_2$ 
as in \thetag{\nmb|{15}} we will try to construct a diffeomorphism $f:M\to M$ 
with $i_2\o f=i_1$. If we meet obstacles to the construction this will
give enough control on the situation to separate $i_1$ from $i_2$.

Choose $\la_0\in C^1_0$; so  
$\on{int} M^1_{\la_0}\ne\emptyset$. Then 
$i_1:\on{int} M^1_{\la_0}\to L_{\be_1(\la_0)}$ is an 
embedding, where $\be_1:C^1\to B$ is the mapping satisfying 
$i_1(M^1_\la)\subseteq L_{\be_1(\la)}$ for all $\la\in C^1$.

We choose $\mu_0\in \be_2\i\be_1(\la_0)\subset C^2_0$ such that 
$f:=(i_2|_{\on{int} M^2_{\mu_0}})\i\o 
     i_1|_{\on{int} M^1_{\la_0}}$
is a diffeomorphism $\on{int} M^1_{\la_0}\to \on{int} M^2_{\mu_0}$; this follows from \thetag{iii}.
Note that $f$ is uniquely determined by the choice of $\mu_0$, if it 
exists, by claim \thetag{\nmb|{8}}. So we will repeat the following 
construction for every $\mu_0\in \be_2\i\be_1(\la_0)\subset C^2_0$.

Now we try to extend $f$. 
We choose $\la_1\in C^1_0$ such that  
$\overline M^1_{\la_0}\cap\overline M^1_{\la_1}\ne \emptyset$.

\noindent
{\bf Case a.} Only $\la_1=\la_0$ is possible. So $M^1_{\la_0}$ 
is dense in $M$ 
since $M$ is connected and we may extend $f$ by continuity to a 
diffeomorphism $f:M\to M$ with $i_2\o f=i_1$.

\noindent
{\bf Case b.} We can find $\la_1\ne\la_0$. We choose 
$x\in\overline M^1_{\la_0}\cap\overline M^1_{\la_1}$ and a sequence 
$(x_n)$ in $M^1_{\la_0}$ with $x_n\to x$. Then we have a sequence
$(f(x_n))$ in $M$.

\noindent
{\bf Case ba.} $y:=\lim f(x_n)$ exists in $M$. Then there is 
$\mu_1\in C^2_0$ such that 
$y\in \overline M^2_{\mu_0}\cap\overline M^2_{\mu_1}$.

Let $U^1_{\al_1}$ be an open neighborhood of $x$ in $M$ such that 
$i_1|{ U^1_{\al_1}}$ is an embedding and let similarly 
$U^2_{\al_2}$ be an open neighborhood of $y$ in $M$ such that
$i_2|{ U^2_{\al_2}}$ is an embedding. We consider now the 
set $i_2\i i_1(U^1_{\al_1})$. There are two cases possible.

\noindent
{\bf Case baa.} The set $i_2\i i_1(U^1_{\al_1})$ is a neighborhood of 
$y$. Then we extend $f$ to $i_1\i(i_1(U^1_{\al_1})\cap i_2(U^2_{\al_2}))$ 
by $i_2\i\o i_1$. Then $f$ is defined on some open subset of 
$\on{int} M^1_{\la_1}$ and by the situation chosen in \thetag{\nmb|{15}} and by \thetag{iii}, the 
diffeomorphism $f$ extends 
to the whole of $\on{int} M^1_{\la_1}$.

\noindent
{\bf Case bab.} The set $i_2\i i_1(U^1_{\al_1})$ is not a neighborhood of 
$y$. This is a definite obstruction to the extension of $f$.

\noindent
{\bf Case bb.} The sequence $(x_n)$ has no limit in $M$. This is a 
definite obstruction to the extension of $f$.

If we meet an obstruction we stop and try another $\mu_0$. If for all 
admissible $\mu_0$ we meet obstructions we stop and remember the data. 
If we do not 
meet an obstruction we repeat the construction with some obvious 
changes.

\medskip\noindent
{\bf\thetag{\nmb:{18}} Claim.}
\emph{
The construction of \thetag{\nmb|{17}} in the 
setting of \thetag{\nmb|{15}} either produces a diffeomorphism $f:M\to M$ 
with $i_2\o f=i_1$ or we may separate $i_1$ and $i_2$ by open sets in 
$\on{Imm}_{\text{\rm prop}}(M,N)$ which are saturated with respect to the action of 
$\Diff(M)$
}

If for some $\mu_0$ we do not meet any obstruction in the 
construction \thetag{\nmb|{17}}, the resulting $f$ is defined on the whole of 
$M$ and it is a 
continuous mapping $M\to M$ with $i_2\o f=i_1$. Since $i_1$ and $i_2$ 
are locally embeddings, $f$ is smooth and of maximal rank. Since $i_1$ 
and $i_2$ are proper, $f$ is proper. So the image of $f$ is open and 
closed and since $M$ is connected, $f$ is a surjective local 
diffeomorphism, thus a covering mapping $M\to M$. But since 
$\de_1=\de_2$ the mapping $f$ must be a 1-fold covering, i.e., a 
diffeomorphism.

If for all $\mu_0\in \be_2\i\be_1(\la_0)\subset C^2_0$ we meet 
obstructions we choose small mutually distinct open neighborhoods 
$V^1_\la$ of the sets $i_1(M^1_\la)$. We consider the Whitney 
$C^0$-open neighborhood $\mathcal V_1$ of $i_1$  consisting of all 
immersions $j_1$ with $j_1(M^1_\la)\subset V^1_\la$ for all $\la$.
Let $\mathcal V_2$ be a similar neighborhood of $i_2$.

We claim that $\mathcal V_1 \o \Diff(M)$ and $\mathcal V_2\o\Diff(M)$ are 
disjoint. For that it suffices to show that for any $j_1\in\mathcal V_1$ 
and $j_2\in \mathcal V_2$ there does not exist a diffeomorphism 
$f\in\Diff(M)$ with $j_2\o f=j_1$. For that to be possible the 
immersions $j_1$ and $j_2$ must have the same image $L$ and the same 
functions $\de(j_1)$, $\de(j_2):L\to \mathbb N$. But now the 
combinatorial relations of the slightly distinct new sets $M^1_\la$, 
$L_\be$, and $M^2_\mu$ are contained in the old ones, so any try to 
construct such a diffeomorphism $f$ starting from the same $\la_0$ 
meets the same obstructions.

Statements \thetag{\nmb|{2}} and \thetag{\nmb|{3}} of the theorem are now proved.

\medskip\noindent
{\bf\thetag{\nmb:{19}} Claim.}
\emph{For a non-free immersion $i\in\on{Imm}(M,N)$, 
the nontrivial isotropy subgroup 
$\Diff(M)_i = \{f\in \Diff(M): i\o f= i\}$
acts properly discontinuously on $M$, so the projection 
$q_1:M\to M_1:= M/\Diff(M)_i$ is a covering map onto a smooth manifold on $M_1$. 
There is an 
immersion $i_1:M_1\to N$ with $i=i_1\o q_1$. In particular $\Diff(M)_i$ 
is countable, and is finite if $M$ is compact.
}

We have to show that for each $x\in M$ there is an open neighborhood 
$U$ such that $f(U)\cap U=\emptyset$ for 
$f\in \Diff(M)_i\setminus \{Id\}$. We consider the setup \thetag{\nmb|{7}} for 
$i$. By the proof of \thetag{\nmb|{8}} we have 
$f(U^i_\al)\cap U^i_\al=\{x\in U^i_\al: f(x)=x\}$ for any 
$f\in\Diff(M)_i$. If $f$ has a fixed point then  $f=\on{Id}$, by \thetag{\nmb|{8}}, 
so $f(U^i_\al)\cap U^i_\al=\emptyset$ for all 
$f\in\Diff(M)_i\setminus\{Id\}$. The rest is clear.

The factorized immersion $i_1$ is in general not a free immersion. 
The following is an example for that: Let 
$M_0 \East{\al}{} M_1 \East{\be}{} M_2 \East{\ga}{} M_3$
be a sequence of covering maps with fundamental groups 
$1\to G_1 \to G_2 \to G_3$. Then the group of deck transformations of 
$\ga$ is given by $\mathcal N_{G_3}(G_2)/G_2$, the normalizer of $G_2$ in 
$G_3$, and the group of deck transformations of 
$\ga\o\be$ is $\mathcal N_{G_3}(G_1)/G_1$. We can 
easily arrange that $\mathcal N_{G_3}(G_2)\nsubseteq\mathcal N_{G_3}(G_1)$, 
then $\ga$ admits deck transformations which do not lift to $M_1$. 
Then we thicken all spaces to manifolds, so that $\ga\o\be$ plays the 
role of the immersion $i$. 

\medskip\noindent
{\bf\thetag{\nmb:{20}} Claim.}
\emph{
Let $i\in\on{Imm}(M,N)$ be an immersion 
which is not free. Then there is a submersive covering map $q_2:M\to M_2$ 
such that $i$ factors to an immersion 
$i_2:M_2\to N$ which is free.
}

Let $q_0:M_0\to M$ be the universal covering of $M$ and 
consider the immersion $i_0=i\o q_0:M_0\to N$ and 
its isotropy group $\Diff(M_0)_{i_0}$. By \thetag{\nmb|{19}} it 
acts properly discontinuously on $M_0$ and we have a submersive 
covering $q_{02}:M_0\to M_2$ and an immersion 
$i_2:M_2\to N$ with 
$i_2\o q_{02}=i_0=i\o q_0$. By comparing the respective groups of 
deck transformations it is easily seen that 
$q_{02}:M_0\to M_2$ factors over
$q_1\o q_0:M_0\to M\to M_1$ 
to a covering $q_{12}:M_1\to M_2$. The mapping 
$q_2:=q_{12}\o q_1:M\to M_2$ is the looked for covering: If 
$f\in \Diff(M_2)$ fixes $i_2$, it lifts to a 
diffeomorphism $f_0\in \Diff(M_0)$ which fixes $i_0$, 
so $f_0\in\Diff(M_0)_{i_0}$, so $f=\on{Id}$.

Statement \thetag{\nmb|{4}} of the theorem follows from \thetag{\nmb|{19}} and \thetag{\nmb|{20}}.

\medskip\noindent
{\bf\thetag{\nmb:{21}} Convention.}
In order to avoid complications 
we assume now that $M$ is such a manifold that 
\begin{itemize}
\item For any covering $M\to M_1$, any diffeomorphism $M_1\to M_1$ 
     admits a lift $M\to M$.
\end{itemize}
If $M$ is simply connected, this condition is satisfied. 
Also for $M=S^1$ the condition is easily seen to be 
valid. So what follows is applicable to loop spaces.

Condition \thetag{\nmb|{21}} implies that in the proof of 
claim \thetag{\nmb|{20}} we have $M_1=M_2$.

\medskip\noindent
{\bf\thetag{\nmb:{22}} Description of a neighborhood of a singular 
orbit.} 
Let $M$ be a manifold satisfying \thetag{\nmb|{21}}.
In the situation of \thetag{\nmb|{19}} we consider the normal bundles 
$p_i:\mathcal N(i)\to M$ and 
$p_{i_1}:\mathcal N(i_1)\to M_1$. Then the 
covering map $q_1:M\to M_1$ lifts uniquely to a vector bundle 
homomorphism $\mathcal N(q_1):\mathcal N(i)\to \mathcal N(i_1)$ which is 
also a covering map, such that $\ta^{i_1}\o\mathcal N(q_1)=\ta^i$.

We have $M_1= M/\Diff(M)_i$ and the group $\Diff(M)_i$ acts also as 
the group of deck transformations  of the covering 
$\mathcal N(q_1):\mathcal N(i)\to \mathcal N(i_1)$ by 
$\Diff(M)_i\ni f\mapsto \mathcal N(f)$, where 
$$\xymatrix@C+0.8cm@R-0.3cm{
\mathcal N(i) \ar[r]_{\mathcal N(f)} \ar[d] &  \mathcal N(i) \ar[d]\\
M          \ar[r]^{f}   &       M
}$$
is a vector bundle isomorphism for each $f\in \Diff(M)_i$; see the end of \thetag{\nmb|{7}}.
If we equip $\mathcal N(i)$ and $\mathcal N(i_1)$ with the fiber Riemann 
metrics induced from the fixed Riemannian metric $g$ on $N$, the 
mappings $\mathcal N(q_1)$ and all $\mathcal N(f)$ are fiberwise linear 
isometries.

Let us now consider the right action of $\Diff(M)_i$ on the space of 
sections $\Ga_c(M\gets \mathcal N(i))$ given by 
$f^*s:=\mathcal N(f)\i\o s\o f$. 

From the proof of claim \thetag{\nmb|{10}} we recall now the sets
$$\xymatrix@C+0.8cm@R-0.3cm{
C^\infty(M,\mathcal N(i)) & \;\mathcal V(i) \ar@{_{(}->}[l] & \ar[l]_{\qquad\qquad\ph_i} \mathcal U(i)\\
\Ga_c(M\gets \mathcal N(i)) \ar@{_{(}->}[u] & \;\Ga_c(M\gets U^i) \ar@{_{(}->}[l] & 
\ar[l]_{\qquad\qquad\ph_i} \mathcal Q(i) \ar@{_{(}->}[u]
}$$
Both mappings $\ph_i$ are diffeomorphisms. But since the action of $\Diff(M)$ on $i$ is 
not free we cannot extend the splitting submanifold $\mathcal Q(i)$ to an 
orbit cylinder as we did in the proof of claim \thetag{\nmb|{10}}. 
$\mathcal Q(i)$ is a smooth transversal for the orbit though $i$.

For any $f\in \Diff(M)$ and 
$s\in \Ga_c(M\gets U^i)\subset \Ga_c(M\gets \mathcal N(i))$ we have 
$$\ph_i\i(f^*s)= \ta^i_*(f^*s) = \ta^i_*(s)\o f.$$
So the space $q_1^*\Ga_c(M\gets \mathcal N(i_1))$ of all sections of 
$\mathcal N(i)\to M$
which factor to sections of $\mathcal N(i_1)\to M_1$, is exactly the 
space of all fixed points of the action of $\Diff(M)_i$ on 
$\Ga_c(M\gets \mathcal N(i))$; and they are mapped by $\ta^i_*=\ph_i\i$ to such
immersions in $\mathcal Q(i)$ which have again $\Diff(M)_i$ as isotropy 
group.

If $s\in \Ga_c(M\gets U^i)\subset \Ga_c(M\gets \mathcal N(i))$ is an 
arbitrary  section, the orbit through $\ta^i_*(s)\in \mathcal Q(i)$ hits 
the transversal $\mathcal Q(i)$ again in the points $\ta^i_*(f^*s)$ for 
$f\in\Diff(M)_i$.

Statement \thetag{\nmb|{5}} of the theorem is now proved.

\medskip\noindent
{\bf\thetag{\nmb:{23}} The orbit structure.} 
\emph{We have the following description of the orbit structure near $i$ 
in $\on{Imm}(M,N)$: For fixed $f\in \Diff(M)_i$ the set of fixed points
$\operatorname{Fix}(f):=\{j\in\mathcal Q(i):j\o f=j\}$ is called a 
\emph{generalized wall}. The union of all generalized walls is called the 
\emph{diagram} $\mathcal D(i)$ of $i$. A connected component of the complement
$\mathcal Q(i)\setminus \mathcal D(i)$ is called a \emph{generalized Weyl
chamber}. 
The group $\Diff(M)_i$ maps walls to walls and chambers to chambers.
The immersion $i$ lies in every wall. 
We shall see shortly that there is only one chamber and that the 
situation is rather distinct from that of reflection groups.}

If we view the diagram in the space
$\Ga_c(M\gets U^i)\subset \Ga_c(M\gets \mathcal N(i))$ which is diffeomorphic 
to $\mathcal Q(i)$, then it consists of
traces of closed linear subspaces, because
the action of $\Diff(M)_i$ on $\Ga_c(M\gets \mathcal N(i))$ consists of
linear isometries in the following way. Let us tensor the vector 
bundle $\mathcal N(i)\to M$ with the natural line bundle of half 
densities on $M$, and let us remember one positive half density to
fix an isomorphism with the original bundle. 
Then $\Diff(M)_i$ still acts on this new bundle $\mathcal N_{1/2}(i)\to M$ 
and the pullback
action on sections with compact support is isometric for the inner product
$$\langle s_1,s_2 \rangle := \int_M g(s_1,s_2).$$
We now extend the walls and chambers from 
$$\mathcal Q(i)=\Ga_c(M\gets U^i)\subset \Ga_c(M\gets \mathcal N(i))$$ 
 to the whole space $\Ga_c(M\gets \mathcal N(i))= \Ga_c(M\gets \mathcal N_{1/2}(i))$; 
 recall from \thetag{\nmb|{22}} that $\Diff(M)_i$ acts on the whole space.  

\medskip\noindent
{\bf\thetag{\nmb:{24}} Claim.}
\emph{
Each wall in $\Ga_c(M\gets \mathcal N_{1/2}(i))$ 
is a closed linear subspace of infinite codimension. Since there are 
at most countably many walls, there is only one chamber. 
}

From the proof of claim \thetag{\nmb|{19}} we know that 
$f(U^i_\al)\cap U^i_\al=\emptyset$ for all $f\in\Diff(M)_i$ and all 
sets $U^i_\al$ from the setup \thetag{\nmb|{7}}. Take a section $s$ in the 
wall of fixed points of $f$.
Choose a section $s_\al$ with support in some $U^i_\al$ and let 
the section $s$ be defined by 
$s|_{U^i_\al}=s_\al|_{U^i_\al}$, $s|_{f\i(U^i_\al)}=-f^*s_\al$, 0 elsewhere.
Then obviously $\langle s,s'\rangle=0$ for all $s'$ in the wall of 
$f$. But this construction furnishes an infinite dimensional space 
contained in the orthogonal complement of the wall of $f$.

\medskip\noindent
{\bf\thetag{\nmb:{25}} The action of $\Diff(M)$ on $\on{Imm}(M,N)\x M$. 
Proof of \thetag{\nmb|{6}}.}

Here we will consider the right action $(i,x).f= (i\o f, f\i(x))$ 
of $\Diff(M)$ on $\on{Imm}(M,N)\x M$. 
This action is free: If $(i\o f,f\i(x))=(i,x)$ then from claim \thetag{\nmb|{8}} we get $f=Id_M$.

{\bf Claim.} \emph{
Let $(i,x)\in\on{Imm}(M,N)\x M$.
Then there is an open neighborhood $\mathcal W(i,x)$ in $\on{Imm}(M,N)\x M$ 
which is saturated for the $\Diff(M)$-action and which splits 
smoothly as 
$$\mathcal W(i,x)=\mathcal Q(i,x)\x \Diff(M).$$
Here $\mathcal Q(i,x)$ is a smooth splitting submanifold of 
$\on{Imm}(M,N)\x M$, 
diffeomorphic to an open neighborhood of $(0,x)$ in $C^\infty(\mathcal N(i))$.} 

\emph{Let $\pi:\on{Imm}(M,N)\x M\to (\on{Imm}(M,N)\x M)/\Diff(M)=S(M,N)$ be the projection onto the 
orbit space, which we equip with the quotient topology. Then 
$\pi|_{ \mathcal Q(i,x)}:\mathcal Q(i,x)\to \pi(\mathcal Q(i,x))$ is bijective 
onto an open subset of the quotient. If $(i,x)$ runs through 
$\on{Imm}(M,N)\x M$ these mappings define a smooth atlas for 
the quotient space, so that 
$$(\on{Imm}(M,N)\x M, \pi, (\on{Imm}(M,N)\x M)/\Diff(M), \Diff(M))$$
is a smooth principal fiber bundle with structure group $\Diff(M)$.}

\emph{If we restrict to the open subset $\on{Imm}_{\text{\rm prop}}(M,N)\x M$ of 
proper immersions times $M$ then the base space is Hausdorff.
}

By claim \thetag{\nmb|{19}}, the isotropy subgroup 
$\Diff(M)_i=\{f\in\Diff(M):i\o f=i\}$ acts 
properly discontinuously on $M$, so $q_1:M\to M/\Diff(M)_i=:M_1$ is a 
covering. We choose an open neighborhood $W_x$ of $x$ in $M$ such 
that $q_1:W_x\to M_1$ is injective.

Now we adapt the second half of the proof of claim \thetag{\nmb|{10}} and use 
freely all the notation from there.
We consider the open set
\begin{multline*}
\{(h\o f, f\i(y)):h\in \mathcal V(i), y\in W_x, f\in\Diff(M) \} \subset\\
\subset C^\infty(M, U^i)\x M \subset C^\infty(M,\mathcal N(i))\x M.
\end{multline*}
We have a smooth mapping from it into 
$\Ga_c(M\gets U^i)\x W_x\x \Diff(M)$ which is given by 
$(h,y)\mapsto (h\o(p\o h)\i,(p\o h)(y), p\o h)$, 
where $\Ga_c(M\gets U^i)$ is the space 
of sections with compact support of $U^i\to M$.
We now put
$$\mathcal Q(i,x):= \ta^i_*(\Ga_c(M\gets U^i)\cap\mathcal V(i))\x W_x
\subset \on{Imm}(M,N)\x M.$$
Then we have 
\begin{align*}
\mathcal W(i,x):&= \{(h\o f,f(y)): h\in\mathcal U(i), y\in W_x,f\in\Diff(M)\}\\
&\cong \mathcal Q(i,x)\x \Diff(M) \cong
     (\Ga_c(M\gets U^i)\cap\mathcal V(i))\x W_x\x \Diff(M),
\end{align*}
since the action of $\Diff(M)$ is free.
The quotient mapping 
$\pi|_{\mathcal Q(i)}: \mathcal Q(i)\to \on{Imm}_{\text{\rm free}}(M,N)/\Diff(M)$
is bijective onto an open set in the quotient.
We now use
$(\ph_i\x Id_{W_x})\o(\pi|_{\mathcal Q(i,x)})\i:
     \pi(\mathcal Q(i,x))\to \Ga_c(M\gets U^i)\x W_x$
as a chart for the quotient space. In order to investigate the chart 
change let $(j,y)\in \on{Imm}(M,N)\x M$ be such that 
$\pi(\mathcal Q(i,x))\cap\pi(\mathcal Q(j,y))\ne\emptyset$. 
Then there exists $(h,z)\in \mathcal W(i,x)\cap \mathcal Q(j,y)$, so there 
exists a unique $f\in\Diff(M)$ (given by $f=p\o\ph_i(h)$) such 
that $(h\o f\i,f(z))\in\mathcal Q(i,x)$. If we consider 
$(j\o f\i,f(y))$ instead of 
$(j,y)$ and call it again $(j,y)$, we have 
$\mathcal Q(i,x)\cap\mathcal Q(j,y)\ne\emptyset$ and consequently 
$\mathcal U(i)\cap\mathcal U(j)\ne\emptyset$. Now the first component of the chart 
change is smooth by the argument in the end of the proof of 
claim \thetag{\nmb|{10}}, and the second component is just $Id_{W_x\cap W_y}$.

The result about Hausdorff follows from claim \thetag{\nmb|{11}}.
The fibers over $\on{Imm}(M,N)/\Diff(M)$ can be read off the following diagram:
$$\xymatrix@C+0.8cm@R=1.2em{
M \ar[r]^{\on{ins}_i} \ar[d] & \on{Imm}(M,N)\x M \ar[r]^{\on{pr}_1} \ar[d]^{\pi} & \on{Imm}(M,N) \ar[d]^{\pi}
\\
\dfrac{M}{\Diff(M)_i} \ar[r] & \dfrac{\on{Imm}(M,N)\x M}{\Diff(M)} \ar[r] & \dfrac{\on{Imm}(M,N)}{\Diff(M)}
}$$
This finishes the proof of Theorem \nmb!{7.3}.
\end{proof}

\section{Weak Riemannian manifolds}\nmb0{8}

If an infinite dimensional manifold is not modeled on a Hilbert space, then a Riemannian metric 
cannot describe the topology on each tangent space. We have to deal with more complicated 
situations.

\subsection{Manifolds, vector fields, differential forms}\nmb.{8.1}
Let $M$ be a smooth manifold modeled on convenient vector spaces. Tangent vectors to 
$M$ are kinematic ones. 

The reason for this is that eventually we want flows of vector 
fields, and that there are too many derivations in 
infinite dimensions, even on a Hilbert space $H$: Let $\al\in L(H,H)$ be a continuous linear 
functional which vanishes on the subspace of compact operators, thus also on $H\otimes H$. Then the linear functional
$f\mapsto \al(d^2f(0))$ is a derivation at 0 on $C^\infty(H)$, since 
$$\al(d^2(f.g)(0))=\al\big(d^2f(0).g(0)+df(0)\otimes dg(0) + dg(0)\otimes df(0) + f(0).d^2g(0)\big)$$
and $\al$ vanishes on the two middle terms. There are even non-zero derivations which differentiate 3 times, 
see \cite[28.4]{KrieglMichor97}. 

The (kinematic) tangent bundle $TM$ is then a smooth vector bundle as usual. 
Differential forms of degree $k$ are then smooth sections of the bundle 
$L^k_{\text{skew}}(TM;\mathbb R)$ of skew symmetric 
$k$-linear functionals on the tangent bundle, since this is the only version which admits exterior 
derivative, Lie derivatives along vector field, and pullbacks along arbitrary smooth mappings; 
see \cite[33.21]{KrieglMichor97}. The de~Rham cohomology equals singular cohomology with real 
coefficients if the manifold is smoothly paracompact; see \cite{Michor83} and \cite [Section 34]{KrieglMichor97}.
If a vector field admits a flow, then each integral curve is 
uniquely given as a flow line; see \cite[32.14]{KrieglMichor97}.

\subsection{Weak Riemannian manifolds}\nmb.{8.2}
Let  $M$  be  a smooth manifold
modeled on  convenient locally convex vector spaces. A smooth Riemannian metric $g$ on $M$ is called weak if 
$g_x:T_xM\to  T_x^*M$ is only injective for each $x\in M$. 
The  image $g(TM)\subset  T^*M$  is  called  the {\em smooth cotangent
bundle}  associated  to  $g$.  
Then  $g\i$ is the metric on the
smooth  cotangent bundle as well as the morphism $g(TM)\to TM$.
We have a special class of 1-forms
$\Om_g^1(M):=  \Ga(g(TM))$ for which the musical mappings make sense: $\al^\sharp=g\i\al\in\X(M)$ 
and $X^\flat  =  gX$. These 1-forms separate points on $TM$.
The exterior derivative is 
defined by
$d:\Om_g^1(M) \to \Om^2(M)=\Ga(L^2_{\text{skew}}(TM;\mathbb R))$
since the embedding $g(TM)\subset T^*M$ is a smooth fiber linear mapping. 

Existence of the Levi-Civita covariant derivative is equivalent to:
{\em The metric itself admits \emph{symmetric} gradients with
respect to itself.}
Locally this means:
If $M$ is $c^\infty$-open in a
convenient vector space $V_M$. Then:
$$
D_{x,X}g_x(X,Y) = g_x(X,\on{grad}_1g(x)(X,Y)) 
= g_x(\on{grad}_2g(x)(X,X),Y)
$$
where $D_{x,X}$ denote the directional derivative at $x$ in the direction $X$, and where
the mappings $\on{grad}_1g$ and $\on{sym}\on{grad}_2g: M\x V_M\x V_M\to V_M$, given by
$(x,X)\mapsto \on{grad}_{1,2}g(x)(X,X)$, are smooth and quadratic in $X\in V_M$.
The geodesic equation then is (again locally) given by
$$
c_{tt} = \tfrac12 \on{grad}_1g(c)(c_t,c_t) - \on{grad}_2g(c)(c_t,c_t)\,.
$$ 
This formula corresponds to the usual formula for the geodesic flow using Christoffel symbols, expanded out using the first derivatives of the metric tensor.
For the existence of the covariant derivative see \cite[2.4]{Micheli2013}, and for the geodesic equation see \cite[2.1 and 2.4]{Michor2007}; there this is done in a special case, but the method works in the general case without changes. See also \cite[4.2, 4.3, and 4.4]{Bauer2011b} for a derivation in another special case.

\subsection{Weak Riemannian metrics on spaces of immersions}\nmb.{8.3}
For a compact manifold $M$ and a finite dimensional Riemannian manifold $(N,\bar g)$ we can consider the 
following weak Riemannian metrics on the manifold $\Imm(M,N)$ of smooth immersions $M\to N$:
\begin{align*}
G^0_f(h,k) &= \int_M \bar g(h,k) \vol(f^*\bar g) &\quad&\text{the  }L^2\text{-metric,}
\\
G^s_f(h,k) &= \int_M \bar g((1+\De^{f^*\bar g})^s h,k) \vol(f^*\bar g) &\quad&\text{a Sobolev 
metric of order }s,
\\
G^\Ph_f(h,g) &= \int_M \Ph(f) \bar g(h,k) \vol(f^*\bar g)  & &\text{an almost local metric.}
\end{align*}
Here $\vol(f^*\bar g)$ is the volume density on $M$ of the pullback metric $g=f^*\bar g$, 
and $\De^{f^*\bar g}$ is the (Bochner) Laplacian with respect to $g$ and $\bar g$ acting on 
sections of $f^*TN$,
and  $\Ph(f)$ is a positive function of the total volume $\on{Vol}(f^*g)=\int_M \vol(f^*g)$, 
of the scalar curvature $\on{Scal}(f^*\bar g)$, and of the mean curvature $\on{Tr}(S^f)$, $S^f$ 
being the second fundamental form. See \cite{Bauer2011b}, \cite{Bauer2012a} for more information.
All these metrics are invariant for the right action of the reparameterization group $\Diff(M)$, so 
they descend to metrics on shape space $B_i(M,N)$ (off the singularities) such that the 
projection $\Imm(M,N)\to B_i(M,N)$ is a Riemannian submersion of a benign type: the $G$-orthogonal 
component to the tangent space to the $\Diff(M)$-orbit consists always of smooth vector fields. 
So there is no need to use the notion of robust weak Riemannian metrics discussed below. 

\begin{theorem}\nmb.{8.4}
The Riemannian metrics on $\Imm(M,N)$ defined in \nmb!{8.3}  have the following properties:
\begin{enumerate}
\item[\thetag{1}]Geodesic distance on $\Imm(M,N)$, defined as the infimum of path-lengths of smooth isotopies between two 
immersions, vanishes for the $L^2$-metric $G^0$.

\item[\thetag{2}] Geodesic distance is positive on $B_i(M,N)$ for the almost local metric $G^\Ph$ if 
$\Ph(f)\ge 1+ A\on{Tr}(S^F)$, or if $\Ph(f)\ge A\on{Vol}(f^*\bar g)$, for some $A>0$.

\item[\thetag{3}] Geodesic distance is positive on $B_i(M,N)$ for the Sobolev metric $G^s$ if $s\ge 1$.

\item[\thetag{4}] The geodesic equation is locally well-posed on $\Imm(M,N)$ for the Sobolev metric $G^s$ if 
$s\ge 1$, and globally well-posed (and thus geodesically complete) on $\Imm(S^1,\mathbb R^n)$, if 
$s\ge 2$.
\end{enumerate}
\end{theorem}

\thetag{1} is due to  \cite{Michor2006c} for $B_i(S^1,\mathbb R^2)$, to \cite{Michor2005} for $B_i(M,N)$ and for 
$\Diff(M)$, which combines to the result for $\Imm(M,N)$ as noted in \cite{Bauer2012c}. \thetag{2} is proved 
in \cite{Bauer2012a}. For \thetag{3} see \cite{Bauer2011b}.
\thetag{4} is due to \cite{Bruveris2014} and \cite{Bruveris2014b}.

\subsection{Analysis tools on regular Lie groups and on $\Diff(M)$ for a Whitney manifold germ}
\nmb.{8.5}
Let $G$ be a regular convenient Lie group, with Lie algebra $\mathfrak g$. We also consider a  Fr\"olicher group $G=\Diff(M)$ for a Whitney manifold germ $M\subset \tilde M$ with Lie algebra $\mathfrak g=\X_{c,\p}(M)$, with the negative of the usual Lie bracket, as described in \nmb!{6.3} -- \nmb!{6.6}.
 
Let $\mu:G\x G\to G$ be the group multiplication, $\mu_x$ the left
translation and $\mu^y$ the right translation, 
$\mu_x(y)=\mu^y(x)=xy=\mu(x,y)$. The adjoint action $\on{Ad}:G\to GL(\mathfrak g)$  is given by 
$\on{Ad}(g)X = T(\mu^{g\i}).T(\mu_g)X$. 
Let $L,R:\mathfrak g\to \X(G)$ be the left and right invariant vector field mappings, given by 
$L_X(g)=T_e(\mu_g).X$ and $R_X=T_e(\mu^g).X$, respectively.
They are related by $L_X(g)=R_{\on{Ad}(g)X}(g)$.
Their flows are given by 
\begin{align*}
\on{Fl}^{L_X}_t(g)&= g.\exp(tX)=\mu^{\exp(tX)}(g),
\\
\on{Fl}^{R_X}_t(g)&= \exp(tX).g=\mu_{\exp(tX)}(g).
\end{align*}
The right
Maurer--Cartan form $\ka=\ka^r\in\Om^1(G,\mathfrak g)$ is given by  
$\ka_x(\xi):=T_x(\mu^{x\i})\cdot \xi$. 
\emph{It satisfies the left Maurer--Cartan equation 
$d\ka^r-\tfrac12[\ka^r,\ka^r]^\wedge_{\mathfrak g}=0$}, 
where
$[\quad,\quad]^\wedge$ denotes the wedge product of $\mathfrak g$-valued forms on
$G$ induced by the Lie bracket. Note that
$\tfrac12[\ka^r,\ka^r]^\wedge (\xi,\et) = [\ka^r(\xi),\ka^r(\et)]$.
\\
Namely, evaluate $d\ka^r$ on right invariant vector fields $R_X,R_Y$ for $X,Y\in{\mathfrak g}$.
\begin{multline*}
(d\ka^r)(R_X,R_Y) = R_X(\ka^r(R_Y)) - R_Y(\ka^r(R_X)) - \ka^r([R_X,R_Y])
\\
= R_X(Y) - R_Y(X) + [X,Y] = 0-0 +[\ka^r(R_X),\ka^r(R_Y)].
\end{multline*}
The left
Maurer--Cartan form $\ka^l\in\Om^1(G,\mathfrak g)$ is given by  
$\ka^l_x(\xi):=T_x(\mu_{x\i})\cdot \xi$. 
\\
\emph{The left Maurer--Cartan form $\ka^l$ satisfies the right Maurer--Cartan equation 
  $d\ka^l+\tfrac12[\ka^l,\ka^l]^\wedge_{\mathfrak g}=0$}.

The (exterior) derivative of the function $\on{Ad}:G\to GL(\mathfrak g)$ satisfies
\begin{displaymath}
d\on{Ad} = (\on{ad}\o \ka^r).\on{Ad} = \on{Ad}.(\on{ad}\o\ka^l) 
\end{displaymath}
since we have
\begin{align*}
d\on{Ad}(T\mu^g.X) &= \p_t|_0 \on{Ad}(\exp(tX).g) = \p_t|_0 \on{Ad}(\exp(tX)).\on{Ad}(g)
\\&
= \on{ad}(\ka^r(T\mu^g.X)).\on{Ad}(g)\,,
\\
d\on{Ad}(T\mu_g.X) &= \p_t|_0 \on{Ad}(g.\exp(tX))
= \on{Ad}(g).\on{ad}(\ka^l(T\mu_g.X))\,.
\end{align*}

\subsection{\!Right invariant weak Riemannian metrics on regular Lie groups and on $\Diff(M)$ for a Whitney manifold germ}
We continue under the assumptions of \nmb!{8.5},
\nmb.{8.6}
Let $\ga=\mathfrak g\x\mathfrak g\to\mathbb R$ be a positive-definite bounded (weak) inner product. Then  
\begin{equation*}
\ga_x(\xi,\et)=\ga\big( T(\mu^{x\i})\cdot\xi,\, T(\mu^{x\i})\cdot\et\big) =  
     \ga\big(\ka(\xi),\,\ka(\et)\big)
\end{equation*}
is a right invariant (weak) Riemannian metric on $G$
and any (weak) right invariant bounded Riemannian metric is of this form, for
suitable  $\ga$.
Denote by $\check\ga:\mathfrak g\to \mathfrak g^*$ the mapping induced by $\ga$, 
from the 
Lie algebra into its dual (of bounded linear functionals) 
and by 
$\langle \al, X \rangle_{\mathfrak g}$ the duality evaluation between $\al\in\mathfrak g^*$ and 
$X\in \mathfrak g$.
 
Let $g:[a,b]\to G$ be a smooth curve.  
The velocity field of $g$, viewed in the right trivializations, 
coincides  with the right logarithmic derivative
\begin{displaymath}
\de^r(g):=T(\mu^{g\i})\cdot \partial_t g =  
\ka(\partial_t g) = (g^*\ka)(\partial_t).
\end{displaymath}
The energy of the curve $g(t)$ is given by  
\begin{equation*}
E(g) = \frac12\int_a^b\ga_g(g',g')dt = \frac12\int_a^b 
     \ga\big( (g^*\ka)(\partial_t),(g^*\ka)(\partial_t)\big)\, dt. 
\end{equation*}
For a variation $g(s,t)$ with fixed endpoints we then use that
$$d(g^*\ka)(\p_t,\p_s)=\p_t(g^*\ka(\p_s))-\p_s(g^*\ka(\p_t))-0,$$ partial integration, and the
left Maurer--Cartan equation to obtain
\begin{align*} 
&\partial_sE(g) = \frac12\int_a^b2 
     \ga\big( \partial_s(g^*\ka)(\partial_t),\,
                            (g^*\ka)(\partial_t)\big)\, dt
\\&
= \int_a^b \ga\big( \partial_t(g^*\ka)(\partial_s) - 
         d(g^*\ka)(\partial_t,\partial_s),\,
                (g^*\ka)(\partial_t)\big)\,dt
\\&
= -\int_a^b \ga\big( (g^*\ka)(\partial_s),\,\partial_t(g^*\ka)(\partial_t)\big)\,dt 
\\&\qquad\qquad
  - \int_a^b \ga\big( [(g^*\ka)(\partial_t),(g^*\ka)(\partial_s)],\, (g^*\ka)(\partial_t)\big)\, dt
\\&
= -\int_a^b \big\langle\check\ga(\partial_t(g^*\ka)(\partial_t)),\,
  (g^*\ka)(\partial_s)\big\rangle_{\mathfrak g} \, dt
\\&\qquad\qquad
  - \int_a^b \big\langle  \check\ga((g^*\ka)(\partial_t)),\, 
  \on{ad}_{(g^*\ka)(\partial_t)}(g^*\ka)(\partial_s)\big\rangle_{\mathfrak g} \, dt
\\&
= -\int_a^b 
\big\langle\check\ga(\partial_t(g^*\ka)(\partial_t)) + (\on{ad}_{(g^*\ka)(\partial_t)})^{*}\check\ga((g^*\ka)(\partial_t)),\, 
  (g^*\ka)(\partial_s)\big\rangle_{\mathfrak g} \, dt.
\end{align*}
Thus the curve $g(0,t)$ is critical for the energy if and only if
$$
\check\ga(\partial_t(g^*\ka)(\partial_t)) + 
(\on{ad}_{(g^*\ka)(\partial_t)})^{*}\check\ga((g^*\ka)(\partial_t)) = 0.
$$
In terms of the right logarithmic derivative $u:[a,b]\to \mathfrak g$ of  
$g:[a,b]\to G$, given by  
$u(t):= g^*\ka(\partial_t) = T_{g(t)}(\mu^{g(t)\i})\cdot g'(t)$, 
the {\em geodesic equation} has the expression
\begin{equation*}
\boxed{\quad
\p_t u = - \,\check\ga\i\on{ad}(u)^{*}\;\check\ga(u)\quad} 
\end{equation*}
Thus the geodesic equation exists in general if and only if 
$\on{ad}(X)^{*}\check\ga(X)$ is in the image of $\check\ga:\mathfrak g\to\mathfrak g^*$, i.e.,
\begin{equation*}
\on{ad}(X)^{*}\check\ga(X) \in \check\ga(\mathfrak g)
\end{equation*}
 for every 
$X\in\mathfrak X$; this leads to the existence of the  
Christoffel symbols. Arnold  \cite{Arnold66} asked for the more restrictive condition  
$\on{ad}(X)^{*}\check\ga(Y) \in \check \ga(\mathfrak g)$ for all $X,Y\in\mathfrak g$.
The geodesic equation for the \emph{momentum} $p:=\ga(u)$ is
$$
p_t = - \on{ad}(\check\ga\i(p))^*p.
$$
There are situations, see theorem \nmb!{8.11} or \cite{BBM14b}, where 
only the more general condition is satisfied, but where
the usual 
transpose $\on{ad}^\top(X)$ of $\on{ad}(X)$, 
\begin{equation*}
\on{ad}^\top(X) := \check\ga\i\o\on{ad}_X^*\o \check\ga
\end{equation*}
does not exist for all $X$. 

We describe now the {\em covariant derivative} and the {\em curvature}.
The right trivialization $(\pi_G,\ka^r):TG\to G\x{\mathfrak g}$  
induces the isomorphism $R:C^\infty(G,{\mathfrak g})\to \X(G)$, given by  
$R(X)(x):= R_X(x):=T_e(\mu^x)\cdot X(x)$, for $X\in C^\infty(G,{\mathfrak g})$ and
$x\in G$. Here $\X(G):=\Ga(TG)$ denotes the Lie algebra of all vector 
fields. For the Lie bracket and the Riemannian metric we have 
\begin{align*} 
[R_X,R_Y] &= R(-[X,Y]_{\mathfrak g} + dY\cdot R_X - dX\cdot R_Y),
\\
R\i[R_X,R_Y] &= -[X,Y]_{\mathfrak g} + R_X(Y) - R_Y(X),
\\
\ga_x(R_X(x),R_Y(x)) &= \ga( X(x),Y(x))\,,\, x\in G. 
\end{align*}
In what follows, we shall perform all computations in $C^\infty(G,{\mathfrak g})$ instead of 
$\X(G)$. In particular, we shall use the convention
\begin{displaymath}
\nabla_XY := R\i(\nabla_{R_X}R_Y)\quad\text{ for }X,Y\in C^\infty(G,{\mathfrak g})
\end{displaymath}
to express the Levi-Civita covariant derivative. 
 
\begin{lemma}\nmb.{8.7}
{\rm \cite[3.3]{BBM14b}}
Assume that for all $\xi\in{\mathfrak g}$ the element $\on{ad}(\xi)^*\check\ga(\xi)\in\mathfrak g^*$ is in the image of
$\check\ga:\mathfrak g\to\mathfrak g^*$   
and that  
$\xi\mapsto \check\ga\i\on{ad}(\xi)^*\check\ga(\xi)$ is bounded  quadratic (or, equivalently, smooth). 
Then the Levi-Civita covariant derivative of the metric $\ga$ 
exists and is given for any $X,Y \in C^\infty(G,{\mathfrak g})$  in
terms of the isomorphism $R$ by
\begin{equation*}
\nabla_XY= dY.R_X + \rh(X)Y - \frac12\on{ad}(X)Y,
\end{equation*}
where 
\[
\rh(\xi)\et = \tfrac14\check\ga\i\big(\on{ad}_{\xi+\et}^*\check\ga(\xi+\et) - \on{ad}_{\xi-\et}^*\check\ga(\xi-\et)\big) = \tfrac12\check\ga\i\big(\on{ad}_\xi^*\check\ga(\et) + \on{ad}_\et^*\check\ga(\xi)\big)
\]
is the polarized version. The mapping 
$\rh:{\mathfrak g}\to L({\mathfrak g},{\mathfrak g})$ is bounded, and we have $\rh(\xi)\et=\rh(\et)\xi$.  
We also have
\begin{gather*}
\ga\big(\rh(\xi)\et,\ze\big) = \frac12\ga(\xi,\on{ad}(\et)\ze) + \frac12\ga(\et,\on{ad}(\xi)\ze),
\\
\ga(\rh(\xi)\et,\ze) + \ga(\rh(\et)\ze,\xi) + \ga(\rh(\ze)\xi,\xi) = 0.
\end{gather*}
\end{lemma} 


For $X,Y\in C^\infty(G,{\mathfrak g})$ we have
$$ 
[R_X,\on{ad}(Y)] = \on{ad}(R_X(Y))\quad\text{  and }\quad  
[R_X,\rh(Y)] = \rh(R_X(Y)).
$$
The {\em Riemannian curvature} is then computed as follows:  
\begin{align*} 
&\mathcal{R}(X,Y) =  
[\nabla_X,\nabla_Y]-\nabla_{-[X,Y]_{\mathfrak g}+R_X(Y)-R_Y(X)}
\\&
= [R_X+\rh_X-\tfrac12\on{ad}_X, 
     R_Y+\rh_Y-\tfrac12\on{ad}_Y]
\\&\quad
- R({-[X,Y]_{\mathfrak g} + R_X(Y) - R_Y(X)}) 
-\rh({-[X,Y]_{\mathfrak g} + R_X(Y)\! - R_Y(X)})
\\&\quad
+\frac12\on{ad}(-[X,Y]_{\mathfrak g} + R_X(Y)\! - R_Y(X)) 
\\&
= [\rh_X,\rh_Y] +\rh_{[X,Y]_{\mathfrak g}} -\frac12[\rh_X,\on{ad}_Y] +\frac12[\rh_Y,\on{ad}_X] 
-\frac14\on{ad}_{[X,Y]_{\mathfrak g}}
\end{align*}
which is visibly a tensor field.

For the numerator of the sectional curvature we obtain
\begin{align*}
\ga\big(\mathcal R(&X,Y)X,Y\big) = \ga(\rh_X\rh_YX,Y) - \ga(\rh_Y\rh_XX,Y) + \ga(\rh_{[X,Y]}X,Y)
\\&\quad
-\frac12\ga(\rh_X[Y,X],Y) + \frac12\ga([Y,\rh_XX],Y) 
\\&\quad
+0 - \frac12\ga([X,\rh_YX],Y) -\frac14\ga([[X,Y],X],Y)
\\&
= \ga(\rh_XX,\rh_YY) - \|\rh_XY\|_\ga^2 + \frac34\|[X,Y]\|_\ga^2  
\\&\quad
-\frac12\ga(X,[Y,[X,Y]]) + \frac12\ga(Y,[X,[X,Y]])
\\&
= \ga(\rh_XX,\rh_YY) - \|\rh_XY\|_\ga^2 + \frac34\|[X,Y]\|_\ga^2  
\\&\quad
-\ga(\rh_XY,[X,Y]]) + \ga(Y,[X,[X,Y]]).
\end{align*}
If the adjoint $\on{ad}(X)^\top:\mathfrak g\to \mathfrak g$ exists, this is easily seen to coincide with 
Arnold's original formula \cite{Arnold66},
\begin{align*}
\ga(\mathcal R(X,Y)X,Y) =& 
- \frac14\|\on{ad}(X)^\top Y+\on{ad}(Y)^\top X\|^2_\ga
+ \ga(\on{ad}(X)^\top X,\on{ad}(Y)^\top Y)   
\\&
+ \frac12\ga(\on{ad}(X)^\top Y-\on{ad}(Y)^\top X,\on{ad}(X)Y)  
+ \frac34\|[X,Y]\|_\ga^2.
\end{align*}

\subsection{Examples of weak right invariant Riemannian metrics on diffeomorphism groups}\nmb.{8.8}
Let $M$ be a finite dimensional manifold. We consider the following regular Lie groups:
$\on{Diff}(M)$, the  group of all diffeomorphisms of $M$ if $M$ is compact. 
$\on{Diff}_c(M)$, the group of diffeomorphisms with compact support, if $M$ is not compact.
If $M=\mathbb R^n$, we also may consider one of the following:
$\on{Diff}_{\mathcal S}(\mathbb R^n)$, the group of all diffeomorphisms which fall rapidly to the identity.
$\on{Diff}_{W^{\infty,p}}(\mathbb R^n)$, the group of all diffeomorphisms which are modeled on the space 
$W^{\infty,p}(\mathbb R^n)^n$, the intersection of all $W^{k,p}$-Sobolev spaces of vector fields.
The last type of groups works also for a \emph{Riemannian manifold of bounded geometry} 
$(M,\bar g)$; see \cite{Eichhorn2007} for Sobolev spaces on them.  
In the following we write $\Diff_{\mathcal A}(M)$ for any of these groups.
The Lie algebras are the spaces $\X_{\mathcal A}(M)$ of vector fields,
where $\mathcal A\in \{C^\infty_c, \mathcal S, W^{\infty,p}\}$, 
with the negative of the usual bracket as Lie bracket.

Most of the following weak Riemannian metrics also make sense on $\Diff(M)$ for a compact Whitney manifold germ $M\subset \tilde M$, but their behavior has not been investigated. In particular, I do not know how the Laplacian $1+\De^g$ behaves on $\X_\p(M)$ and its Sobolev completions.  

A right invariant weak inner product on $\Diff_{\mathcal A}(M)$ is given by a smooth positive-definite 
inner product $\ga$ on the Lie algebra $\X_{\mathcal A}(M)$ which is described by the \emph{inertia operator} 
$L=\check \ga:\X_{\mathcal A}(M)\to \X_{\mathcal A}(M)'$ and we shall denote its inverse by 
$K=L\i:L(\X_{\mathcal A}(M))\to \X_{\mathcal A}(M)$. Under suitable conditions on $L$ (like 
an elliptic coercive (pseudo) differential operator of high enough order) the operator $K$ turns 
out to be the reproducing kernel of a Hilbert space of vector fields which is contained in the 
space of either $C^1_b$ (bounded $C^1$ with respect to $\bar g$) or $C^2_b$ vector fields. 
See \cite[Chapter 12]{Younes10}, \cite{Micheli2013}, and \cite{MumfordMichor13} for uses of 
the reproducing Hilbert space approach.  
The right invariant metric is then 
defined as in \nmb!{8.5}, where $\langle \;,\; \rangle_{\X_{\mathcal A}(M)}$  is the duality:
$$
G^L_\ph(X\o\ph, Y\o\ph) = G^L_{\on{Id}}(X,Y) = \ga(X,Y)  =  \langle L(X),Y 
\rangle_{\X_{\mathcal A}(M)}.
$$
For example,  the Sobolev metric of order $s$ corresponds  to the inertia operator $L(X) = (1+\De^{\bar g})^s(X).\vol(\bar g)$. 
Examples of metrics are:
\begin{align*}
G^0_{\on{Id}}(X,Y) &= \int_M \bar g(X,Y) \vol(\bar g) &\qquad&\text{the  }L^2\text{  metric,}
\\
G^s_{\on{Id}}(X,Y) &= \int_M \bar g((1+\De^{\bar g})^s X,Y) \vol(\bar g) &\qquad&
\text{a Sobolev metric of order }s,
\\
G^{\dot H^1}_{\on{Id}}(X,Y) &= \int_{\mathbb R} X'.Y' dx = -\int_{\mathbb R} X''Y\,dx &&
\text{where  }X,Y\in \X_{\mathcal A}(\mathbb R).
\end{align*}
As explained in \nmb!{8.8}, the geodesic 
equation on $\Diff_{\mathcal A}(M)$ is given as follows:
Let $\ph:[a,b]\to \Diff_{\mathcal A}(M)$ be a smooth curve.
In terms of its right logarithmic derivative 
$$u:[a,b]\to \X_{\mathcal A}(M),\quad 
u(t):= \ph^*\ka(\partial_t) = \ph'(t)\o \ph(t)\i\,,$$
the geodesic equation is
$$
L(u_t) = L(\p_t u)= - \on{ad}(u)^* L(u).
$$
The {\em condition for the existence of the geodesic equation} is as follows:
$$
X\mapsto K(\on{ad}(X)^*L(X))
$$
is bounded quadratic $\X_{\mathcal A}(M)\to \X_{\mathcal A}(M)$.
Using  \emph{ Lie derivatives}, the computation of $\on{ad}_X^*$ is
especially  simple.  Namely,  for  any  section  
$\om$ of $T^*M \otimes  \on{vol}$   and vector fields 
$\xi,\et \in \X_{\mathcal A}(M)$, we have:
$$ \int_M (\om, [\xi,\et]) = \int_M (\om, \L_\xi(\et)) = 
-\int_M(\L_\xi(\om),\et),$$ 
hence $\on{ad}^*_\xi(\om) = +\L_\xi(\om)$. 
Thus the Hamiltonian version of the geodesic equation on the smooth dual 
$L(\X_{\mathcal A}(M))\subset \Ga_{C^2_b}(T^*M\otimes \on{vol})$ becomes
$$
\p_t\al  = - \on{ad}^*_{K(\al)}\al = - \L_{K(\al)}\al,
$$
or, keeping track of everything,
\begin{align*}
\p_t\ph &= u\o \ph, &\quad& 
\p_t\al = - \L_u\al &&
u = K(\al) = \al^\sharp,&&\al=L(u) = u^\flat.
\end{align*}

\begin{theorem}\nmb.{8.9}
Geodesic distance vanishes on $\Diff_{\mathcal A}(M)$ for any Sobolev metric of order 
$s<\frac12$. If $M=S^1\x C$ with $C$ compact, then geodesic distance vanishes also for $s=\frac12$.
It also vanishes for the $L^2$-metric on the Virasoro group 
$\mathbb R\rtimes \Diff_{\mathcal A}(\mathbb R)$.

Geodesic distance is positive on $\Diff_{\mathcal A}(M)$ for any Sobolev metric of order 
$s\ge1$. If $\dim(M)=1$ then geodesic distance is also positive for $s>\frac12$.
\end{theorem}

This is proved in \cite{Bauer2013b}, \cite{Bauer2013c}, and \cite{Bauer2012c}.
Note that low order Sobolev metrics have geodesic equations corresponding to well-known nonlinear 
PDEs:
On $\Diff(S^1)$ or $\Diff_{\mathcal A}(\mathbb R)$ the $L^2$-geodesic equation is Burgers' equation, 
on the Virasoro group it is the KdV equation, and the (standard) $H^1$-geodesic is (in both cases a 
variant of) the Camassa--Holm equation; see 
\cite[7.2]{Bauer2014} for a more comprehensive overview. All these are completely integrable 
infinite dimensional Hamiltonian systems.

\begin{theorem}\nmb.{8.10}
Let $(M,\bar g)$ be a compact Riemannian manifold.
Then the geodesic equation is locally well-posed on $\Diff_{\mathcal A}(M)$ and the geodesic exponential 
mapping is a local diffeomorphism for a Sobolev metric of 
integer order $s\ge 1$.
For a Sobolev metric of integer order $s>\frac{\dim(M)+3}{2}$ the geodesic equation is even globally 
well-posed, so that $(\Diff_{\mathcal A}(M), G^s)$ is geodesically complete. This is also true for 
non-integer order $s$ if $M=\mathbb R^n$.

For $M=S^1$, the geodesic equation is locally well-posed even for $s\ge \frac12$.
\end{theorem}

For these results see \cite{Bauer2011b}, \cite{Escher2014}, \cite{Escher2014b}, \cite{Bauer2015}. 

\begin{theorem}\nmb.{8.11}
{\rm \cite{BBM14b}}
For $\mathcal A\in\{C^\infty_c, \mathcal S, W^{\infty,1}\}$ 
let 
$$
\mathcal A_1(\mathbb R)=\{f\in C^\infty(R)\,:\, f'\in\mathcal A(\mathbb R)\,,\, 
f(-\infty)=0\}
$$
and let $\Diff_{\mathcal A_1}(\mathbb R)=\{\ph = \on{Id}+f\,:\, 
f\in \mathcal A_1(\mathbb R)\,,\, f'>-1\}$. These are all regular Lie groups. 
The right invariant weak Riemannian metric
$$G^{\dot H^1}_{\on{Id}}(X,Y)= \int_{\mathbb R}X'Y'\,dx$$
is positive definite both on $\Diff_{\mathcal A}(\mathbb R)$ where it does not admit a geodesic 
equation (a non-robust weak Riemannian manifold), and on 
$\Diff_{\mathcal A_1}(\mathbb R)$ where it admits a geodesic equation but not in the stronger sense 
of Arnold. 
On $\on{Diff}_{\mathcal A_1}(\mathbb R)$ the geodesic equation is the Hunter-Saxton equation
$$
(\ph_t)\circ\ph\i=u,  \qquad u_{t} = -u u_x +\frac12 \int_{-\infty}^x   (u_x(z))^2 \,dz\;,
$$
and the induced geodesic distance is positive. 
We define the $R$-map by:
$$ R:
 \Diff_{\mathcal A_1}(\mathbb R)\to 
\mathcal A\big(\mathbb R,\mathbb R_{>-2}\big)\subset\mathcal A(\mathbb R,\mathbb R),
\quad R(\ph) = 2\;\big((\ph')^{1/2}-1\big)\, .
$$
The $R$-map is invertible with inverse
$$R\i :
\mathcal A\big(\mathbb R,\mathbb R_{>-2}\big) \to \Diff_{\mathcal A_1}(\mathbb R),
\quad R\i(\ga)(x) = x+\frac14 \int_{-\infty}^x \ga^2+4\ga\;dx\; .
$$ 
The pullback of the flat $L^2$-metric via $R$ is the $\dot H^1$-metric on $\on{Diff}_\mathcal A(\mathbb R)$, i.e., 
$R^*\langle \cdot,\cdot\rangle_{L^2(dx)} = G^{\dot H^1}$. 
Thus the space $\big(\Diff_{\mathcal A_1}(\mathbb R),\dot H^1\big)$ is a flat space in the sense of Riemannian geometry.
There are explicit formulas for geodesics, geodesic distance, and geodesic splines, even for more 
restrictive spaces $\mathcal A_1$ like Denjoy--Carleman ultradifferentiable function classes.
There are also soliton-like solutions. 
$(\Diff_{\mathcal A_1}(\mathbb R), G^{\dot H^1})$ is geodesically convex, but not geodesically 
complete; the geodesic completion is the smooth semigroup 
$$\on{Mon}_{\mathcal A_1} = \{\ph=\on{Id}+f\,:\, f\in\mathcal A_1(\mathbb R)\,,\, f'\ge -1\}\,.$$
Any geodesic can hit the subgroup 
$\Diff_{\mathcal A}(\mathbb R)\subset \Diff_{\mathcal A_1}(\mathbb R)$ at most twice.
\end{theorem}

\subsection{Trouv\'e groups for reproducing kernel Hilbert spaces}
\nmb.{8.12}
This is the origin of the notion of a Trouv\'e group. It puts the approach of \nmb!{8.1} to \nmb!{8.11} upside down and gets rid of the use of the Lie algebra structure on the space of vector fields. 
If the generating space $\mathcal A$ of vector fields on $\mathbb R^d$ for the Trouv\'e group 
$\mathcal G_{\mathcal A}$ (see \nmb!{6.12}) is a  reproducing kernel Hilbert space $(\mathcal A(\mathbb R^d,\mathbb R^d),\langle\;,\;,\rangle_{\mathcal A})$ contained in $C^1_b$, then 
$$
\on{dist}(\on{Id}, \ph) := \inf\big\{ \int_0^1 \|X(t)\|_{\mathcal A}\, dt\;:\; X\in \mathcal F_{\mathcal A},\;\on{evol}^X = \ph\big\}
$$
defines a metric which makes the Trouv\'e group $\mathcal G_{\mathcal A}$ into a topological group;
see \cite{Trouve95}, \cite{Younes10}. This is widely used for the \emph{Large Deformation Diffeomorphic Metric Matching} (LDDMM) method in image analysis and computational anatomy. 
The most popular reproducing kernel Hilbert space is the one where the kernel is a Gaussian 
$e^{- |x|^2/\si}$. Here the  space $\mathcal A$ is a certain space of entire real analytic functions, and a direct description of the Trouv\'e group is severely lacking.

\section{Robust weak Riemannian manifolds and Riemannian submersions}\nmb0{9}


\subsection{Robust weak Riemannian manifolds}\nmb.{9.1}
Some  constructions may lead to vector fields whose values do not lie in
$T_x M$, but in the Hilbert space completion $\overline{T_x M}$
with respect to the weak inner product $g_x$. 
We  need that
$\bigcup_{x\in M}\overline{T_x M}$ forms a smooth vector bundle
over  $M$.  In a coordinate chart on open
$U \subset M$, $TM|_U$ is a trivial bundle $U \times V$
and all the inner products $g_x, x \in U$ define inner products
on the  same topological vector space $V$. They all should be 
bounded with respect to each other, so that
the completion $\overline{V}$ of $V$ with respect to $g_x$ does
not depend on $x$ and $\bigcup_{x\in U}\overline{T_x M} \cong U
\times    \overline{V}$.   This   means   that   $\bigcup_{x\in
M}\overline{T_x  M}$ forms a smooth vector bundle over $M$ with
trivializations the linear extensions of the trivializations of
the tangent bundle  $TM\to M$. Chart changes should respect this. This is a compatibility 
property between the weak Riemannian metric and some smooth atlas of $M$.

\smallskip\noindent
{\bf Definition}  A convenient weak Riemannian manifold
$(M,g)$ will be called a {\em robust} Riemannian manifold if
\begin{itemize}
\item The Levi-Civita convariant derivative of the metric $g$ exists: 
The symmetric gradients should exist and be smooth.
\item The completions $\overline{T_x M}$ form a smooth vector bundle as above.
\end{itemize}

\begin{theorem}\nmb.{9.2}
If a right invariant weak Riemannian metric on a regular Lie group admits the Levi-Civita 
covariant derivative, then it is already robust. 
\end{theorem}

\begin{proof}
By right invariance, each right translation $T\mu^g$ extends to an isometric isomorphisms 
$\overline{T_xG} \to \overline{T_{xg}G}$. By the smooth uniform boundedness theorem these isomorphisms 
depend smoothly on $g\in G$.
\end{proof}

\subsection{Covariant curvature and O'Neill's formula}\nmb.{9.3}
In \cite[2.2]{Micheli2013} one finds the following 
formula for the numerator of sectional curvature,
which is valid for {\em closed smooth} 1-forms $\al,\be \in\Om^1_g(M)$ on a weak Riemannian
manifold $(M,g)$. Recall that we 
view $g:TM\to T^*M$ and so $g\i$ 
is the dual inner product on $g(TM)$ and $\al^\sharp = g\i(\al)$. 
$$
\begin{aligned}
&g\big(R(\al^{\sharp},\be^{\sharp})\al^{\sharp},\be^{\sharp}\big) =
\\&
-\tfrac12\al^{\sharp}\al^{\sharp}(\|\be\|_{g\i}^2)
-\tfrac12\be^{\sharp}\be^{\sharp}(\|\al\|_{g\i}^2)
+\tfrac12(\al^{\sharp}\be^{\sharp}+\be^{\sharp}\al^{\sharp})g\i(\al,\be)
\\&\qquad\qquad
\big(\text{last line }= -\al^\sharp \be([\al^\sharp,\be^\sharp])+\be^\sharp\al([\al^\sharp,\be^\sharp]])\big)
\\&
-\tfrac14\|d(g\i(\al,\be))\|_{g\i}^2
+\tfrac14 g\i\big(d(\|\al\|_{g\i}^2),d(\|\be\|_{g\i}^2)\big)
\\&
+\tfrac34 \big\|[\al^{\sharp},\be^{\sharp}]\big\|_{g}^2
\end{aligned}
$$
This is called Mario's formula since Mario Micheli derived the coordinate version in his 2008 
thesis. Each term depends only on $g\i$ with the exception of the last term. The role of the last 
term (which we call the O'Neill term) will become clear  in the next result.
Let $p:(E,g_E)\to (B,g_B)$ be a Riemannian submersion between infinite dimensional 
robust Riemannian manifolds; i.e., for each $b\in B$ and $x\in E_b:=p\i(b)$ 
the tangent mapping $T_xp:(T_xE, g_E)\to (T_bB,g_B)$ 
is a surjective metric quotient map so that 
\begin{equation*}
\|\xi_b\|_{g_B} := \inf\bigl\{\|X_x\|_{g_E}\,:\,X_x\in T_xE, T_xp.X_x=\xi_b\bigr\}.
\end{equation*}
The infinimum need not be attained in $T_xE$ but will be in the completion 
$\overline{T_xE}$. 
The orthogonal subspace $\{Y_x: g_E(Y_x,T_x(E_b))=0\}$ 
will therefore be taken in $\overline{T_x(E_b)}$ in $T_xE$. 
If $\al_b=g_B(\al_b^\sharp,\quad)\in g_B(T_bB)\subset T_b^*B$ 
is an element in the $g_B$-smooth dual, then 
$p^*\al_b:=(T_xp)^*(\al_b)= g_B(\al_b^\sharp,T_xp\quad):T_xE\to \mathbb R$ 
is in $T_x^*E$ but in general it is not an element in the smooth dual $g_E(T_xE)$. 
It is, however, an element of the Hilbert space completion $\overline{g_E(T_xE)}$ 
of the $g_E$-smooth dual $g_E(T_xE)$ with respect to the norm $\|\quad\|_{g_E\i}$, 
and the element $g_E\i(p^*\al_b)=: (p^*\al_b)^\sharp$ is in the 
$\|\quad\|_{g_E}$-completion $\overline{T_xE}$ of $T_xE$. 
We can call $g_E\i(p^*\al_b)=: (p^*\al_b)^\sharp$ the \emph{horizontal lift} of 
$\al_b^\sharp = g_B\i(\al_b)\in T_bB$.

\begin{theorem}\nmb.{9.4}
{\rm \cite[2.6]{Micheli2013}}
Let $p:(E,g_E)\to (B,g_B)$ be a Riemannian submersion between infinite
dimensional robust Riemannian manifolds. Then for {\bf closed} 1-forms
$\al,\be\in\Om_{g_B}^1(B)$ O'Neill's formula holds in the form:
\begin{align*}
 g_B\big(R^B(\al^{\sharp},\be^{\sharp})\be^{\sharp},\al^{\sharp}\big)
&= g_E\big(R^E((p^*\al)^{\sharp},(p^*\be)^{\sharp})(p^*\be)^{\sharp},(p^*\al)^{\sharp}\big)
\\&\quad
+\tfrac34\|[(p^*\al)^{\sharp},(p^*\be)^{\sharp}]^{\text{ver}}\|_{g_E}^2
\end{align*}
\end{theorem}

\begin{proof} The last (O'Neill) term is the difference between curvature on $E$ and the pullback 
of the curvature on $B$.
\end{proof}

\subsection{Semilocal version of Mario's formula, force, and stress}\nmb.{9.5}
In all interesting examples of orbits of diffeomorphisms groups through a template shape, 
Mario's covariant curvature formula leads to complicated and impenetrable formulas. 
Efforts to break this down to comprehensible pieces led to the concepts of symmetrized force and 
(shape-) stress explained below. Since acceleration sits in the second tangent bundle, one either 
needs a covariant derivative to map it down to the tangent bundle, or at least rudiments of  local 
charts. In \cite{Micheli2013} we managed the local version. Interpretations in mechanics or 
elasticity theory are still lacking. 

Let $(M,g)$ be a robust Riemannian manifold, $x\in M$, $\al,\be\in g_x(T_xM)$.
Assume we are given local smooth vector fields $X_\al$ and $X_\be$ such that:
\begin{enumerate}
\item $X_\al(x) = \al^\sharp(x),\quad X_\be(x) = \be^\sharp(x)$,
\item Then $\al^\sharp-X_\al$ is zero at $x$. Therefore it has a well-defined derivative 
       $D_x(\al^\sharp-X_\al)$ lying in Hom$(T_xM,T_xM)$. For a vector field $Y$ we have 
       $D_x(\al^\sharp-X_\al).Y_x = [Y,\al^\sharp-X_\al](x) = \L_Y(\al^\sharp-X_\al)|_x$.
       The same holds for $\be$.
\item $\L_{X_\al}(\al)=\L_{X_\al}(\be)=\L_{X_\be}(\al)=\L_{X_\be}(\be)=0$,
\item $[X_\al, X_\be] = 0$.
\end{enumerate}
Locally constant 1-forms and vector fields will do. We then define:
\begin{align*}
\mathcal F(\al,\be) :&= \tfrac12 d(g\i(\al,\be)), \qquad \text{a 1-form on $M$ called the \emph{force},}\\
\mathcal D(\al,\be)(x) :&= D_x(\be^\sharp - X_\be).\al^\sharp(x)
\\&
= d(\be^\sharp - X_\be).\al^\sharp(x), \quad\in T_xM\text{ called the \emph{stress}.}
\\
\implies &\mathcal D(\al,\be)(x) - \mathcal D(\be,\al)(x) = [\al^\sharp,\be^\sharp](x)
\end{align*}  
Then in terms of force and stress the numerator of sectional curvature looks as follows:
\begin{align*}
&g\big(R(\al^{\sharp},\be^{\sharp})\be^{\sharp},\al^{\sharp}\big)(x) = R_{11} +R_{12} + R_2 + 
R_3\,, 
\qquad\text{  where }
\\ 
& R_{11} = \tfrac12 \big(
\L_{X_\al}^2(g\i)(\be,\be)-2\L_{X_\al}\L_{X_\be}(g\i)(\al,\be)
+\L_{X_\be}^2(g\i)(\al,\al)
 \big)(x)\,, 
\\
& R_{12} = \langle \mathcal F(\al,\al), \mathcal D(\be,\be) \rangle + \langle \mathcal F(\be,\be),\mathcal D(\al,\al)\rangle  
- \langle \mathcal F(\al,\be), \mathcal D(\al,\be)+\mathcal D(\be,\al) \rangle 
\\
& R_2 = \big(
\|\mathcal F(\al,\be)\|^2_{g\i}
-\big\langle \mathcal F(\al,\al)),\mathcal F(\be,\be)\big\rangle_{g\i} \big)(x)\,, 
\\
& R_3 = -\tfrac34 \| \mathcal D(\al,\be)-\mathcal D(\be,\al) \|^2_{g_x} \,.
\end{align*}

\subsection{Landmark space as homogeneous space of solitons}\nmb.{9.6}
This subsection is based on \cite{Micheli2012};
the method explained here has many applications in computational anatomy and elsewhere, under the 
name LDDMM (large diffeomorphic deformation metric matching).

A {\em landmark} $q=(q_1,\dots,q_N)$ is an $N$-tuple 
of distinct points in $\mathbb R^n$;    
landmark space $\on{Land}^N(\mathbb R^n)\subset (\mathbb R^n)^N$ is open.
Let $q^0=(q^0_1,\dots,q^0_N)$ be a fixed standard template 
landmark. Then we have the surjective mapping
\begin{align*}
&\on{ev}_{q^0}:\on{Diff}_{\mathcal A}(\mathbb R^n)\to \on{Land}^N(\mathbb R^n),
\\&
\ph\mapsto \on{ev}_{q^0}(\ph)=\ph(q^0)=(\ph(q^0_1),\dots,\ph(q^0_N)).
\end{align*}
Given a Sobolev metric of order $s>\frac{n}2 + 2$ on $\Diff_{\mathcal A}(\mathbb R^n)$,
we want to induce a Riemannian metric on $\on{Land}^N(\mathbb R^n)$ such that $\on{ev}_{q_0}$ becomes a 
Riemannian submersion. 

The fiber of $\on{ev}_{q^0}$ over a landmark $q=\ph_0(q^0)$ is  
\begin{align*}
\{\ph\in\on{Diff}_{\mathcal A}(\mathbb R^n): \ph(q^0)=q\}
&= \ph_0\o\{\ph\in\on{Diff}_{\mathcal A}(\mathbb R^n): \ph(q^0)=q^0\}
\\&
= \{\ph\in\on{Diff}_{\mathcal A}(\mathbb R^n): \ph(q)=q\}\o\ph_0\,.
\end{align*}
The tangent space to the fiber is 
$$
\{X\o \ph_0: X\in \X_{\mathcal S}(\mathbb R^n), X(q_i) = 0 \text{  for all } i \}.
$$
A tangent vector $Y\o\ph_0 \in T_{\ph_0}\on{Diff}_{\mathcal S}(\mathbb R^n)$ is 
$G_{\ph_0}^L$-perpendicular to the fiber over $q$ if and only if
$$
\int_{\mathbb R^n} \langle LY, X \rangle\,dx =0\quad 
\forall X\text{  with }X(q)=0.
$$
If we require $Y$ to be smooth then $Y=0$. So we assume that 
$LY=\sum_i P_i.\de_{q_i}$, a distributional vector field with support
in $q$. Here $P_i\in T_{q_i}\mathbb R^n$. 
But then  
\begin{align*}
Y(x) &= L\i\Big(\sum_i P_i.\de_{q_i}\Big) 
= \int_{\mathbb R^n} K(x-y)\sum_i P_i.\de_{q_i}(y)\,dy
= \sum_i K(x-q_i).P_i,
\\
&T_{\ph_0}(\on{ev}_{q^0}).(Y\o\ph_0) = Y(q_k)_k = \sum_i (K(q_k-q_i).P_i)_k\,.
\end{align*}
Now let us consider a tangent vector $P=(P_k)\in T_q\on{Land}^N(\mathbb R^n)$. Its
horizontal lift with footpoint $\ph_0$ is $P^{\text{hor}}\o\ph_0$
where the vector field $P^{\text{hor}}$ on $\mathbb R^n$ is given as follows:
Let $K\i(q)_{ki}$ be the inverse of the $(N\x N)$-matrix $K(q)_{ij}=K(q_i-q_j)$.
Then  
\begin{align*}
P^{\text{hor}}(x) &= \sum_{i,j} K(x-q_i)K\i(q)_{ij}P_j\,,
\\
L(P^{\text{hor}}(x))&= \sum_{i,j} \de(x-q_i)K\i(q)_{ij}P_j\,.
\end{align*}
Note that $P^{\text{hor}}$ is a vector field of class $H^{2l-1}$.

The Riemannian metric on the finite dimensional manifold $\on{Land}^N$ induced by the $g^L$-metric on
$\on{Diff}_{\mathcal S}(\mathbb R^n)$ is given by
\begin{align*}
g^L_q(P,Q) &= G^L_{\ph_0}(P^{\text{hor}}, Q^{\text{hor}})
= \int_{\mathbb R^n}\langle L(P^{\text{hor}}),Q^{\text{hor}} \rangle\,dx
\\&
= \int_{\mathbb R^n}\Big\langle \sum_{i,j}\de(x-q_i)K\i(q)_{ij}P_j,
 \sum_{k,l} K(x-q_k)K\i(q)_{kl}Q_l \Big\rangle\,dx
\\&
= \sum_{i,j,k,l} K\i(q)_{ij}K(q_i-q_k)K\i(q)_{kl}\langle P_j,Q_l \rangle
\\
g^L_q(P,Q) &= \sum_{k,l} K\i(q)_{kl}\langle P_k,Q_l \rangle.
\end{align*}
The {\em geodesic equation} in vector form is:  
\begin{align*}
\ddot q_n =
& -\frac12
\sum_{k,i,j,l}K\i(q)_{ki}
\on{grad} K(q_i-q_j)(K(q)_{in}-K(q)_{jn}) 
K\i(q)_{jl}\langle \dot q_k,\dot q_l \rangle
\\&
+ \sum_{k,i}K\i(q)_{ki}
\Big\langle\on{grad} K(q_i-q_n),\dot q_i-\dot q_n\Big\rangle 
\dot q_k\,.
\end{align*}
The cotangent bundle  $T^*{\on{Land}^N(\mathbb R^n)}=  
\on{Land}^N(\mathbb R^n)\x ((\mathbb R^n)^N)^*\ni (q,\al)$. 
We treat $\mathbb R^n$ like scalars;
$\langle\quad,\quad  \rangle$ is always the standard inner product on $\mathbb R^n$.
\\
The inverse metric is then given by  
\begin{align*}
(g^L)\i_q(\al,\be) &= \sum_{i,j} K(q)_{ij}\langle \al_i,\be_j \rangle,
\qquad
K(q)_{ij} = K(q_i-q_j).
\end{align*}
The energy function is
\begin{align*}
E(q,\al)&=\tfrac12 (g^L)\i_q(\al,\al) 
= \tfrac12\sum_{i,j} K(q)_{ij}\langle \al_i,\al_j \rangle
\end{align*}
and its Hamiltonian vector field (using $\mathbb R^n$-valued derivatives to save notation)  is
\begin{align*}
H_E(q,\al) &=
  \sum_{i,k=1}^N \Big(K(q_k-q_i)\al_i\frac{\p}{\p q_k} 
+ \on{grad} K(q_i-q_k)\langle \al_i,\al_k \rangle\frac{\p}{\p \al_k}\Big). 
\end{align*}
So the {\em Hamiltonian version of the geodesic equation} is the flow of this vector field:
$$
\begin{cases}
\dot q_k &= \sum_i K(q_i-q_k)\al_i
\\
\dot \al_k &= -\sum_i  \on{grad} K(q_i-q_k) \langle \al_i,\al_k \rangle
\end{cases}
$$
We shall use {\em stress and force} to express the geodesic equation and curvature:
\begin{align*}
\al^\sharp_k =\sum_i &K(q_k-q_i)\al_i,
\quad \al^\sharp = \sum_{i,k} K(q_k-q_i)\langle \al_i,\tfrac{\p}{\p q^k} \rangle
\\
\mathcal D(\al,\be) :&
= \sum_{i,j} 
dK(q_i-q_j)(\al^\sharp_i-\al^\sharp_j)\Big\langle \be_j,\frac{\partial}{\partial q_i}\Big\rangle, 
\quad\text{the \emph{stress.}}
\\
\mathcal D(\al,\be) &- \mathcal D(\be,\al) = (D_{\al^\sharp}\be^\sharp) - 
D_{\be^\sharp}\al^\sharp =[\al^\sharp,\be^\sharp],
\quad\text{\emph{Lie bracket.}}
\\
\mathcal F_i(\al,\be) &= \frac12\sum_k \on{grad} K(q_i-q_k)(\langle \al_i,\be_k \rangle + \langle \be_i,\al_k \rangle )
\\
\mathcal F(\al,\be) :&= \sum_i \langle \mathcal F_i(\al,\be),dq_i\rangle = \frac12\, d\,g\i(\al,\be) 
\quad\text{the \emph{force}.} 
\end{align*}
The geodesic equation on $T^*\on{Land}^N(\mathbb R^n)$ then becomes 
$$
\begin{cases}
\dot q &= \al^\sharp 
\\
\dot \al &= - \mathcal F(\al,\al) \,.
\end{cases}
$$
Next we shall compute {\em curvature via the cotangent bundle}.
From the semilocal version of Mario's formula for the numerator of the sectional curvature
for constant 1-forms $\al,\be$ on landmark space, where 
$\al^\sharp_k =\sum_i K(q_k-q_i)\al_i$, we get directly:
\begin{align*}
&g^L\big(R(\al^{\sharp},\be^{\sharp})\al^{\sharp},\be^{\sharp}\big) =
\\&=
\big\langle \mathcal D(\al,\be) + \mathcal D(\be,\al),\mathcal F(\al,\be)\big\rangle 
\\&\quad
-\big\langle \mathcal D(\al,\al),\mathcal F(\be,\be)\big\rangle 
-\big\langle \mathcal D(\be,\be),\mathcal F(\al,\al)\big\rangle
\\&\quad
-\tfrac12\sum_{i,j}\Big(
d^2K(q_i-q_j)(\be^\sharp_i-\be^\sharp_j,\be^\sharp_i-\be^\sharp_j)\langle \al_i,\al_j\rangle
\\&\qquad\qquad
-2d^2K(q_i-q_j)(\be^\sharp_i-\be^\sharp_j,\al^\sharp_i-\al^\sharp_j)\langle \be_i,\al_j\rangle
\\&\qquad\qquad
+d^2K(q_i-q_j)(\al^\sharp_i-\al^\sharp_j,\al^\sharp_i-\al^\sharp_j)\langle \be_i,\be_j\rangle
\Big)
\\&\quad
-\|\mathcal F(\al,\be)\|_{g\i}^2 + g\i\bigl(\mathcal F(\al,\al),\mathcal F(\be,\be)\bigr). 
\\&\quad
+\tfrac34 \| [\al^\sharp,\be^\sharp] \|_g^2
\end{align*}

\subsection{Shape spaces of submanifolds as homogeneous spaces for the diffeomorphism group}\nmb.{9.7}
Let $M$ be a compact manifold and $(N,\bar g)$ a Riemannian manifold of bounded geometry as in 
subsection \nmb!{3.6}.
The diffeomorphism group $\Diff_{\mathcal A}(N)$ acts also from the left on the manifold of 
$\Emb(M,N)$ embeddings and also on the {\em nonlinear Grassmannian} or 
{\em differentiable Chow variety} $B(M,N)=\Emb(M,N)/\Diff(M)$. 
For a Sobolev metric of order $s>\frac{\dim(N)}2+ 2$ one can then again induce a Riemannian metric 
on each $\Diff_{\mathcal A}(N)$-orbit, as we did above for landmark spaces. 
This is done in \cite{Micheli2013}, where the geodesic equation is computed and where curvature is 
described in terms of stress and force. 

\newpage
\bibliographystyle{abbrv}
\bibliography{../../ref/ref,../../ref/biblio,../../ref/preprints,../../ref/articles}

\end{document}